\documentclass[sn-mathphys]{sn-jnl}

\jyear{2022}

\theoremstyle{thmstyleone}
\newtheorem{theorem}{Theorem}
\newtheorem{proposition}[theorem]{Proposition}
\newtheorem{corollary}[theorem]{Corollary}

\theoremstyle{thmstyletwo}
\newtheorem{example}{Example}
\newtheorem{remark}{Remark}

\theoremstyle{thmstylethree}
\newtheorem{definition}{Definition}
\newtheorem{lemma}[theorem]{Lemma}

\newcommand{\Nmap}{\mathcal{N}_\mathrm{m}}

\graphicspath{{./figs/}}

\definecolor{darkblue}{rgb}{0.1, 0.1, 0.8}

\newcommand{\N}{\mathbb{N}}

\newcommand{\mc}{\mathcal}

\newcommand{\etal} {{{et al.}}}
\newcommand{\ie} {{{i.e.}}}
\newcommand{\eg} {{{e.g.}}}

\newcommand{\para}[1]{\vspace{1mm}\noindent{\textbf{#1}}}

\newcommand {\mm}[1] {\ifmmode{#1}\else{\mbox{\(#1\)}}\fi}

\newcommand{\denselist}{\vspace{-2pt} \itemsep -1pt\parsep=-1pt\partopsep -1pt}

\newcommand{\Rspace}        {\mm{\mathbb{R}}}

\newcommand{\Hcal}        {\mm{\mathcal{H}}}

\newcommand{\Pcal}        {\mm{\mathcal{P}}}

\usepackage{hyperref}
\usepackage{amsmath}
\usepackage{amssymb}
\usepackage{url}
\usepackage{mathrsfs}
\usepackage{blkarray}
\usepackage{algorithm}
\usepackage{algpseudocode}
\usepackage{bbm}
\usepackage{amscd}

\begin{document}

\title[Hypergraph Co-Optimal Transport]{Hypergraph Co-Optimal Transport: Metric and Categorical Properties}

\author[1]{\fnm{Samir} \sur{Chowdhury}}\email{samirc@stanford.edu}
	
\author*[2]{\fnm{Tom} \sur{Needham}}\email{tneedham@fsu.edu}

\author[2]{\fnm{Ethan} \sur{Semrad}}\email{esemrad@fsu.edu}

\author[3]{\fnm{Bei} \sur{Wang}}
     \email{beiwang@sci.utah.edu}
     
\author[3]{\fnm{Youjia} \sur{Zhou}}
     \email{zhou325@sci.utah.edu}     

\affil[1]{\orgname{Stanford University}, 
\orgaddress{\city{Stanford}, \postcode{94304}, \state{CA}, \country{USA}}}

\affil[2]{\orgname{Florida State University}, \orgaddress{\city{Tallahassee}, \postcode{32306}, \state{FL}, \country{USA}}}

\affil[3]{\orgname{University of Utah}, \orgaddress{\city{Salt Lake City}, \postcode{610101}, \state{UT}, \country{USA}}}


\abstract{
Hypergraphs capture multi-way relationships in data, and they have consequently seen a number of applications in higher-order network analysis, computer vision, geometry processing, and machine learning.  
In this paper, we develop theoretical foundations for studying the space of hypergraphs using ingredients from optimal transport. By enriching a hypergraph with probability measures on its nodes and hyperedges, as well as relational information capturing local and global structures, we obtain a general and robust framework for studying the collection of all hypergraphs.  
First, we introduce a hypergraph distance based on the co-optimal transport framework of Redko {\etal} and study its theoretical properties. 
Second, we formalize common methods for transforming a hypergraph into a graph as maps between the space of hypergraphs and the space of graphs, and study their functorial properties and Lipschitz bounds. 
Finally, we demonstrate the versatility of our Hypergraph Co-Optimal Transport (HyperCOT) framework through various examples.  

}


\maketitle


\section{Introduction}
\label{sec:introduction}

The study of hypergraphs is motivated by higher-order network analysis. 
Graphs are commonly used to encode data from complex systems in cybersecurity, biology, sociology, and physical infrastructure, where nodes represent entities and edges represent \emph{pairwise} relations between entities. 
However, real-world data contain abundant \emph{multi-way} relationships. For example, Patania \etal~\cite{PataniaPetriVaccarino2017} studied coauthorship of scientific papers from arXiv, where a paper represents a multi-way relationship between several authors; Cencetti \etal~\cite{CencettiBattistonLepri2021} found that higher-order interactions are ubiquitous in human face-to-face communications characterized by heterogeneous dynamics.

These multi-way relationships are better captured by hypergraphs. 
A \emph{hypergraph} consists of a set of \emph{nodes} and a set of subsets of the node set, whose elements are called \emph{hyperedges}. 
A hypergraph generalizes the notion of a graph, which is simply a hypergraph whose hyperedges each contain exactly two nodes. For example, a hypergraph can be used to represent molecular structure, where nodes are individual atoms, hyperedges represent either covalent bonds between pairs of atoms or polycentric bonds between multiple atoms (see~\cite[Chapter 7.1]{Bretto2013} and~\cite{KonstantinovaSkorobogatov235}). 
Hypergraphs can also model cellular networks~\cite{KlamtHausTheis2009}, where protein complexes are hyperedges consisting of multiple proteins.
Lotito {\etal} demonstrated the informative power of higher-order motifs among a number of hypergraph datasets~\cite{LotitoMusciottoMontresor2022} in sociology, biology, and technology.

Hypergraphs also arise from applications in computer vision, geometry processing, pattern recognition, and machine learning, where establishing correspondences between two feature sets is a fundamental issue. 
The correspondence problem is traditionally formulated as a \emph{graph matching} problem~\cite{Umeyama1988, GoldRangarajan1996}: each graph is constructed with nodes representing features and edges representing \emph{pairwise} relationships between features.  
For example, in geometry processing, absolute locations of features are not as relevant as their pairwise relations, \eg, when imposing invariance over translations and rotations.
However, pairwise relations are not sufficient to incorporate the higher-order relations among features in establishing correspondences -- this can be addressed by a \emph{hypergraph matching} problem~\cite{ChertokKeller2010,LeeChoLee2011}. 

 Recent years have seen the extension of the Gromov-Wasserstein (GW) framework -- originally developed as a tool for comparing metric measure spaces \cite{Memoli2007,dghlp-focm} -- to probabilistic graph matching tasks \cite{solomon2016entropic,hendrikson,xu2019gromov,xu2019scalable,ChowdhuryMemoli2019,vayer2019optimal,chowdhury2021quantized}. Details on this framework are provided in \autoref{sec:networks}, but the rough idea is as follows. Optimal transport metrics are typically used to compare probability distributions on a common metric space, and are widely perceived as a powerful framework for quantifying uncertainty within geometric measurements \cite{peyre2019computational,solomon2018optimal}. GW distances, in contrast, are used to obtain correspondences across different spaces, which in turn permits comparison between spaces that are not comparable a priori \cite{Memoli2007,dghlp-focm}. GW distances are based on finding probabilistic alignments between points in two metric spaces which minimize an overall metric distortion objective. The numerous benefits of this approach include computability via gradient descent \cite{pcs16,pot} or backpropagation \cite{xu2020gromov}, state-of-the-art performance in tasks such as graph partitioning \cite{xu2019scalable,chowdhury2021generalized}, and an underlying Riemannian theoretical  framework \cite{sturm2012space,chowdhury2020gromov}. 
 These successes motivate the development of a GW framework for hypergraphs, which is the goal of this paper. Our contributions include: 
\begin{itemize}
\item We extend the co-optimal transport framework of Redko \etal~\cite{TitouanRedkoFlamary2020} to define an optimal transport-based distance between hypergraphs. To put this on firm theoretical footing, we introduce a  generalized notion of a hypergraph, called a \emph{measure hypernetwork}, see~\autoref{sec:hypernetworks}. We then establish fundamental properties of our hypergraph distance; in particular, we show that it is a pseudometric on the space of measure hypernetworks which induces a complete, geodesic metric on the space of measure networks modulo the distance-0 equivalence relation (\autoref{thm:pseudometric}).

\item We introduce categorical structures on the space of measure (hyper)networks. Through this lens, we study common transformations taking hypergraphs to graphs (e.g., incidence graphs, clique expansions, and line graphs) as functors from the category of hypernetworks to the category of networks. We also study their Lipschitz properties with respect to our hypernetwork distance and the GW distance on the space of networks. We show that the incidence graph map is a functorial isometry of our distance to a novel variant of the GW distance (\autoref{thm:bipartite_equivalence}). We then introduce 1-parameter families of clique expansion and line graph maps and prove that they are Lipschitz functors (\autoref{thm:p_clique}). 

\item We obtain some structural results on our categories of measure (hyper)networks. Proposition \ref{prop:isomorphism_equivalences} shows the equivalence of certain metric notions of isomorphism from the literature and  isomorphisms in the category-theoretic sense. The (non)existence of various limit constructions in our categories is established in Propositions \ref{prop:limits_maps} and \ref{prop:no_limits}.

\item We illustrate our open source computational framework\footnote{Hypergraph Co-Optimal Transport: \url{https://github.com/samirchowdhury/HyperCOT}} for hypergraph matching and comparison across multiple examples in \autoref{sec:examples}. Our method is incorporated into a framework which builds on a topology-based algorithm for hypergraph simplification \cite{ZhouRathorePurvine2022}; in particular, our hypergraph distance is used to highlight simplification levels of interest for real-world datasets.
\end{itemize} 

\subsection{Related Work}
\label{sec:related-work}

\para{Hypergraph metrics, similarity and dissimilarity measures.}
Although there is an abundant amount of data to be modeled as hypergraphs, metrics between hypergraphs have not been explored extensively in the literature.  
Karonski and Palka~\cite{KaronskiPalka1977} considered hypergraphs defined over the same set of nodes as clusterings and defined the Marczewski–Steinhaus (MS) distance between the clusterings. 
The MS distance is a modification from a generalization of the Hausdorff metric for sets~\cite{Ulam1972}. 
Karonski and Palka also discussed a distance between hypergraphs generated by arborescences with the same set of terminal vertices. 
In comparison with the MS distance, our hypergraph distance does not require hypergraphs to be defined on the same set of nodes; and more importantly, our distance comes with nice metric properties, encodes explicitly structural relations between nodes and hyperedges, and naturally gives rise to meaningful hypergraph matching. Lee \etal~\cite{LeeChoLee2011} generalized the edit distance from graphs to hypergraphs. 
However, such a distance is impractical as the graph edit distance is NP-hard to compute~\cite{ZengTungWang2009} and APX-hard to approximate~\cite{Lin1994}.
Smaniotto and Pelillo~\cite{SmaniottoPelillo2021} recently defined distances between attributed hypergraphs using their maximal common subhypergraph; however, computing the maximum similarity subhypergraph isomorphism between the two hypergraphs is NP-complete.  
In comparison, our hypergraph distance may be approximated efficiently by following the optimization routine of co-optimal transport. Another approach to hypergraph comparison is taken in~\cite{SuranaChenRajapakse2021}, which transforms a hypergraph into a graph and applies standard graph similarity or dissimilarity measures. 
Hypergraphs can also be encoded as tensors~\cite{BanerjeeCharMondal2017}, so that spectral comparisons between hypergraphs can be made through their tensor representations. 
The work in \cite{duchenne2011tensor} describes a principled framework for hypergraph matching and proposed to use this framework for going beyond isometry-invariant shape matching and obtaining invariance under similarity, affine, and projective transformations. However, the framework in~\cite{duchenne2011tensor} constructs hypergraphs with a fixed size for each hyperedge, and is not suited for cases where the input data is an arbitrary hypergraph. In contrast, our hypergraph distance can accept any pair of arbitrary hypergraphs as input. 

\para{Optimal transport and Gromov-Wasserstein distances.}
 Gromov-Wasserstein (GW) distances, as described above, essentially require only square matrices encoding pairwise relational information in the spaces to be compared \cite{ChowdhuryMemoli2019}, and are thus well-suited for comparing graphs via adjacency, shortest path distance, or spectral representations \cite{xu2019gromov,vayer2019optimal,chowdhury2021generalized}. However, hypergraphs typically encode more than pairwise relations; our strategy for extending the GW framework to hypergraphs is to work directly with multi-way relations (modeled as rectangular matrices encoding node-hyperedge relations in finite settings) and to leverage the recently-developed co-optimal transport framework of \cite{TitouanRedkoFlamary2020}. Using this framework, we are able to simultaneously infer node-node and hyperedge-hyperedge correspondences when comparing hypergraphs.

\section{Hypernetworks: Theoretical Formulations}
\label{sec:theory}

This section introduces the basic concepts of measured hypernetworks and distances between them.

\subsection{Background: Measure Networks}
\label{sec:networks}

A \emph{graph} is a pair $(V,E)$, where $V$ is a set of \emph{nodes} and $E$ is a set of 2-element subsets of $V$, each of which is called an \emph{edge}. The following gives a far-reaching generalization of the notion of a graph.

\begin{definition}[\cite{ChowdhuryMemoli2019}]
\label{def:measure-network}
A \emph{measure network} is a triple $N=(X,\mu,\omega)$, where $X$ is a Polish space, $\mu$ is a Borel probability measure on $X$ and $\omega:X \times X \to \Rspace$ is a  measurable,  bounded \emph{network function}. We take the simplifying conventions that $\mu$ is fully supported and $\omega$ takes non-negative values.
Let $\mathcal{N}$ denote the collection of all measure networks.
\end{definition}

Recall that the support of a Borel probability measure $\mu$ on a Polish space $X$ is the set $\mathrm{supp}(\mu)$ of $x \in X$ such that $\mu(U) > 0$ for every open neighborhood of $x$ (this is equivalent to \cite[Definition 2.2]{dghlp-focm}). Then a measure network is assumed to have $\mathrm{supp}(\mu) = X$. 

A graph $(V,E)$ with a probability distribution $\mu$ on $V$ defines a measure network by, for example, taking $\omega$ to be a binary adjacency function. However, this definition encompasses a huge variety of other structures; any metric measure space (metric space endowed with a Borel probability measure) is, in particular, a measure network. As was also observed in, e.g.,  \cite{sturm2012space,pcs16}, measure networks give a general setting for defining GW distances (which were originally defined only for metric measure spaces \cite{Memoli2007}). 

In the following, we use tools from measure theory. Given a measurable map $f:X \to Y$ from a Borel measure space $(X,\mu)$ to a Borel space $Y$, the \emph{pushforward of $\mu$} is the measure $f_\# \mu$ on $Y$ defined by $f_\# \mu (U) = \mu(f^{-1}(U))$ for all Borel sets $U \subset Y$. For a measurable function $f:X \to \Rspace$ and $p \in [1,\infty)$, the $L^p$ norm of $f$ is
\[
\|f\|_{L^p(\mu)} = \left(\int_X \lvert f(x) \rvert^p \; \mu(dx)\right)^{1/p},
\]
provided this integral converges. For $p = \infty$, $\|\cdot\|_{L^\infty(\mu)}$ is the \emph{essential supremum},
\[
\| f\|_{L^\infty(\mu)} = \inf\{c \geq 0 \mid \lvert f(x) \rvert \leq c \mbox{ for $\mu$-almost every $x \in X$}\},
\]
provided this is finite. Given two Borel measure spaces $(X,\mu)$ and $(Y,\nu)$, the \emph{product measure} $\mu \otimes \nu$ on $X \times Y$ satisfies $\mu \otimes \nu (U \times V) = \mu(U)\cdot \nu(V)$ for all Borel $U \subset X$ and $V \subset Y$.

\begin{definition}
\label{def:coupling}
Given two probability spaces $(X,\mu)$ and $(Y,\nu)$, a \emph{coupling} $\pi$ is a probability measure on $X\times Y$ satisfying $(p_X)_\# \pi = \mu$ and $(p_Y)_\# \pi = \nu$, where $p_X:X \times Y \to X$ and $p_Y:X \times Y \to Y$ are coordinate projections. That is, the marginals of $\pi$ are $\mu$ and $\nu$, respectively. The collection of couplings between $\mu,\nu$ is denoted $\mathcal{C}(\mu,\nu)$.
\end{definition}

\begin{definition}[Gromov-Wasserstein Distance]
For measure networks $N = (X,\mu,\omega)$ and $N' = (X',\mu',\omega')$ and $p \in [1,\infty]$, the \emph{$p$-distortion} functional is
$
\mathrm{dis}^{\mathcal{N}}_p = \mathrm{dis}^{\mathcal{N}}_{N,N',p}: \mathcal{C}(\mu,\mu') \to \Rspace
$
defined by
\begin{equation}\label{eqn:GW_loss}
\mathrm{dis}^{\mathcal{N}}_p(\pi) = \| \omega - \omega' \|_{L^p(\pi \otimes \pi)},
\end{equation}
where $\omega - \omega'$ is understood as the function $X \times X' \times X \times X' \to \Rspace$ defined by $(x,x',y,y') \mapsto \omega(x,y) - \omega'(x',y')$.
Equivalently, for $p < \infty$, \eqref{eqn:GW_loss} may be expanded as, 
\[
\mathrm{dis}^{\mathcal{N}}_p(\pi) = \left(\int_{X \times X'} \int_{X \times X'} \mid\omega(x,y) - \omega'(x',y')\mid^p 
\pi(dx \times dx') \pi(dy \times dy')\right)^{1/p}.
\]

The \emph{Gromov-Wasserstein (GW) $p$-distance} is then defined to be
\begin{equation}
\label{eqn:GW_distance}
    d_{\mathcal{N},p}(N,N') := \inf_{\pi \in \mathcal{C}(\mu,\mu')} \mathrm{dis}^{\mathcal{N}}_p(\pi).
\end{equation}
\end{definition}

An \emph{optimal coupling}---i.e. a coupling $\pi$ realizing \eqref{eqn:GW_distance}---always exists \cite[Theorem 2.2]{ChowdhuryMemoli2019}. Intuitively, an optimal coupling  provides a \emph{soft matching} between the points in the measure networks which minimally distorts their network functions.

GW $p$-distance is known to define a pseudometric on $\mathcal{N}$ \cite[Theorem 2.3]{ChowdhuryMemoli2019} and the distance-0 equivalence classes are well-understood. We define measure networks $N, N'$ to be \emph{weakly isomorphic} if $d_{\mathcal{N},p}(N,N') = 0$. It follows from \cite[Theorem 2.4]{ChowdhuryMemoli2019} that $N = (X,\mu,\omega),N' = (X',\mu',\omega')$ are weakly isomorphic if and only if there exists a measure network $\overline{N} = (\overline{X},\overline{\mu},\overline{\omega})$ and measure-preserving maps $\phi:\overline{X} \to X$ and $\phi': \overline{X} \to X'$ such that $\omega(\phi(x),\phi(y)) = \omega'(\phi'(x),\phi'(y)) = \overline{\omega}(x,y)$ for $\overline{\mu} \otimes \overline{\mu}$-almost every $(x,y) \in \overline{X} \times \overline{X}$.

\begin{remark}
The notion of weak isomorphism is introduced because it is a relaxation of the more obvious notion of a strong isomorphism. Following \cite{ChowdhuryMemoli2019}, a \emph{strong isomorphism} from $N$ to $N'$ is defined to be a bijective measure-preserving map $\phi:X \to X'$ with measure-preserving inverse such that for all $x,y \in X$, $\omega(x,y) = \omega'(\phi(x),\phi(y))$. If a strong isomorphism exists, then $N$ and $N'$ are weakly isomorphic, but the converse to this statement does not hold; see \cite[Example 2.1]{ChowdhuryMemoli2019}.
\end{remark}

\subsection{Measure Hypernetworks and Hypernetwork Distance}
\label{sec:hypernetworks}

A \emph{hypergraph} is a pair $(X,Y)$, where $X$ is a set of \emph{nodes} and $Y$ is a collection of subsets of $X$, each of which is called a \emph{hyperedge}. \autoref{fig:hypergraph-a-b} shows some simple hypergraphs and different hypergraph visualization techniques.

\begin{figure}[!ht]
	\centering
	\includegraphics[width=0.9\columnwidth]{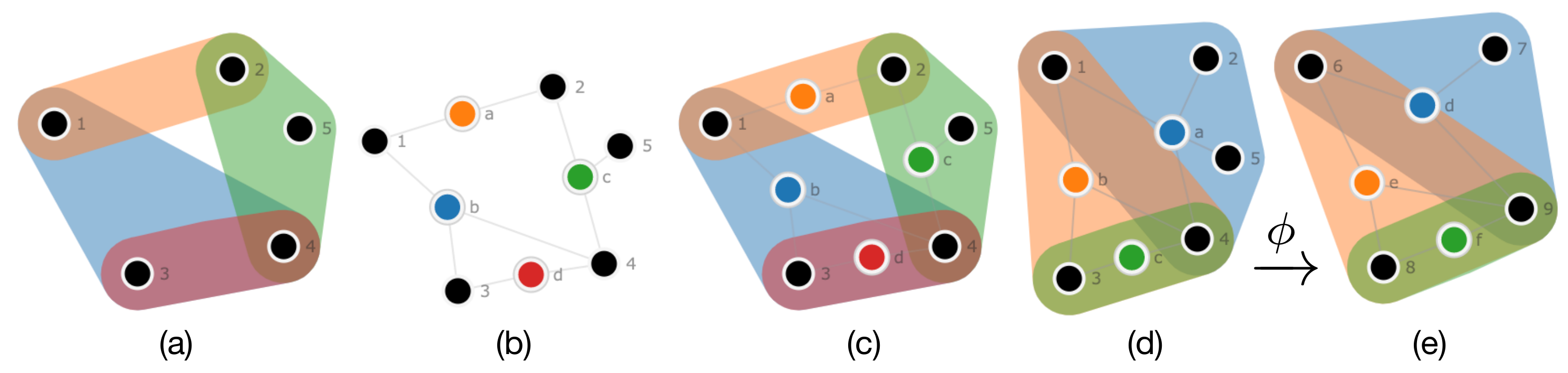}
	\caption{(a) A hypergraph visualized as a Venn diagram; nodes are black, and hyperedges are colored convex hulls. (b) Replacing each hyperedge in the Venn diagram with its own node representation, we visualize the same hypergraph as an \emph{incidence graph} (see~\autoref{sec:graphification} for theoretical properties of this representation). (c) Hybrid Venn diagram and incidence graph visualization. (d)-(e) An example of a \emph{morphism} between hypergraphs---categorical structures on the space of hypergraphs will be studied in detail in~\autoref{sec:graphification}.}
	\label{fig:hypergraph-a-b}
\end{figure}

The notion of a hypergraph can be formalized~\cite{DorflerWaller1980} as a triple $(X,Y,\iota)$, where $X$ and $Y$ are sets and $\iota:Y \to \mathcal{P}(X)$ is a map from $Y$ to the power set of $X$, assigning a subset of $X$ to each point in $Y$. One could equivalently consider a triple $(X,Y,\omega)$, where $\omega:X \times Y \to \{0,1\}$ is an indicator function encoding containment of a point $x \in X$ in $y \in Y$ by $\omega(x,y)=1$. Inspired by the definition of a measure network given above, we now give a generalization of the notion of a hypergraph, which introduces probability measures on $X$ and $Y$ and allows for more general functions $\omega$. 

\begin{definition}
\label{def:hyprenetwork}
A \emph{measure hypernetwork} is a quintuple 
$H = (X,\mu,Y,\nu,\omega)$, 
where $(X,\mu)$ and $(Y,\nu)$ are Polish spaces endowed with Borel probability measures and 
$\omega: X \times Y \to \Rspace$
is a measurable, bounded \emph{hypernetwork function}. We take the simplifying conventions that $\mu$ and $\nu$ are fully supported and that $\omega$ takes non-negative values. Let $\mathcal{H}$ denote the collection of all hypernetworks.
\end{definition}

\begin{example}[Hypergraphs]
\label{example:hypergraph-a}
Our motivating examples of hypernetworks arise from classical hypergraphs. To emphasize the distinction between the classical setting and our generalized measure hypernetworks, we refer to a pair $(X,Y)$ where $X$ is a finite set of nodes and $Y \subset \Pcal(X)$ is a set of hyperedges as a \emph{combinatorial hypergraph}. 

A combinatorial hypergraph $(X,Y)$ gives rise to a number of representations as a measure hypernetwork. For example, we may endow $X$ and $Y$ with  uniform measures $\mu$ and $\nu$, respectively; \ie, $\mu(x) = {1}/{\lvert X \rvert}$ for any $x \in X$ and $\mu(y) = 1/{\lvert Y \rvert}$ for any $y \in Y$.
We may define a hypernetwork function $\omega:X \times Y \to \Rspace$ by the \emph{incidence relation}:
\[
\omega(x,y) := \left\{\begin{array}{cl}
1 & \mbox{if $x \in y$} \\
0 & \mbox{otherwise.}
\end{array}\right.
\]
Alternatively, $\omega$ may be defined by a \emph{shortest path relation}: (a) if $x \in y$, then $\omega(x,y)=0$; (b) if $x \notin y$, then $\omega(x,y)$ is the length of the shortest \emph{hyperedge path}---a sequence of hyperedges with nontrivial sequential pairwise overlap, with length equal to the sum of the size of the overlaps---from any hyperedge $y'$ containing $x$ to hyperedge $y$. This method for transforming a hypergraph into a measure hypernetwork assumes that the hypergraph is \emph{connected}: any two hyperedges are joined by a hyperedge path. For the example hypergraph shown in~\autoref{fig:hypergraph-a-b}c, we have 
$X = \{1,2,3,4,5\}$, $Y = \{a, b, c, d\}$, $\mu(x) = 1/5$ for each $x \in X$, and $\nu(y) = 1/4$ for each $y \in Y$. 
If $\omega$ encodes the incidence relation or the shortest path relation, we have, respectively,   
\[
\omega = 
\begin{bmatrix}
1 & 1 & 0 & 0 \\
1 & 0 & 1 & 0 \\
0 & 1 & 0 & 1 \\
0 & 1 & 1 & 1 \\
0 & 0 & 1 & 0 \\
\end{bmatrix}, 
\text{ or }
\omega = 
\begin{bmatrix}
0 & 0 & 1 & 2 \\
0 & 1 & 0 & 1 \\
1 & 0 & 1 & 0 \\
1 & 0 & 0 & 0 \\
1 & 1 & 0 & 1 \\
\end{bmatrix},
\] 
where the matrices encode the function $\omega$ by indexing rows by $X$ and columns by $Y$.
\end{example}

\begin{remark}
The constructions in Example~\autoref{example:hypergraph-a} involve several choices: a measure on the nodes, a measure on the hyperedges and a hypernetwork function. There are many more ways to model a hypergraph as a measure hypernetwork; some are described in detail in~\autoref{sec:examples}. For instance, one may encode the degree information of a hypergraph in constructing its corresponding hypernetwork. In general, we consider the design of a measure hypernetwork representation of a hypergraph as treating the assignment of node and hyperedge weights $\mu$ and $\nu$, respectively, and hypernetwork function $\omega$ independently. Intuitively, $\mu$ and $\nu$ are treated as measuring the ``importance" of each node and hyperedge in the model (e.g., uniform weights mean equal importance in the analysis), while $\omega$ encodes interactions between nodes and hyperedges in a measure-independent manner. 
\end{remark}

\begin{example}[Data Matrices]\label{ex:data_matrices}
The main motivation for introducing the co-optimal transport framework in \cite{TitouanRedkoFlamary2020} was to compare \emph{data matrices}; i.e. rectangular matrices used to store vector-valued datasets. Such an object is represented in our formalism by taking $X$ to be a set of samples, $Y$ to be a set of features (in the machine learning sense), $\mu$ and $\nu$ to be some data-dependent distributions and defining $\omega(x,y)$ to be the value of feature $y$ on sample $x$.
\end{example}

\begin{example}[Continuous Examples]
Hypergraphs and data matrices involve finite sets, whereas we are working in a generality that allows infinite sets. A theoretical motivation for allowing infinite sets is that, as we will see in \autoref{thm:pseudometric}, this is necessary to guarantee completeness for the metric on hypernetworks defined below. More concretely, there are many interesting instances of infinite measure hypernetworks. For example, a hypernetwork $H = (X,\mu,Y,\nu,\omega)$ can encode a \emph{machine learning problem}, where $X$ is a \emph{data space} (say, $\Rspace^n$), $\mu$ is a distribution of data in $X$, $Y$ is a \emph{parameter space} (say, $\Rspace^m$), $\nu$ is a distribution representing constraints on the parameters, and $\omega:X \times Y \to \Rspace$ is considered as a parametric family of machine learning models; model training involves considering $\mu$ as a distribution of training data and optimizing an objective function of $\omega$ over parameters $Y$.  As another example, suppose that $(X,\mu)$ and $(Y,\nu)$ are bounded, measured subsets of a common ambient metric space $(Z,d_Z)$ and define $\omega:X \times Y \to Z$ by $\omega(x,y) = d_Z(x,y)$; this is the sort of setup that arises in a metric construction of Sturm, related to the usual embedding formulation of Gromov-Hausdorff distance  \cite{sturm2006geometry}. See \cite{mutlu2016bipolar} for other examples along these lines.
\end{example}

We are now prepared to define our distance between measure hypernetworks, which is a direct extension of the co-optimal transport framework of Redko {\etal} introduced in \cite{TitouanRedkoFlamary2020}.

\begin{definition}
\label{def:coot-distance}
Let $H = (X,\mu,Y,\nu,\omega), H' = (X',\mu',Y',\nu',\omega') \in \mathcal{H}$. For $p \in [1,\infty]$, the \emph{$p$-th co-optimal transport distortion} is the functional 
$
\mathrm{dis}_p^\mathcal{H} = \mathrm{dis}_{H,H',p}^\mathcal{H}: \mathcal{C}(\mu,\mu') \times \mathcal{C}(\nu,\nu') \to \Rspace
$
defined by
\begin{equation}
\label{eqn:COOT_loss}
\mathrm{dis}_p^\mathcal{H}(\pi,\xi) = \|\omega - \omega'\|_{L^p(\pi \otimes \xi)},
\end{equation}
where $\omega - \omega'$ is understood as the function $X \times X' \times Y \times Y' \to \Rspace$ taking $(x,x',y,y')$ to $\omega(x,y) - \omega'(x',y')$. For $p < \infty$, \eqref{eqn:COOT_loss} reads
\[
\mathrm{dis}_p^\mathcal{H}(\pi,\xi) = \left( \int_{Y \times Y'} \int_{X \times X'} \lvert \omega(x,y) - \omega'(x',y')\rvert^p \pi(dx \times dx') \xi(dy \times dy') \right)^{1/p}.
\]

The \emph{hypernetwork $p$-distance} is then defined to be
\begin{equation}\label{eqn:distance}
    d_{\mathcal{H},p}(H,H')= \inf_{\pi \in \mathcal{C}(\mu,\mu')} \inf_{\xi \in \mathcal{C}(\nu,\nu')} \mathrm{dis}_p^\mathcal{H}(\pi,\xi).
\end{equation}
\end{definition}

Similar to the GW setting, an important feature of this metric is that its practical computation involves determining optimal couplings $\pi$ and $\xi$ of both $X,X'$ and $Y,Y'$, respectively. In the hypergraph setting, these couplings correspond to soft matchings of both node sets \emph{and} hyperedge sets. Computational examples are provided in~\autoref{sec:soft-matching-toy}.

We emphasize that $d_{\mathcal{H},p}$ is a straightforward extension of the co-optimal transport  distance defined in \cite{TitouanRedkoFlamary2020}; there the authors restricted their attention to  measure hypernetworks whose underlying sets are finite, specifically focusing on data matrices. Our extension of \cite{TitouanRedkoFlamary2020} encompasses the infinite examples described above, and the generalization to infinite spaces is necessary in order to get a complete metric space---see  \autoref{thm:pseudometric} below---but practical computations will always reduce back to the finite setting.

\subsection{Properties of Hypernetwork Distance} 
\label{sec:properties}

We now describe some basic theoretical properties of the hypernetwork distance $d_{\mathcal{H},p}$. Our first result shows that $d_{\mathcal{H},p}$ is a peudometric on the space of measure hypernetworks. We say that $H$ and $H'$ are \emph{weakly isomorphic} if $d_{\mathcal{H},p}(H,H') = 0$---a more geometric characterization of weak isomorphism, along the lines of the characterization for $d_{\mathcal{N},p}$ in~\autoref{sec:networks}, is given below in Proposition \ref{prop:weak_isomorphism}.

\begin{theorem}
\label{thm:pseudometric}
The hypernetwork $p$-distance $d_{\mathcal{H},p}$ is a pseudometric on $\mathcal{H}$. The induced metric on the space of hypernetworks considered up to weak isomorphism is complete and geodesic.
\end{theorem}

The proof of~\autoref{thm:pseudometric} adapts techniques of \cite{ChowdhuryMemoli2019} and \cite{sturm2012space} in the GW setting. The main idea is that GW distance involves computing an integral over a product of a coupling with itself $\pi \otimes \pi$; the proofs can be adapted to handle the co-optimal transport distance, which integrates over a product of distinct couplings $\pi \otimes \xi$. For completeness, we have included the details of the proof in the Appendix~\ref{sec:proofs_properties}. In particular, we construct an explicit geodesic between the equivalence classes of $H$ and $H'$ of the form  
\[
\gamma(t) = \left(\mathrm{supp}(\pi), \pi, \mathrm{supp}(\xi), \xi, (1-t)\omega + t \omega'\right),
\]
where $\pi$ and $\xi$ are optimal couplings of the hypernetworks.

Next, we give a characterization of weak isomorphism in the measure hypernetwork setting.

\begin{definition}
\label{def:basic-weak-iso}
A \emph{basic weak isomorphism} of measure hypernetworks $H$ and $H'$ is a pair of measure-preserving maps $\phi:X \to X'$ and $\psi:Y \to Y'$ such that $\omega(x,y) = \omega'(\phi(x),\psi(y))$ for $\mu \otimes \nu$-almost every $(x,y) \in X \times Y$. 

Similar to the measure network setting, we define a \emph{strong isomorphism} between measure hypernetworks to be a basic weak isomorphism $(\phi,\psi)$ such that equality of hypernetwork functions holds for all points $(x,y) \in X \times Y$ and both maps are bijective with measure-preserving inverses.
\end{definition}

\begin{proposition}\label{prop:weak_isomorphism}
A pair of measure hypernetworks $H=(X,\mu,Y,\nu,\omega)$ and $H'=(X',\mu',Y',\nu',\omega')$ are weakly isomorphic if and only if there exists a measure hypernetwork $\overline{H} = (\overline{X},\overline{\mu},\overline{Y},\overline{\nu},\overline{\omega})$ and basic weak isomorphisms $(\phi,\psi)$ from $\overline{H}$ to $H$ and $(\phi',\psi')$ from $\overline{H}$ to $H'$.
\end{proposition}

The proof is similar to the one in the GW setting, but it involves some useful techniques, so we will include it here.  

\begin{proof}
For the forward direction, if $H$ and $H'$ are weakly isomorphic with coupling $(\pi,\xi)\in \mathcal{C}(\mu,\mu')\times\mathcal{C}(\nu,\nu')$ realizing $d_{\mathcal{H},p}(H,H') = 0$ (such a pair always exists; see Lemma \ref{lemma:realized} in the Appendix), for $\overline{H}$ define the following: $\overline{X}:=\mathrm{supp}(\pi)$, $\overline{\mu}:=\pi$, $\overline{Y}:=\mathrm{supp}(\xi)$ and $\overline{\nu}:=\xi$. Let $\phi:\overline{X}\rightarrow X$, $\phi':\overline{X}\rightarrow X'$, $\psi:\overline{Y}\rightarrow Y$, and $\psi':\overline{Y}\rightarrow Y'$ be coordinate projection maps. Each of these maps is measure-preserving by the marginal conditions on $\pi$ and $\xi$. We define $\overline{\omega}$ as the pullback function of the projection map; i.e., $\overline{\omega}((x,x'),(y,y')) = \omega(\phi(x),\psi(y))$. Then $d_{\mathcal{H},p}(H,H') = 0$ being realized by $\pi$ and $\xi$ implies that  $\overline{\omega}(\overline{x},\overline{y}) = \omega(\phi(\overline{x}),\psi(\overline{y}))=\omega'(\phi'(\overline{x}),\psi'(\overline{y}))$ for almost every $(\overline{x},\overline{y}) \in \overline{X} \times \overline{Y}$. 

For the reverse direction, we define couplings $\pi:= (\phi \times \phi')_\#\overline{\mu}$ and $\xi:=(\psi \times \psi')_\#\overline{\nu}$, which satisfy
\begin{align*}
    \mathrm{dis}_p^\mathcal{H}(\pi,\xi) &= \|\omega - \omega'\|_{L^p(\pi \otimes \xi)} \\
    &= \| \omega \circ (\phi \times \psi)-\omega' \circ (\phi' \times \psi') \|_{L^p(\overline{\mu} \otimes \overline{\nu})}  \\
    &\leq  \| \omega-\overline{\omega}\|_{L^p(\overline{\mu} \otimes \overline{\nu})} + \| \overline{\omega}-\omega'\|_{L^p(\overline{\mu} \otimes \overline{\nu})}  = 0,
\end{align*}
where the second line follows by the change-of-variables formula.
\end{proof}

We point out in the following example that the notion of basic weak isomorphism provides a formalism for describing a procedure used in hypergraph simplification, referred to as node and hyperedge collapses~\cite{ZhouRathorePurvine2022}.

\begin{example}[Node and Hyperedge Collapses as Basic Weak Isomorphisms]
\label{example:collapses}
In \cite{ZhouRathorePurvine2022}, simplification operations for hypergraphs called \emph{node collapses} and \emph{hyperedge collapses} are studied. 
As illustrated in~\autoref{fig:collapses}, node  collapse combines nodes that belong to exactly the same set of hyperedges into a single ``super-node'' (visualized by concentric ring glyph), while hyperedge collapse merges hyperedges that share exactly the same set of nodes into a ``super-hyperedge'' (visualized by a pie-chart glyph). 
This operation can be defined in generality for finite hypernetworks. 

Let $H = (X,\mu,Y,\nu,\omega)$ be a finite (i.e. $\lvert X \rvert, \lvert Y \rvert<\infty$) measured hypernetwork. Suppose that there exist $x_1,x_2 \in X$ such that $\omega(x_1,y) = \omega(x_2,y)$ for every $y \in Y$. Define a new hypernetwork $H' = (X',\mu',Y',\nu',\omega')$ with $X' = X \setminus \{x_2\}$, 
\[
\mu'(x) = \left\{\begin{array}{cl}
\mu(x) & x \neq x_1 \\
\mu(x_1) + \mu(x_2) & x = x_1,
\end{array}\right.
\]
$Y' = Y$, $\nu' = \nu$ and $\omega' = \omega \mid_{X' \times Y'}$. 
Then the map $\phi:X \to X'$ sending $x \mapsto x$ for $x \neq x_2$ and $x_2 \mapsto x_1$, together with the identity map $\psi:Y \to Y'$ define a basic weak isomorphism $H \to H'$. This basic weak isomorphism is called a \emph{node collapse} (see~\autoref{fig:collapses} left). The formal definition of a hyperedge collapse is similar (see~\autoref{fig:collapses} right). 

\begin{figure}[!ht]
	\centering		
	\includegraphics[width=1.0\columnwidth]{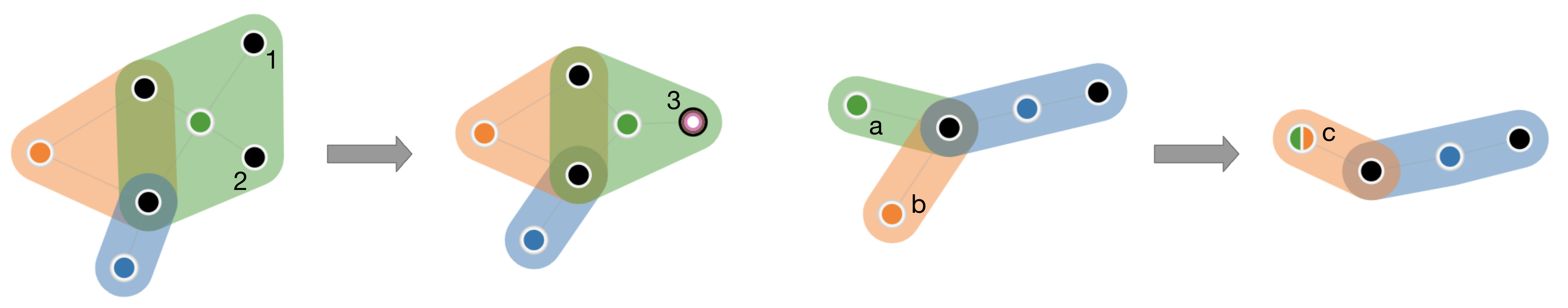}
	\vspace{-4mm}
	\caption{Node and hyperedge collapses. Left: nodes 1 and 2 merge into a super-node 3. Right: hyperedges a and b merge into a super-hyperedge c.}
	\label{fig:collapses}
\end{figure}
\end{example}

We note that the notion of weak isomorphism is strictly more general than basic weak isomorphism. As a simple example, consider the measure hypernetworks $H$ and $H'$ encoded by matrices
\[
\omega = \left[\begin{array}{ccc}
0 & 2 & 2 \\
1 & 2 & 2
\end{array}\right], 
\mu = \left[\begin{array}{c}
1/3 \\
2/3
\end{array}\right], 
\nu = \left[\begin{array}{c}
1/3 \\
1/3 \\
1/3
\end{array}\right] \mbox{ and } 
\omega' = \left[\begin{array}{cc}
0 & 2  \\
1 & 2  \\
1 & 2
\end{array}\right], 
\mu' = \nu, 
\nu' = \mu.
\]
Then there is no basic weak isomorphism $H \to H'$ or $H' \to H$. However, the measure hypernetworks are weakly isomorphic; a measure hypernetwork $\overline{H}$ as in Proposition \ref{prop:weak_isomorphism} is encoded by the matrices
\[
\overline{\omega} = \left[\begin{array}{ccc}
0 & 2 & 2 \\
1 & 2 & 2 \\
1 & 2 & 2
\end{array}\right], \,
\overline{\mu} = \overline{\nu}  = \left[\begin{array}{c}
1/3 \\
1/3 \\
1/3
\end{array}\right].
\]

\section{Graphification}
\label{sec:graphification}

A common technique in hypergraph analysis and machine learning is to transform a hypergraph into a traditional graph, which has a more tractable structure (\eg, \cite{agarwal2005beyond,pu2012hypergraph,SuranaChenRajapakse2021,xia2021self}). 
In this section, we formalize \emph{graphification} as the study of functors from the category of measure hypernetworks to the category of measure networks.

\subsection{Transformations of Hypergraphs to Graphs}\label{sec:transformations}

Let $H = (X, Y)$ be a (combinatorial) hypergraph with nodes $X = \{x_1, \cdots, x_n\}$ and hyperedges $Y = \{y_1, \cdots, y_m\}$, where each $y_i \subseteq X$. 
There are several transformations from a hypergraph to a graph: we focus on the  \emph{bipartite incidence graph} (also called the \emph{star expansion})~\cite{sun2008hypergraph}, the \emph{line graph}~\cite{BermondHeydemannSotteau1977}, and the \emph{clique expansion}~\cite{ZienSchlagChan1999}, shown in~\autoref{fig:hypergraph-dual}. 

\begin{figure}[!ht]
	\centering
	\includegraphics[width=0.9\columnwidth]{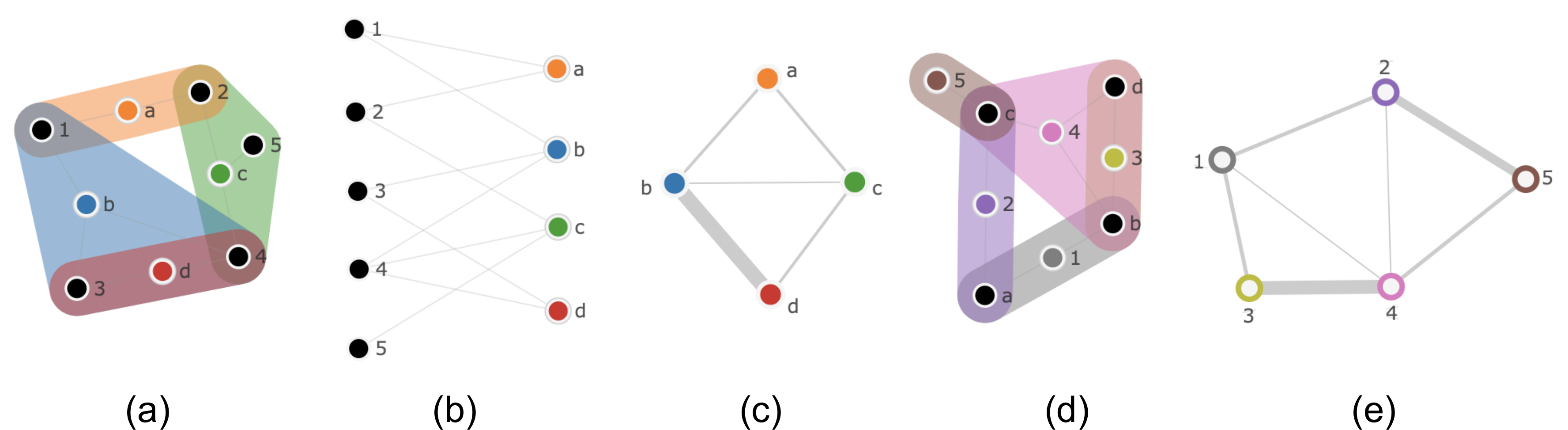}
	\caption{A hypergraph $H$ (a) with its bipartite incidence graph $\mathsf{B}(H)$ (b), line graph $\mathsf{L}(H)$ (c), dual hypergraph $H_*$ (d), and clique expansion $\mathsf{Q}(H)$ (e).}
	\label{fig:hypergraph-dual}
\end{figure}

A hypergraph $H$ may be represented by a bipartite incidence graph $\mathsf{B}(H)$ as follows: $X \cup Y$ forms a partition of the node  set of $\mathsf{B}(H)$, and $x \in X$ and $y \in Y$ are connected by an edge if and only if node $x$ is contained in hyperedge $y$ in $H$.  
The line graph $\mathsf{L}(H)$ of $H$ is a graph whose node set corresponds to the set of hyperedges of $H$;  two nodes are adjacent in $\mathsf{L}(H)$ if their corresponding hyperedges have a nonempty intersection in $H$. 
The clique expansion $\mathsf{Q}(H)$ of $H$ consists of node set $X$, and there is an edge $\{x_i, x_j\}$ in $\mathsf{Q}(H)$ if there exists some hyperedge $y \in Y$ such that $x_i, x_j \in y$.
That is, $\mathsf{Q}(H)$ is a graph constructed from $H$ by replacing each hyperedge with a clique among its nodes. 

The line graph and clique expansion are related through the concept of a \emph{dual hypergraph} $H_*$, which swaps the roles of nodes and hyperedges by considering the transpose of the incidence matrix of $H$: the clique expansion is the line graph of the dual, $\mathsf{Q}(H) = \mathsf{L}(H_*)$.

Finally, the edges in $\mathsf{B}(H)$, $\mathsf{L}(H)$, and $\mathsf{Q}(H)$ may be weighted based on  the similarities or dissimilarities between nodes and hyperedges of $H$. The particular weighting scheme is typically application-dependent or ad hoc---for example, a common practice \cite{rodri2002laplacian} is to weight an edge $\{x_i,x_j\}$ of $\mathsf{Q}(H)$ by the number of hyperedges in which $x_i$ and $x_j$ are commonly contained. In the following subsection, we develop principled models for these various transformations in the setting of measure hypernetworks.

In what follows, we take a categorical perspective on transformations from hypernetworks to networks, as described above. A function $\mathsf{F}:\mathcal{H} \to \mathcal{N}$ is called a \emph{graphification}. Given a graphification $\mathsf{F}$, the questions that we aim to address are:
\begin{itemize}
    \item Is $\mathsf{F}$ functorial (with respect to categorical structures defined below)?
    \item Is $\mathsf{F}$ Lipschitz with respect to $d_{\mathcal{H},p}$ and $d_{\mathcal{N},p}$?
\end{itemize}
Intuitively, functoriality tells us that the graphification process preserves relational information between hypergraphs---i.e., that it is \emph{structure-preserving}. Functoriality of the bipartite incidence, clique expansion and line graph maps for combinatorial hypergraphs was shown by D\"{o}rfler and Waller~\cite{DorflerWaller1980}; we aim to extend these results to our more general categories. The Lipschitz condition says that graphification is \emph{stable} in the sense that similar hypernetworks must be sent to similar networks.

\begin{remark}\label{rmk:hypergraphification}
Our current category-theoretic approach was  inspired by a line of work on the categorical aspects of (hierarchical) metric space and network partitioning schemes \cite{carlsson2018hierarchical,carlsson2013axiomatic,carlsson2021robust,carlsson2010characterization,carlsson2013classifying,carlsson2008persistent} and more general clustering schemes allowing overlapping clusters \cite{culbertson2016consistency,culbertson2018functorial}. A clustering algorithm can be interpreted as a function $\mathcal{N} \to \mathcal{H}$ (i.e., a \emph{hypergraphification}), and one could likewise ask about functoriality or Lipschitzness of various hypergraphifications. As an obvious example, consider the map which takes a measure network $(X,\mu,\omega)$ to the measure hypernetwork $(X,\mu,X,\mu,\omega)$; with the categorical structures defined in~\autoref{sec:network-category}, this map is functorial and it is clearly 1-Lipschitz (in fact, it is an isometric embedding when restricted to a certain subcategory of measure networks; see \cite[Proposition 3]{TitouanRedkoFlamary2020}). To keep things focused, we only consider graphifications in the present paper. We plan to explore theoretical properties of hypergraphifications in future work.
\end{remark}

\subsection{Categories of Networks}
\label{sec:network-category}

We now proceed with describing categorical structures on our spaces $\mathcal{H}$ and $\mathcal{N}$, starting with the network setting. 

\begin{definition}[Category of Measure Networks with Expansive Maps]
Let $N,N' \in \mathcal{N}$ be measure networks. A function $\phi:X \to X'$ between their underlying sets is called an \emph{expansive map} if $\omega(x,y) \leq \omega'(\phi(x),\phi(y))$ for every $(x,y) \in X \times X$. We will sometimes use the notation $N \xrightarrow[]{\phi} N'$ for an expansive map to emphasize its dependence on the network functions. Let $\Nmap$ denote the category whose objects are elements of $\mathcal{N}$ and whose morphisms are expansive maps. Composition of morphisms is function composition and the identity morphism is the identity map. 
\end{definition}

In the (combinatorial) graph theory literature, a common convention is to consider maps which preserve edge relations \cite{hell2004graphs}; that is, a morphism between combinatorial graphs $(V,E)$ and $(V',E')$ is a function $\phi:V \to V'$ such that $\{v,w\} \in E$ implies $\{\phi(v),\phi(w)\} \in E'$.  This property is generalized by expansive maps in the measure hypernetwork setting---indeed, if $N,N'$ are finite measure network representations of graphs, with $\omega$ and $\omega'$ binary adjacency functions, then the expansive condition says that $\phi$ takes edges to edges (if there is an edge joining vertices $x$ and $y$ in a network, then $\omega(x,y) = 1$, which forces $\omega(\phi(x),\phi(y))=1$, hence there is an edge joining $\phi(x)$ and $\phi(y)$). Thus, the category of combinatorial graphs embeds into $\mathcal{N}_{\mathrm{m}}$ as a full subcategory.

\begin{remark}
One could similarly define a \emph{non-expansive map} from $N$ to $N'$ to be a function $\phi:X \to X'$ with $\omega(x,y) \geq \omega'(\phi(x),\phi(y))$ for all $x,y \in X$. Non-expansive maps give the classical categorical structure for metric spaces, going back to Isbell \cite{isbell1964six}---in the metric space setting, this is the same as specifying that morphisms are 1-Lipschitz. This convention has been adopted in more general settings of networks (\eg,  \cite{culbertson2018functorial,carlsson2021robust}).  The appropriateness of the choice of either categorical structure depends on context: if one is considering measure networks as generalized metric measure spaces, then it makes more sense to consider the non-expansive (1-Lipschitz) category, whereas if one is considering measure networks as generalizing weighted adjacency structures, then the non-expansive category is more relevant. Since we are primarily interested in graphification functors, which are typically defined in terms of adjacency information, we focus on the category of expansive maps in this paper.
\end{remark}

It is natural to consider network morphisms which more strongly encode information about the measures. A first attempt would be to consider expansive maps which are also measure-preserving. This is too restrictive---for finite measure spaces, the set of measure-preserving maps between two  spaces is frequently empty. We give the following alternative, which defines a measured version of the Schmidt-Str{\"o}hlein categorical structure on graphs \cite[Chapter 7]{schmidt2012relations}, which consider morphisms between graphs given by relations between their node sets; we instead use couplings, which can be seen intuitively as weighted relations. The idea of defining morphisms between probability spaces as couplings has appeared previously in the literature, e.g., \cite{lawvere1962category, perrone2021} (see Remark \ref{rmk:lawvere_perrone}). Our notion of \emph{expansive coupling} defined below is, to our knowledge, novel.

\begin{definition}
Let $N,N' \in \mathcal{N}$ be measure networks. An \emph{expansive coupling} is a coupling $\pi$ of $\mu$ and $\mu'$ such that for $\pi \otimes \pi$-almost every pair $(x,x'),(y,y') \in X \times X'$, we have $\omega(x,y) \leq \omega'(x',y')$.
\end{definition}

Expansive couplings can be composed via a gluing operation, which takes some care to describe. We begin with a lemma, which is well-known in the optimal transport literature; see, e.g., \cite[Lemma 1.4]{sturm2012space} or \cite[Lemma 11.8.3]{villani2003topics}.

\begin{lemma}[Disintegration of Measures and the Gluing Lemma]\label{lem:gluing_lemma}
Let $(X,\mu),(X',\mu'),(X'',\mu'')$ be Polish spaces endowed with Borel probability measures and let $\pi \in \mathcal{C}(\mu,\mu')$ and $\pi' \in \mathcal{C}(\mu',\mu'')$.
\begin{enumerate}
\item For each $x' \in X'$, there exists a probability measure $\pi_{x'}$ on $X$  such that $\pi_{x'}(dx) = \pi(dx \times dx')$. The family $\{\pi_{x'}\}_{x' \in X'}$ is called a \emph{disintegration kernel for $\pi$ with respect to $X'$} and the $\pi_{x'}$ are uniquely determined for $\mu'$-almost every $x'$. Likewise, there exists a disintegration kernel $\{\pi_{x}\}_{x \in X}$ with respect to $X$, consisting of probability measures on $X'$ with analogous properties. 
\item There exists a unique probability measure $\pi' \boxtimes \pi$ on $X \times X' \times X''$ such that 
\[
(p_{X} \times p_{X'})_\# (\pi' \boxtimes \pi) = \pi \qquad \mbox{and} \qquad (p_{X'} \times p_{X''})_\# (\pi' \boxtimes \pi) = \pi'.
\]
This is called the \emph{gluing of $\pi$ and $\pi'$}. The gluing is characterized by the formula
\[
\pi' \boxtimes \pi (dx \times dx' \times dx'') = \pi'_{x'}(dx'') \pi_{x'}(dx) \mu'(dx').
\]
\end{enumerate}
\end{lemma}

\begin{definition}
    Let $N,N',N'' \in \mathcal{N}$ and let $\pi \in \mathcal{C}(\mu,\mu')$ and $\pi' \in \mathcal{C}(\mu',\mu'')$. We define the \emph{composition of $\pi$ and $\pi'$} to be 
    \[
    \pi' \bullet \pi := (p_X \times p_{X''})_\# (\pi' \boxtimes \pi).
    \]
\end{definition}

When working with expansive couplings, technical considerations require us to work with measure networks whose underlying topological spaces are compact. An element $N \in \mathcal{N}$ which has this property will be referred to as a \emph{compact measure network}. The need for this restriction is to get a well-defined composition operation for expansive couplings.

\begin{proposition}\label{prop:composition_preserves_expansiveness}
Let $N,N',N'' \in \mathcal{N}$ be compact measure networks and let $\pi \in \mathcal{C}(\mu,\mu')$ and $\pi' \in \mathcal{C}(\mu',\mu'')$ be expansive couplings. Then the composition $\pi' \bullet \pi$ is an expansive coupling of $N$ and $N''$. 
\end{proposition}

The proof will use the following lemma.

\begin{lemma}\label{lem:supp1}
    Let $X,X',X''$ be compact Polish spaces endowed with Borel probability measures $\mu,\mu',\mu''$ and let $\pi\in\mathcal{C}(\mu,\mu')$ and $\pi' \in \mathcal{C}(\mu',\mu'')$. 
    \begin{enumerate}
    \item If $(x,x')\in\mathrm{supp}(\pi)$, then $x\in\mathrm{supp}(\mu)$ and $x'\in\mathrm{supp}(\mu')$.
    \item If $x\in\mathrm{supp}(\mu)$, then there exists $x'\in X'$ such that $(x,x')\in\mathrm{supp}(\pi)$.
    \item If $(x,x'') \in \mathrm{supp}(\pi' \bullet \pi)$ then there exists $x' \in X'$ such that $(x,x',x'') \in \mathrm{supp}(\pi' \boxtimes \pi)$.
    \end{enumerate}
\end{lemma}

\begin{proof}
    
       We first prove Part 1. Suppose $(x,x')\in\mathrm{supp}(\pi)$, so that every open neighborhood of $(x,x')$ has positive measure with respect to $\pi$. Then for any open neighborhood $U$ of $x\in X$, we have that $U\times X'$ is open in $X\times X'$, so that $\mu(U)=\pi(U\times X')>0$. It follows that $x\in\mathrm{supp}(\mu)$. Similarly, $x'\in\mathrm{supp}(\mu')$. 
    
    Part 2 follows from \cite[Lemma 2.2]{dghlp-focm} and implies Part 3. Indeed, consider the map $X \times X' \times X'' \to X \times X'' \times X'$ given by rearranging coordinates. This map pushes the measure $\pi' \boxtimes \pi$ forward to a measure $\xi \in \mathcal{C}(\pi' \bullet \pi, \mu')$. By Part 2, there is $x' \in X'$ such that $(x,x'',x') \in \mathrm{supp}(\xi)$. The inverse of the coordinate rearrangement map pushes $\xi$ back to $\pi' \boxtimes \pi$ and it follows that $(x,x',x'') \in \mathrm{supp}(\pi' \boxtimes \pi)$.\end{proof}

\begin{proof}[Proof of Proposition \ref{prop:composition_preserves_expansiveness}]
For $(x,x''), (y,y'')$ in the support of $\pi' \bullet \pi$, by Part 3 of Lemma \ref{lem:supp1}, we have that there exists some $x',y'\in N'$ such that $(x,x',x'')$ and $(y,y',y'')$ are in the support of $\pi'\boxtimes\pi$. Then by Part 1 of Lemma \ref{lem:supp1}, we get that $(x,x')$ and $(y,y')$ are in the support of $\pi$ and $(x',x'')$ and $(y',y'')$ are in the support of $\pi'$. As $\pi$ and $\pi'$ are expansive couplings, we get that 
\[ \omega(x,y)\leq \omega'(x',y')\leq \omega''(x'',y'').\]
So $\pi' \bullet \pi$ is an expansive coupling of $N$ and $N''$.
\end{proof}

\begin{example}
We now observe that the restriction to compact measure networks was not only a matter of convenience, but in fact a necessity to get well-defined compositions.

Consider measure networks $N,N',N''$ defined as follows. Let $N = (X,\mu,\omega)$, where $X = \{\frac{1}{n} \mid n \in \mathbb{Z}_{>0}\} \cup \{0\}$, 
\[
\mu(x) = \left\{\begin{array}{cl}
c \cdot \frac{1}{n^2} & \mbox{if $x = \frac{1}{n}$} \\
0 & \mbox{if $x = 0$,}
\end{array}\right. \quad \mbox{and} \quad 
\omega(x,y) = \left\{\begin{array}{cl}
1 & \mbox{if $x = y = 0$} \\
0 & \mbox{otherwise,}
\end{array}\right.
\]
with $c = \frac{6}{\pi^2}$ chosen to normalize $\mu_X$  to be a probability measure. Observe that the set $X$ is Polish when topologized as a subspace of $\mathbb{R}$ and  that $\mu$ is fully supported. Next, we define $N' = (X',\mu',\omega')$ with $X' = \mathbb{Z}_{> 0}$ ($X'$ is a Polish space, as a subspace of $\mathbb{R}$) $\mu'(n) = c \cdot \frac{1}{n^2}$ (a fully supported measure) and $\omega'(m,n) = 1$ for all $m,n \in X'$. Finally, let $N'' = (X'',\mu'',\omega'')$ be defined by setting $X'' = X$, $\mu'' = \mu$, and $\omega''(x,y) = 1 - \omega(x,y)$. 

We now construct couplings to compose. Let $\pi \in \mathcal{C}(\mu,\mu')$ be defined by 
\[
\pi(x,x') = \left\{\begin{array}{cl}
c \cdot \frac{1}{n^2} & \mbox{if $x = \frac{1}{n}$ and $x' = n$} \\
0 & \mbox{otherwise}
\end{array}\right.
\]
and define $\pi' \in \mathcal{C}(\mu',\mu'')$ by $\pi'(x',x'') = \pi(x'',x')$. Then 
\[
\mathrm{supp}(\pi) = \{(1/n,n) \mid n \in \mathbb{Z}_{> 0}\} \quad \mbox{and} \quad \mathrm{supp}(\pi') = \{(n,1/n) \mid n \in \mathbb{Z}_{> 0}\}.
\]
It follows by construction that for all pairs $(\frac{1}{n},n)$ and $(\frac{1}{m},m)$ in the support of $\pi$, 
$
\omega\left(1/n,1/m\right) \leq \omega'(n,m)$,
so that $\pi$ is expansive. Likewise, $\pi'$ is expansive. However, $(0,0) \in \mathrm{supp}(\pi' \bullet \pi)$ and we have 
\[
1 = \omega(0,0) > \omega''(0,0) = 0,
\]
so that the composition $\pi' \bullet \pi$ is not expansive.
\end{example}

The identity morphism for $N=(X,\mu,\omega)$ in the context of expansive couplings will be the \emph{diagonal coupling}  $\mathbbm{1}_N:= (\mathrm{id}_X\times\mathrm{id}_X)_{\#}\mu $.  
This coupling acts as an identity follows from \cite[Proposition 4.6]{perrone2021}; the idea is that the disintegration kernel for $\mathbbm{1}_N$ at $x \in X$ is $(\mathbbm{1}_N)_x(dy) = \delta_{x,y} \cdot \mu(dx)$, with $\delta_{x,y}$ the Dirac indicator function for the condition $x=y$, so it follows that $\pi\bullet\mathbbm{1}_N=\pi=\mathbbm{1}_{N'}\bullet \pi.$
Associativity of composition of couplings is also shown in \cite[Proposition 4.6]{perrone2021}. With these ingredients, we are now able to define a new category structure for measure networks.

\begin{definition}[Category of Measure Networks with Expansive Couplings]
The \emph{category of measure networks with expansive couplings} $\mathcal{N}_\mathrm{c}$ is the category whose objects are \emph{compact} measure networks and whose morphisms are expansive couplings. Morphisms are composed by the gluing operation and the identity morphism of $N$ is the diagonal coupling $\mathbbm{1}_N$. We sometimes denote an expansive coupling between $N$ and $N'$ as $N \xrightarrow[]{\pi} N'$.
\end{definition}

We will see below (\autoref{thm:p_clique} and Corollary~\ref{cor:clique_not_functorial}) that this categorical structure allows for a richer variety of graphification functors than the expansive maps category. Moreover, this category allows us to capture the notion of strong isomorphism between measure networks in categorical formalism; see Section \ref{sec:isomorphisms} and Proposition \ref{prop:isomorphism_equivalences}.

\begin{remark}\label{rmk:lawvere_perrone}
Similar constructions of categorical structures on probability spaces are explored in  the lecture notes of Lawvere \citep{lawvere1962category} and more recently by Perrone \citep{perrone2021}. The former describes a category where the objects are arbitrary measurable spaces and the morphisms are Markov kernels. The author also adds a deterministic map which defines a subcategory that utilizes Dirac-delta probability measures. The latter defines a weighted category where the objects are probability measures on a Polish space $X$ paired with a measurable (pseudo-quasi-)metric and morphisms are couplings with a weight given by a cost function that mimics the $p$-distortion functional we have in Equation~\eqref{eqn:GW_loss}.
\end{remark}

We end this subsection by pinning down the precise categorical relationship between measure-preserving expansive maps and expansive couplings. Let $\mathcal{N}_\mathrm{mp}$ denote the category of compact measure networks whose morphisms are expansive measure-preserving maps. Consider the map $\mathsf{P}:\mathcal{N}_\mathrm{mp} \to \mathcal{N}_\mathrm{c}$ which is the identity on objects and which takes a measure-preserving expansive map $N \xrightarrow{\phi} N'$ to the \emph{induced coupling} $(\mathrm{id}_X \times \phi)_\# \mu$.

\begin{proposition}\label{prop:embedding_graphs}
    The map $\mathsf{P}:\mathcal{N}_\mathrm{mp} \to \mathcal{N}_\mathrm{c}$ is a faithful functor.
\end{proposition}

\begin{proof}
We first observe that $\mathsf{P}$ is well defined in the sense that if $N \xrightarrow[]{\phi} N'$ is an expansive measure-preserving map, then $\pi_\phi:= (\mathrm{id}_X \times \phi)_\# \mu$ is an expansive coupling. Indeed, the support of $\pi_\phi$ is $\{(x,\phi(x)) \mid x \in X\}$, so expansiveness of $\pi_\phi$ follows immediately from expansiveness of $\phi$.

We need to show that $\mathsf{P}$ respects compositions. Let $N \xrightarrow[]{\phi}N'$ and $N' \xrightarrow[]{\phi'} N''$ be expansive measure-preserving maps and let $\pi_\phi = \mathsf{P}(\phi)$ and $\pi_{\phi'} = \mathsf{P}(\phi')$. We first observe that the disintegration kernel for $\pi_\phi$ at $x \in X$ is the Dirac measure $\delta_{\phi(x)}(dx')$, so that $\pi_\phi(dx \times dx') = \delta_{\phi(x)}(dx')\mu(dx)$. Likewise, $\pi_{\phi'}(dx' \times dx'') = \delta_{\phi'(x')}(dx'')\mu'(dx')$. 

We claim that the gluing of $\pi'$ and $\pi$ is given by the formula
\[
\pi' \boxtimes \pi (dx \times dx' \times dx'') = \delta_{\phi(x)}(dx') \delta_{\phi'(x')}(dx'') \mu(dx)
\]
(note that this is not our usual formula because we disintegrated in a different order for the sake of explicitness). This is verified by checking the $X \times X'$ and $X' \times X''$ marginals and appealing to the uniqueness part of Lemma \ref{lem:gluing_lemma}: respectively, we have,
\begin{align*}
\int_{X''} \delta_{\phi(x)}(dx') \delta_{\phi'(x')}(dx'') \mu(dx) &= \delta_{\phi(x)}(dx')\mu(dx) \int_{X''} \delta_{\phi'(x')}(dx'') \\
&= \delta_{\phi(x)}(dx')\mu(dx) =  \pi_\phi(dx \times dx')
\end{align*}
and
\begin{align*}
\int_X \delta_{\phi(x)}(dx') \delta_{\phi'(x')}(dx'') \mu(dx) &= \delta_{\phi'(x')}(dx'') \int_X \delta_{\phi(x)}(dx') \mu(dx) \\
&= \delta_{\phi'(x')}(dx'') \int_{\{x \in X \mid \phi(x) = x'\}} \mu(dx) \\ 
&= \delta_{\phi'(x')}(dx'') \mu'(dx') = \pi_{\phi'}(dx' \times dx'').
\end{align*}

We obtain $\pi_{\phi'} \bullet \pi_\phi$ by taking the $X \times X''$ marginal of $\pi' \boxtimes \pi$, which yields
\begin{align*}
    \int_{X'} \delta_{\phi(x)}(dx') \delta_{\phi'(x')}(dx'') \mu(dx) &= \mu(dx) \int_{X'} \delta_{\phi'(x')}(dx'') \delta_{\phi(x)}(dx')  \\
    &= \mu(dx) \int_{\{x' \in X' \mid \phi'(x') = x''\}} \delta_{\phi(x)}(dx') \\
    &= \delta_{\phi' \circ \phi (x)}(dx'') \mu(dx) = \pi_{\phi' \circ \phi}(dx \times dx'').
\end{align*}
Therefore, $\mathsf{P}$ respects compositions.

Faithfulness of $\mathsf{P}$ follows by checking supports. Suppose that $\phi$ and $\psi$ are measure-preserving expansive maps from $N$ to $N'$. Then 
\[
\mathrm{supp}(\pi_\phi) = \{(x,\phi(x)) \mid x \in X'\} \quad \mbox{and} \quad \mathrm{supp}(\pi_\psi) = \{(x,\psi(x)) \mid x \in X'\},
\]
so that $\pi_\phi = \pi_\psi$ if and only if $\phi = \psi$.
\end{proof}

\subsection{Categories of Hypernetworks}

In the setting of combinatorial hypergraphs, an obvious generalization of a graph morphism is to consider a morphism from $(X,Y)$ to $(X',Y')$ to be a pair of maps $(\phi,\psi)$ with $\phi:X \to X'$ and $\psi:Y \to Y'$ such that $x \in y$ implies $\phi(x) \in \psi(y)$---this is the D\"{o}rfler-Waller hypergraph category~\cite{DorflerWaller1980}. This directly generalizes to the measure hypernetwork setting as an \emph{expansive map}. Similar to the measure network setting, we can involve measures in the morphism structure through \emph{expansive couplings}, generalizing the relation-based Schmidt-Str{\"o}hlein categorical structure on hypergraphs~\cite[Section 7.2]{schmidt2012relations}. 

\begin{definition}
Let $H,H' \in \mathcal{H}$ be measure hypernetworks. A pair of functions $(\phi,\psi)$, $\phi:X \to X'$ and $\psi:Y \to Y'$, is called an \emph{expansive map} if $\omega(x,y) \leq \omega'(\phi(x),\psi(y))$ for all $(x,y) \in X \times Y$. An \emph{expansive coupling} is a pair $(\pi,\xi)$, where $\pi \in \mathcal{C}(\mu,\mu')$, $\xi \in \mathcal{C}(\nu,\nu')$ and for $\pi \otimes \xi$-almost every pair $(x,x') \in X \times X'$ and $(y,y') \in Y \times Y'$, $\omega(x,y) \leq \omega'(x',y')$.

We use $\mathcal{H}_\mathrm{m}$ to denote the category whose objects are measure hypernetworks and whose morphisms are expansive maps. We use $\mathcal{H}_\mathrm{c}$ to denote the category whose objects are \emph{compact measure hypernetworks} (measure hypernetworks whose underlying topological spaces are both compact) and whose morphisms are expansive couplings. Expansive couplings are composed as $(\pi',\xi') \bullet (\pi,\xi) := (\pi' \bullet \pi, \xi' \bullet \xi)$ and the identity morphism on $H$ is $\mathbbm{1}_H = ((\mathrm{id}_X \times \mathrm{id}_X)_\# \mu, (\mathrm{id}_Y \times \mathrm{id}_Y)_\# \nu )$.
\end{definition}

The results of the previous subsection extend straightforwardly to show that the category $\mathcal{H}_\mathrm{c}$ is well-defined.

The next few subsections address the main problem above: are the transformations described in \autoref{sec:transformations} Lipschitz (with respect to $d_\mathcal{N}$ and $d_\mathcal{H}$) and functorial (with respect to the category structures defined above).

\subsection{Dualization Functor}

Generalizing the discussion in~\autoref{sec:transformations}, a measure hypernetwork $H = (X,\mu,Y,\nu,\omega)$ has a \emph{dual hypernetwork} $H_\ast = (Y,\nu,X,\mu,\omega_\ast)$, where
$
\omega_\ast(y,x) := \omega(x,y).
$
As a first application of the categorical perspective, we show that the dualization map is a functorial isometry.

\begin{proposition}\label{prop:dualization}
The dualization map $H \mapsto H_\ast$ is a covariant endofunctor for both categories $\mathcal{H}_\mathrm{m}$ and $\mathcal{H}_\mathrm{c}$ and is an isometry of the pseudometric space $(\Hcal,d_{\Hcal,p})$. 
\end{proposition}

\begin{proof}
Given measure hypernetworks $H,H' \in \mathcal{H}$ and an expansive map $(\phi,\psi)$ with $\phi:X \to X'$ and $\psi:Y \to Y'$, define $\phi_\ast = \psi$ and $\psi_\ast = \phi$, we have
\[
\omega_\ast(y,x) = \omega(x,y) \leq \omega'(\phi(x),\psi(y)) = \omega_\ast(\psi_\ast(y),\phi_\ast(x)),
\]
for every $(y,x)$, so that $(\psi_\ast, \phi_\ast)$ is an expansive map for $H_\ast$ and $H'_\ast$. The dualization map respects composition: given also $H'' \in \mathcal{H}$ and $\phi':X' \to X''$, $\psi':Y' \to Y''$, $(\phi' \circ \phi)_\ast = \psi' \circ \psi = \phi'_\ast \circ \phi_\ast$ and $(\psi' \circ \psi)_\ast = \phi' \circ \phi = \psi'_\ast \circ \psi_\ast$. The arguments in the category of expansive couplings are similar; given an expansive coupling $(\pi,\xi)$ of $H$ and $H'$, $(\xi,\pi)$ is an expansive coupling of $H_\ast$ and $H_\ast'$, and compositions are respected. 

The isometry claim follows from the observation that, for any couplings $\pi \in \mathcal{C}(\mu,\mu')$ and $\xi \in \mathcal{C}(\nu,\nu')$ for $H$ and $H'$, we have
\[
\mathrm{dis}_{H,H',p}(\pi,\xi)  = \mathrm{dis}_{H_\ast,H_\ast',p}(\xi,\pi).
\]
\end{proof}

\subsection{The Bipartite Incidence Map} 

For the rest of this section, we consider graphifications $\mathsf{F}: \mathcal{H} \to \mathcal{N}$ inspired by the transformations described in~\autoref{sec:transformations}.

As a first example of a graphification, we define a map $\mathsf{B}: \Hcal \to \mathcal{N}$ by $\mathsf{B}(X,\mu,Y,\nu,\omega) = (Z, \mu_\mathsf{B}, \omega_\mathsf{B})$, where $Z = X \cup Y$ is endowed with the disjoint union topology,
\[
\omega_\mathsf{B}(w,z) = \left\{
\begin{array}{cl}
\omega(w,z) & \mbox{if $w \in X$ and $z \in Y$} \\
\omega_\ast(w,z) & \mbox{if $z \in X$ and $w \in Y$} \\
0 & \mbox{otherwise.}
\end{array}\right.
\]
and $\mu_\mathsf{B}$ is defined as
\begin{equation}\label{eqn:incidence_measure}
\mu_\mathsf{B} = \frac{1}{2}((i_X)_\# \mu + (i_Y)_\# \nu),
\end{equation}
where $i_X:X \to Z$ and $i_Y:Y \to Z$ are inclusions. 

The map $\mathsf{B}$ generalizes the transformation from a hypergraph to its bipartite incidence (or star expansion) graph (\autoref{sec:transformations}). As such, we refer to $\mathsf{B}:\mathcal{H} \to \mathcal{N}$ as the \emph{bipartite incidence map}. In the statement below, we abuse notation and write $\mathsf{B}:\mathcal{H}_\mathrm{m} \to \mathcal{N}_\mathrm{m}$ and $\mathsf{B}:\mathcal{H}_\mathrm{c} \to \mathcal{N}_\mathrm{c}$ for functors defined on objects by $\mathsf{B}$ (or the restiction of $\mathsf{B}$ to the set of compact hypernetworks in the expansive coupling category).

\begin{proposition}
\label{prop:incidence_map}
The bipartite incidence map induces functors $\mathsf{B}:\mathcal{H}_\mathrm{m} \to \mathcal{N}_\mathrm{m}$ and $\mathsf{B}:\mathcal{H}_\mathrm{c} \to \mathcal{N}_\mathrm{c}$ and is $\frac{1}{2^{1/p}}$-Lipschitz with respect to $d_{\mathcal{H},p}$ and $d_{\mathcal{N},p}$ for all $p \in [1,\infty)$ and $1$-Lipschitz for $p = \infty$.
\end{proposition}

The proof will use an elementary lemma describing distributive properties of various probabilistic constructions.

\begin{lemma}\label{lem:distributive_laws}
Let $(X,\mu)$, $(X',\mu')$ and $(X'',\mu'')$ be Polish, Borel probability spaces, $\alpha,\beta \in \mathcal{C}(\mu,\mu')$ and $\alpha',\beta' \in \mathcal{C}(\mu',\mu'')$. Also let $\pi = \frac{1}{2}(\alpha + \beta) \in \mathcal{C}(\mu,\mu')$ and $\pi' = \frac{1}{2}(\alpha' + \beta') \in \mathcal{C}(\mu',\mu'')$. Then we have the following distributive laws:
\begin{enumerate}
    \item for $\mu'$-almost all $x' \in X'$, $\pi_{x'} = \frac{1}{2}(\alpha_{x'} + \beta_{x'})$;
    \item $\alpha' \boxtimes \pi = \frac{1}{2}(\alpha' \boxtimes \alpha + \alpha' \boxtimes \beta)$, and likewise $\pi' \boxtimes \alpha = \frac{1}{2}(\alpha' \boxtimes \alpha + \beta' \boxtimes \alpha)$;
    \item $\alpha' \bullet \pi = \frac{1}{2}(\alpha' \bullet \alpha + \alpha' \bullet \beta)$, and likewise $\pi' \bullet \alpha = \frac{1}{2}(\alpha' \bullet \alpha + \beta' \bullet \alpha)$.
\end{enumerate}
\end{lemma}

\begin{proof}
Part 1 follows from
\begin{align*}
(\alpha_{x'} + \beta_{x'})(dx)\mu'(dx') &= \alpha_{x'}(dx)\mu'(dx') + \beta_{x'}(dx)\mu'(dx') \\
&= \alpha(dx \times dx') + \beta(dx \times dx') = 2 \cdot \pi(dx \times dx').
\end{align*}
Part 2 holds because
\begin{align*}
\alpha' \boxtimes \pi(dx \times dx' \times dx'') &= \alpha'_{x'}(dx'') \pi_{x'}(dx) \mu'(dx') \\
&= \alpha'_{x'}(dx'') \cdot \frac{1}{2}(\alpha_{x'}(dx) + \beta_{x'}(dx)) \cdot \mu'(dx') \\
&= \frac{1}{2} \alpha'_{x'}(dx'') \alpha_{x'}(dx) \mu'(dx') + \frac{1}{2} \alpha'_{x'}(dx'') \beta_{x'}(dx') \mu'(dx') \\
&= \frac{1}{2}(\alpha' \boxtimes \alpha + \alpha' \boxtimes \beta)(dx \times dx' \times dx'').
\end{align*}
Finally, Part 3 follows easily: 
\begin{align*}
\alpha' \bullet \pi &= (p_{X \times X''})_\# \alpha' \boxtimes \pi  \\
&= (p_{X \times X''})_\# \left( \frac{1}{2}(\alpha' \boxtimes \alpha + \alpha' \boxtimes \beta) \right) \\
&= \frac{1}{2} \left( (p_{X \times X''})_\# \alpha' \boxtimes \alpha + (p_{X \times X''})_\# \alpha' \boxtimes \beta \right) \\
&= \frac{1}{2}(\alpha' \bullet \alpha + \alpha' \bullet \beta).
\end{align*}\end{proof}

\begin{proof}[Proof of Proposition \ref{prop:incidence_map}]
Throughout the proof, let $H,H',H'' \in \mathcal{H}$ be arbitrary measure hypernetworks.

\para{Functorial.} We first establish functoriality with respect to expansive maps. Let $(\phi,\psi)$ be an expansive map for measure hypernetworks $H$ and $H'$ and define $\mathsf{B}(\phi,\psi): \mathsf{B}(H) \to \mathsf{B}(H')$ by
\[
\mathsf{B}(\phi,\psi)(z) = \left\{\begin{array}{cl}
\phi(z) & \mbox{if $z \in X$} \\
\psi(z) & \mbox{if $z \in Y$}.
\end{array}\right.
\]
Then for every $(w,z) \in Z \times Z$,
\begin{align*}
\omega_\mathsf{B}(w,z) &= \left\{
\begin{array}{cl}
\omega(w,z) & \mbox{if $w \in X$ and $z \in Y$} \\
\omega_\ast(w,z) = \omega(z,w) & \mbox{if $z \in X$ and $w \in Y$} \\
0 & \mbox{otherwise}
\end{array}\right. \\
&\leq 
\left\{
\begin{array}{cl}
\omega'(\phi(w),\psi(z)) & \mbox{if $w \in X$ and $z \in Y$} \\
\omega'(\phi(z),\psi(w)) = \omega_\ast'(\psi(w),\phi(z)) & \mbox{if $z \in X$ and $w \in Y$} \\
0 & \mbox{otherwise}
\end{array}\right. \\
&= \omega_\mathsf{B}'(\mathsf{B}(\phi,\psi)(w),\mathsf{B}(\phi,\psi)(z)),
\end{align*}
so that $\mathsf{B}(\phi,\psi)$ is also expansive. Moreover, $\mathsf{B}$ respects composition: given expansive maps $(\phi,\psi)$ from $H$ to $H'$ and $(\phi',\psi')$ from $H'$ to $H''$, we have
\[
\mathsf{B}(\phi',\psi') \circ \mathsf{B}(\phi,\psi)(z) = \left\{\begin{array}{cl}
\mathsf{B}(\phi',\psi')(\phi(z)) & \mbox{if $z \in X$} \\
\mathsf{B}(\phi',\psi')(\psi(z)) & \mbox{if $z \in Y$.}
\end{array}\right.
\]

Functoriality for expansive couplings is similar. Now assume that $H,H',H''$ are compact measure hypernetworks. Given an expansive coupling $(\pi,\xi)$ of $H$ and $H'$, define
\begin{equation}\label{eqn:incidence_coupling}
\mathsf{B}(\pi,\xi) = \frac{1}{2}\left((i_{X \times X'})_\# \pi + (i_{Y \times Y'})_\# \xi\right),
\end{equation}
where, for $A \in \{X,Y\}$, $i_{A \times A'}$ is the inclusion map taking $A \times A'$ into
\[
Z \times Z' = (X \cup Y) \times (X' \cup Y') = X \times X' \cup X \times Y' \cup Y \times X' \cup Y \times Y'.
\]
We claim that $\mathsf{B}(\pi,\xi)$ is an expansive coupling of $\mathsf{B}(H)$ and $\mathsf{B}(H')$. In the following, we denote the proposed coupling as $\rho := \mathsf{B}(\pi,\xi)$ in order to compactify notation. To see that $\rho$ is a coupling, consider the coordinate projection $p_Z:Z \times Z' \to Z$. We have 
\begin{align*}
(p_Z)_\# \rho &= (p_Z)_\#\frac{1}{2}\left((i_{X \times X'})_\# \pi + (i_{Y \times Y'})_\# \xi\right) \\
&= \frac{1}{2}\left((p_Z)_\#(i_{X \times X'})_\# \pi + (p_Z)_\#(i_{Y \times Y'})_\# \xi\right).
\end{align*}
Now observe that
\[
(p_Z)_\#(i_{X \times X'})_\# = (p_Z \circ i_{X \times X'})_\# = (i_X \circ p_X)_\# = (i_X)_\# (p_X)_\#,
\]
where $i_X:X \to Z$ is inclusion and $p_X:X \times X' \to X$ is projection. Using a similar equivalence for $Y$, we have
\[
(p_Z)_\# \rho = \frac{1}{2}\left((i_X)_\# (p_X)_\# \pi + (i_Y)_\# (p_Y)_\# \xi \right) = \frac{1}{2}\left((i_X)_\# \mu + (i_Y)_\# \nu \right) = \mu_\mathsf{B}.
\]
Similar reasoning holds for the projection to $Z'$, so we conclude that $\rho$ is a coupling. Expansiveness of $\rho$ follows by an argument similar to the expansive maps case above.

It remains to show that compositions are respected under $\mathsf{B}$. Let $(\pi,\xi)$ be an expansive coupling of measure hypernetworks $H$ and $H'$ and let $(\pi',\xi')$ be an expansive coupling of $H'$ and $H''$. It suffices to show that $\mathsf{B}(\pi' \bullet \pi, \xi' \bullet \xi)(A \times C) = \mathsf{B}(\pi', \xi') \bullet \mathsf{B}(\pi, \xi)(A \times C)$ on product sets $A \times C \subset Z \times Z''$. For an arbitrary product set, we have 
\[
A \times C = (A \cap X) \times (C \cap X'') \cup (A \cap X) \times (C \cap Y'') \cup (A \cap Y) \times (C \cap X'') \cup (A \cap Y) \times (C \cap Y''),
\]
where the unions are disjoint. By additivity, we can restrict our attention to the special cases where $A \times C$ is a subset of one of the product sets $X \times X''$, $X \times Y''$, $Y \times X''$ or $Y \times Y''$. The proof strategies for the various cases are the same, so we will give details in the case that $A \times C \subset X \times X''$.  In this case, we have
\begin{align}
2 \cdot \mathsf{B}(\pi' \bullet \pi, \xi' \bullet \xi)(A \times C) &= (i_{X \times X''})_\# (\pi' \bullet \pi) (A \times C) + (i_{Y \times Y''})_\# (\xi' \bullet \xi) (A \times C) \nonumber \\
&=  \pi' \bullet \pi (A \times C), \label{eqn:equality_of_compositions_3}
\end{align}
since $i_{Y \times Y''}^{-1}(A \times C) = \emptyset$.

Now consider $\mathsf{B}(\pi', \xi') \bullet \mathsf{B}(\pi, \xi)(A \times C)$. Lemma \ref{lem:distributive_laws} implies
\begin{align*}
&\mathsf{B}(\pi', \xi') \bullet \mathsf{B}(\pi, \xi) \\
& = \frac{1}{2}\left((i_{X' \times X''})_\# \pi' + (i_{Y' \times Y''})_\# \xi'\right)\bullet \frac{1}{2}\left((i_{X \times X'})_\# \pi + (i_{Y \times Y'})_\# \xi\right) \\
&= \frac{1}{4} \left((i_{X' \times X''})_\# \pi' \bullet (i_{X \times X'})_\# \pi + (i_{X' \times X''})_\# \pi' \bullet (i_{Y \times Y'})_\# \xi \right. \\
&\hspace{1in} \left. + (i_{Y' \times Y''})_\# \xi' \bullet (i_{X \times X'})_\# \pi + (i_{Y' \times Y''})_\# \xi' \bullet (i_{Y \times Y'})_\# \xi\right).
\end{align*}
By examining supports of the factors, we conclude that, for a product set $A \times C \subset X \times X''$, 
\begin{equation}\label{eqn:equality_of_compositions_2}
\mathsf{B}(\pi', \xi') \bullet \mathsf{B}(\pi, \xi)(A \times C) = \frac{1}{4}\cdot \left((i_{X' \times X''})_\# \pi' \bullet (i_{X \times X'})_\# \pi\right) (A \times C).
\end{equation}
Comparing \eqref{eqn:equality_of_compositions_3} and \eqref{eqn:equality_of_compositions_2}, we see that it is left to show that
\begin{equation}\label{eqn:equality_of_compositions}
\left((i_{X' \times X''})_\# \pi' \bullet (i_{X \times X'})_\# \pi\right) (A \times C) = 2\cdot \pi' \bullet \pi(A \times C),
\end{equation}
for $A \times C \subset X \times X''$. To see this, we first observe that, for $x \in X$ and $x' \in X'$,
\begin{align*}
\pi_{x'}(dx)\mu'(dx') &= \pi(dx \times dx') \\
&= (i_{X \times X'})_\# \pi(dx \times dx') \\
&= \left((i_{X \times X'})_\# \pi \right)_{x'}(dx) \mu'_{\mathsf{B}}(dx') \\
&= \frac{1}{2}  \left((i_{X \times X'})_\# \pi \right)_{x'}(dx) \mu'(dx'),
\end{align*}
and we deduce that 
\begin{equation}\label{eqn:comparing_disintegrations}
\left((i_{X \times X'})_\# \pi \right)_{x'} = 2 \cdot \pi_{x'}.
\end{equation}
We apply this to derive \eqref{eqn:equality_of_compositions} as follows:
\begin{align*}
&\left((i_{X' \times X''})_\# \pi' \bullet (i_{X \times X'})_\# \pi\right) (A \times C) \\
& \hspace{0.5in}= \int_{X'} \left((i_{X' \times X''})_\# \pi' \right)_{x'}(C)\left((i_{X \times X'})_\# \pi \right)_{x'}(A) \mu'_{\mathsf{B}}(dx') \\
&\hspace{0.5in}= \int_{X'} 2\pi'_{x'}(C) \cdot 2 \pi_{x'}(A) \cdot \frac{1}{2} \mu'(dx') \\
&\hspace{0.5in}= 2 \cdot \pi' \bullet \pi (A \times C),
\end{align*}
where the second equality follows by applying \eqref{eqn:comparing_disintegrations}, together with the corresponding statement for $X' \times X''$, and from the calculation
\[
\mu'_{\mathsf{B}}(dx') = \frac{1}{2}((i_{X'})_\#\mu'(dx') + (i_{Y'})_\#\nu'(dx')) = \frac{1}{2}\mu'(dx').
\]

\para{Lipschitz.} It remains to show that $\mathsf{B}$ is Lipschitz. Let $H,H' \in \mathcal{H}$ and let $\pi \in \mathcal{C}(\mu,\mu')$ and $\xi \in \mathcal{C}(\nu,\nu')$ and consider the coupling $\rho = \mathsf{B}(\pi,\xi) \in \mathcal{C}(\mu_\mathsf{B},\mu_\mathsf{B}')$ from \eqref{eqn:incidence_coupling}. 

For $p \in [1,\infty)$, consider the $p$-network distortion for $\mathsf{B}(H)$ and $\mathsf{B}(H')$,
\begin{equation}\label{eqn:network_distortion_bipartite}
\mathrm{dis}_{p}^\mathcal{N}(\rho)^p = \int_{Z \times Z'} \int_{Z\times Z'} \lvert \omega_\mathsf{B}(w,z) - \omega_\mathsf{B}'(w',z') \rvert^p \rho(dw \times dw') \rho(dz \times dz').
\end{equation}
The measure being integrated against can be written as
\begin{align*}
&\rho \otimes \rho = \frac{1}{4}\left( (i_{X \times X'})_\# \pi \otimes (i_{X \times X'})_\# \pi \right. \\
&\qquad \qquad \qquad \left.+ 2 (i_{X \times X'})_\# \pi \otimes (i_{Y \times Y'})_\# \xi + (i_{Y \times Y'})_\# \xi \otimes (i_{Y \times Y'})_\# \xi\right),
\end{align*}
so that \eqref{eqn:network_distortion_bipartite} is a sum of integrals with respect to each term. The integral over the first term is
\begin{align*}
    &\frac{1}{4}\int_{Z \times Z'} \int_{Z \times Z'} \lvert \omega_\mathsf{B}(w,z) - \omega_\mathsf{B}'(w',z') \rvert^p (i_{X \times X'})_\# \pi(dw \times dw') (i_{X \times X'})_\# \pi(dz \times dz') \\
    &\quad = \frac{1}{4}\int_{X \times X'} \int_{X \times X'} \lvert \omega_\mathsf{B}(x_1,x_2) - \omega_\mathsf{B}'(x_1',x_2') \rvert^p \pi(dx_1 \times dx_1') \pi(dx_2 \times dx_2') = 0,
\end{align*}
by the definition of $\omega_\mathsf{B}$. Likewise, the integral with respect to the $(i_{Y \times Y'})_\# \xi \otimes (i_{Y \times Y'})_\# \xi$-term vanishes. Finally, the integral with respect to the cross term becomes
\[
\frac{1}{2} \int_{X \times X'} \int_{Y \times Y'} \lvert \omega(x,y) - \omega'(x',y') \rvert^p \pi(dx \times dx')\xi(dy \times dy') = \frac{1}{2} \mathrm{dis}^\mathcal{H}_p(\pi,\xi)^p.
\]
Therefore
$
\mathrm{dis}_p^\mathcal{N}(\rho) = \frac{1}{2^{1/p}}\mathrm{dis}_p^\mathcal{H}(\pi,\xi),
$
and by a limiting argument we see that 
$
\mathrm{dis}_\infty^\mathcal{N}(\rho) = \mathrm{dis}_\infty^\mathcal{H}(\pi,\xi).
$
Since $\pi$ and $\xi$ were arbitrary, the Lipschitz claim follows.

\end{proof}

In general, one can construct non-weakly isomorphic hypernetworks $H,H'$ such that $d_{\mathcal{N},p}(\mathsf{B}(H),\mathsf{B}(H')) = 0$ (e.g., take $H' = H_\ast$ for any $H$ not weakly isomorphic to its dual), indicating that $\mathsf{B}$ does not define a bi-Lipschitz equivalence. However, the proof of Proposition \ref{prop:incidence_map} suggests an alternative characterization of $d_{\mathcal{H},p}$, which we will now explain.

\begin{definition}
A measure network $N = (Z,\mu,\omega)$ is called \emph{bipartite} if there exists a topological disconnection $Z = X \cup Y$ such that $\mu(X) = \mu(Y) = \frac{1}{2}$ and for all $x_1,x_2 \in X$ and $y_1,y_2 \in Y$ we have $\omega(x_1,x_2) = \omega(y_1,y_2) = 0$ and $\omega(x_1,y_1) = \omega(y_1,x_1)$. A \emph{labeled bipartite network} is a bipartite network with a fixed decomposition, written as $N = (X,Y,\mu,\omega)$. The collection of labeled bipartite measure networks is denoted $\mathcal{B}$. One defines expansive maps and expansive couplings between labeled bipartite measure networks in the obvious way. Let $\mathcal{B}_\mathrm{m}$ denote the category of labeled bipartite networks and expansive maps as morphisms, and let $\mathcal{B}_\mathrm{c}$ denote the category of compact labeled bipartite networks with expansive couplings as morphisms.
\end{definition}

We now define a variant of the network GW distance for labeled bipartite networks $N = (X,Y,\mu,\omega)$ and $N' = (X',Y',\mu',\omega')$ which respects labels. A \emph{labeled coupling} of $\mu$ and $\mu'$ is a coupling $\pi \in \mathcal{C}(\mu,\mu')$ such that 
$
\mathrm{supp}(\pi) \subset (X \times X') \cup (Y \times Y').
$
Let $\mathcal{C}_\mathcal{B}((X,Y,\mu),(X',Y',\mu'))$ denote the set of all labeled couplings.

\begin{definition}
The \emph{labeled Gromov-Wasserstein $p$-distance} $d_{\mathcal{B},p}:\mathcal{B} \times \mathcal{B} \to \Rspace$ is 
\[
    d_{\mathcal{B},p}(N,N') = \inf_{\pi \in \mathcal{C}_\mathcal{B}((X,Y,\mu),(X',Y',\mu'))} \mathrm{dis}_p^\mathcal{N}(\pi).
\] 
\end{definition}

\begin{theorem}
\label{thm:bipartite_equivalence}
The bipartite incidence map induces functors $\mathsf{B}:\mathcal{H}_\mathrm{m} \to \mathcal{B}_\mathrm{m}$ and $\mathsf{B}:\mathcal{H}_\mathrm{c} \to \mathcal{B}_\mathrm{c}$, and is an isometry from $(\mathcal{H},d_{\mathcal{H},p})$ to $(\mathcal{B},2^{1/p} \cdot d_{\mathcal{B},p})$ for $p \in [1,\infty)$ and to $(\mathcal{B},d_{\mathcal{B},p})$ for $p = \infty$.
\end{theorem}
\begin{proof}
Functoriality follows from Proposition \ref{prop:incidence_map}. It remains to show that $\mathsf{B}$ is an isometry onto $\mathcal{B}$.

The bipartite incidence functor is a bijection $\mathsf{B}:\mathcal{H} \to \mathcal{B}$, with inverse
\[
(X,Y,\mu,\omega) \to (X,2\cdot \mu\mid_X,Y,2\cdot \mu\mid_Y,\omega\mid_{X \times Y}),
\]
where $\mu\mid_X$ is the measure on $X$ defined by 
$
\mu\mid_X(U) = \mu(U)
$
for any Borel $U \subset X \subset X \cup Y$, $\mu\mid_Y$ is defined similarly and $\omega\mid_{X \times Y}$ is the restriction of $\omega$ to pairs in $X \times Y \subset (X \cup Y) \times (X \cup Y)$. 

Now let $H,H' \in \mathcal{H}$. We consider $p \in [1,\infty)$, with the $p = \infty$ case following similarly. Proposition \ref{prop:incidence_map} shows that $2^{1/p} d_{\mathcal{B},p}(\mathsf{B}(H),\mathsf{B}(H')) \leq  d_{\mathcal{H},p}(H,H')$: given $\pi \in \mathcal{C}(\mu,\mu')$ and $\xi \in \mathcal{C}(\nu,\nu')$, the coupling $\mathsf{B}(\pi,\xi)$ defined in \eqref{eqn:incidence_coupling} is an element of $\mathcal{C}_\mathcal{B}((X,Y,\mu_\mathsf{B}),(X',Y',\mu_\mathsf{B}'))$ ($\mu_\mathsf{B}$ and $\mu_\mathsf{B}'$ defined as in \eqref{eqn:incidence_coupling}), so the proof of Lipschitzness goes through. To prove the reverse inequality, let $\rho$ be an arbitrary coupling in $\mathcal{C}_\mathcal{B}((X,Y,\mu_\mathsf{B}),(X',Y',\mu_\mathsf{B}'))$ and define $\pi$ and $\xi$ on Borel sets $U,V$ by
\[
\pi(U) = 2\cdot\rho(U \cap (X \times X')) \qquad \mbox{and} \qquad \xi(V) = 2\cdot\rho(V \cap (Y \times Y')).
\]
Then $\pi \in \mathcal{C}(\mu,\mu')$ and $\xi \in \mathcal{C}(\nu,\nu')$ and computations similar to those in the proof of Proposition \ref{prop:incidence_map} can be used to show that
\begin{align*}
    \mathrm{dis}_{\mathsf{B}(H),\mathsf{B}(H'),p}(\rho) &= \frac{1}{2^{1/p}}\mathrm{dis}_{H,H',p}(\pi,\xi).
\end{align*}
It follows that $2^{1/p} d_{\mathcal{B},p}(\mathsf{B}(H),\mathsf{B}(H')) \geq d_{\mathcal{H},p}(H,H')$.
\end{proof}

\subsection{The Clique Expansion and Line Graph Maps}\label{sec:clique_expansion}

We now give similar consideration to the line graph and clique expansion maps described in~\autoref{sec:transformations}.

\begin{definition}\label{def:clique_expansion}
Let $H = (X,\mu,Y,\nu,\omega)$ be a measure hypernetwork. For $q \in [1,\infty]$, we define the \emph{$q$-clique expansion} of $H$ to be the measure network $\mathsf{Q}_q(H) = (X,\mu,\omega_{\mathsf{Q}_q})$, where
\[
\omega_{\mathsf{Q}_q}(x_1,x_2) := \left\|\min_{j=1,2} \omega(x_j,\cdot) \right\|_{L^q(\nu)},
\]
where $\omega(x,\cdot)$ denotes the measurable function $Y \to \Rspace$ defined by $y \mapsto \omega(x,y)$. 

Similarly, the \emph{$q$-line graph} of $H$ is defined to be $\mathsf{L}_q(H) = (Y,\nu,\omega_{\mathsf{L}_q})$, where
\[
\omega_{\mathsf{L}_q}(y_1,y_2) := \left\|\min_{j=1,2} \omega(\cdot,y_j) \right\|_{L^q(\mu)},
\]
where $\omega(\cdot,y):X \to \Rspace$ is defined by $x \mapsto \omega(x,y)$.
\end{definition}

\begin{remark}
Suppose that $(X,Y)$ is a (combinatorial) hypergraph, represented as a measure hypernetwork $H = (X,\mu,Y,\nu,\omega)$ with $\mu$ and $\nu$ uniform and $\omega$ encoding hyperedge incidence ($\omega(x,y) = 1$ if $x \in y$ and zero otherwise). Then $Q_\infty(H) = (X,\mu,\omega_{\mathsf{Q}_\infty})$ captures the classical clique expansion construction, where $\omega_{\mathsf{Q}_\infty}$ is the adjacency function. Indeed, for all $x_1,x_2 \in X$, $\omega_{\mathsf{Q}_\infty}(x_1,x_2) = 1$ if and only if there exists $y \in Y$ such that $x_1,x_2 \in y$, and is equal to zero otherwise. Similarly, $Q_1(H)$ returns the clique expansion with edge $\{x_1,x_2\}$ weighted by the proportion of hyperedges $y$ containing both $x_1$ and $x_2$, another typical convention.
\end{remark}

\begin{remark}
It is clear, by definition, that $\mathsf{L}_q(H) = \mathsf{Q}_q(H_\ast)$, where $H_\ast$ is the dual of $H$. Then $\mathsf{L}_q(H)$ captures the line graph transformation from~\autoref{sec:transformations} (particularly, in the $q=1$ and $q=\infty$ cases).
\end{remark}

\begin{theorem}\label{thm:p_clique}
The $q$-clique expansion and $q$-line graph maps $\mathsf{Q}_q,\mathsf{L}_q:\mathcal{H} \to \mathcal{N}$:
\begin{itemize}
    \item are Lipschitz with respect to $d_{\mathcal{H},p}$ and $d_{\mathcal{N},p}$ when $1 \leq q \leq p \leq \infty$, and
    \item induce functors $\mathsf{Q}_q,\mathsf{L}_q:\mathcal{H}_\mathrm{c} \to \mathcal{N}_\mathrm{c}$ for all $q \in [1,\infty]$ and $\mathsf{Q}_\infty,\mathsf{L}_\infty:\mathcal{H}_\mathrm{m} \to \mathcal{N}_\mathrm{m}$.
\end{itemize}
\end{theorem}

The proof will use some technical lemmas. Their proofs will be delayed to Appendix \ref{sec:proofs_clique}.

The first lemma says that Gromov-Wasserstein and co-optimal transport distances can always be realized as $L^p$ distances between functions over a common space, by ``lifting" to a weak isomorphism representative. We remark that the idea of the lemma is implicitly implemented in \cite{chowdhury2020gromov} in the setting of GW distance for the purpose of computation of Fr\'{e}chet means of ensembles of networks (see also \cite{li2021sketching}).

\begin{lemma}\label{lem:weak_iso_reps}
Let $H,H'$ be measure hypernetworks. There exist measure hypernetworks $\overline{H} = (\overline{X},\overline{\mu},\overline{Y},\overline{\nu},\overline{\omega})$ and $\overline{H'} = (\overline{X},\overline{\mu},\overline{Y},\overline{\nu},\overline{\omega}')$ (i.e., with the same underlying measure spaces, but different hypernetwork functions) such that $\overline{\mu}$ and $\overline{\nu}$ are fully supported, $H$ is weakly isomorphic to $\overline{H}$, $H'$ is weakly isomorphic to $\overline{H'}$, and for all $p \in [1,\infty]$ 
\begin{equation}\label{eqn:weak_iso_reps}
d_{\mathcal{H},p}(H,H') = d_{\mathcal{H},p}(\overline{H},\overline{H'}) = \|\overline{\omega} - \overline{\omega}'\|_{L^p(\overline{\mu} \otimes \overline{\nu})},
\end{equation}
where, in this setting, $\overline{\omega} - \overline{\omega}'$ is understood as the function $X \times Y \to \Rspace$ taking $(x,y)$ to $\overline{\omega}(x,y) - \overline{\omega}'(x,y)$. 

Similarly, for measure networks $N,N'$, there exist measure networks $\overline{N} = (\overline{X},\overline{\mu},\overline{\omega})$ and $\overline{N'} = (\overline{X},\overline{\mu},\overline{\omega}')$ such that $\overline{\mu}$ is fully supported, $N$ is weakly isomorphic to $\overline{N}$, $N'$ is weakly isomorphic to $\overline{N'}$, and for all $p \in [1,\infty]$, 
\[
d_{\mathcal{N},p}(N,N') = d_{\mathcal{N},p}(\overline{N},\overline{N'}) = \|\overline{\omega} - \overline{\omega}'\|_{L^p(\overline{\mu} \otimes \overline{\mu})}.
\]
\end{lemma}

\begin{lemma}\label{lem:well_defined}
If $H$ and $H'$ are weakly isomorphic measure hypernetworks, then $\mathsf{Q}_q(H)$ and $\mathsf{Q}_q(H')$ are weakly isomorphic measure networks for all $q \in [1,\infty]$.
\end{lemma}

\begin{lemma}[\cite{dghlp-focm}, Theorem 5.1(d)]\label{lem:norm_lower_bound}
Let $N = (X,\mu,\omega), N' = (X,\mu,\omega')$ be measure networks over the same measure space. Then, for $p \in [1,\infty]$,
    $
    d_{\mathcal{N},p}(N,N') \leq \|\omega - \omega'\|_{L^p(\mu \otimes \mu)}.
    $
    Here, we are considering $\omega - \omega'$ as the function $X \times X \to \Rspace$ defined by $(x,x') \mapsto \omega(x,x') - \omega'(x,x')$.
\end{lemma}

\begin{proof}[Proof of Theorem \ref{thm:p_clique}]
Since properties of the clique expansion transfer to those of the line graph via dualization (Proposition \ref{prop:dualization}), we will focus on the map $\mathsf{Q}_q$.

\para{Lipschitz.} First, we will prove that $\mathsf{Q}_q$ is Lipschitz when $1\leq q \leq p \leq \infty$. We will proceed by establishing some estimates, then combining them at the end to get the desired inequality.

Let $f:X \times X \times Y \to \Rspace$ be the function $f(x_1,x_2,y) = \min_j \omega(x_j,y)$ and, likewise, let $f'(x_1,x_2,y) = \min_j \omega'(x_j,y)$. We claim that
\begin{equation}\label{eqn:lipschitz_1}
    \|\omega_{\mathsf{Q}_q} - \omega_{\mathsf{Q}_q}'\|_{L^p(\mu \otimes \mu)} \leq \|f - f'\|_{L^p(\mu \otimes \mu \otimes \nu)}.
\end{equation}
Indeed, when $p < \infty$, we have 
\begin{align}
    &\|\omega_{\mathsf{Q}_q} - \omega'_{\mathsf{Q}_q}\|_{L^p(\mu \otimes \mu)}^p  \nonumber\\
    &\quad = \int_X \int_X \lvert \|\min_j \{\omega(x_j,\cdot)\}\|_{L^q(\nu)} - \|\min_j \{\omega'(x_j,\cdot)\}\|_{L^q(\nu)} \rvert^p \; \mu(dx_1)\mu(dx_2) \nonumber \\
    &\quad \leq \int_X \int_X \|\min_j \{\omega(x_j,\cdot)\} - \min_j \{\omega'(x_j,\cdot)\} \|_{L^q(\nu)}^p \; \mu(dx_1) \mu(dx_2) \label{eqn:lipschitz_2} \\
    &\quad \leq \int_X \int_X \|\min_j \{\omega(x_j,\cdot)\} - \min_j \{\omega'(x_j,\cdot)\} \|_{L^p(\nu)}^p \; \mu(dx_1) \mu(dx_2) \label{eqn:lipschitz_3} \\
    &\quad = \int_X \int_X \int_Y \lvert \min_j \{\omega(x_j,y)\} - \min_j \{\omega'(x_j,y)\} \rvert^p \; \nu(dy) \mu(dx_1)\mu(dx_2) \nonumber \\
    &\quad = \|f - f'\|_{L^p(\mu \otimes \mu \otimes \nu)}^p, \nonumber
\end{align}
where \eqref{eqn:lipschitz_2} follows by the reverse triangle inequality form of Minkowski's inequality, \eqref{eqn:lipschitz_3} follows from H\"{o}lder's inequality. The $p=\infty$ case follows by a limiting argument.

Now define three new functions:
\begin{align*}
    f_j:X \times X \times Y \to \Rspace &: (x_1,x_2,y) \mapsto \omega(x_j,y), \quad  \mbox{for $j \in \{1,2\}$, and} \\
    g:X \times X \times Y \to \Rspace &: (x_1,x_2,y) \mapsto \lvert \omega(x_1,y) - \omega(x_2,y) \rvert.
\end{align*}
Analogously define $f_1',f_2',g'$ by replacing all instances of $\omega$ with $\omega'$. Our next goal is to establish the following:
\begin{align}
    \|f_j - f_j'\|_{L^p(\mu \otimes \mu \otimes \nu)} &= \| \omega - \omega'\|_{L^p(\mu \otimes \nu)} \label{lipschitz_4} \\
    \|g' - g\|_{L^p(\mu \otimes \mu \otimes \nu)} &\leq 2\| \omega - \omega'\|_{L^p(\mu \otimes \nu)}. \label{lipschitz_6}
\end{align}
The equality \eqref{lipschitz_4} (for, say, $j=1$) holds because the functions $f_1$ and $f_1'$ do not depend on $x_2$, so we marginalize over the second copy of $\mu$. To see \eqref{lipschitz_6} for $p < \infty$,
\begin{align}
    &\| g' - g\|_{L^p(\mu \otimes \mu \otimes \nu)}^p \nonumber \\
    &\quad = \int_Y \int_X \int_X \lvert \lvert \omega'(x_1,y) - \omega'(x_2,y) \rvert - \lvert \omega(x_1,y) - \omega(x_2,y) \rvert \rvert^p  \mu(dx_1)\mu(dx_2)\nu(dy) \nonumber \\
    &\quad\leq \int_Y \int_X \int_X \lvert (\omega'(x_1,y) - \omega(x_1,y)) + (\omega(x_2,y) - \omega'(x_2,y)) \rvert^p  \mu(dx_1)\mu(dx_2)\nu(dy),
\end{align}
where we applied the reverse triangle inequality $\lvert \lvert a \rvert  - \lvert b \rvert \rvert \leq \lvert a - b \rvert$ for real numbers $a,b$ and rearranged terms. Applying the triangle inequality for the $L^p$ norm and then marginalizing, we obtain
\begin{align}
&\| g' - g\|_{L^p(\mu \otimes \mu \otimes \nu)} \nonumber \\
    &\quad\leq \left(\int_Y \int_X \lvert \omega'(x_1,y) - \omega(x_1,y) \rvert^p \mu(x_1)\nu(dy) \right)^{1/p} \nonumber \\
    &\qquad \qquad \qquad \qquad + \left(\int_Y \int_X \lvert \omega(x_2,y) - \omega'(x_2,y) \rvert^p \mu(x_2)\nu(dy) \right)^{1/p} \nonumber \\
    &\quad= \| \omega - \omega'\|_{L^p(\mu \otimes \nu)} + \| \omega - \omega'\|_{L^p(\mu \otimes \nu)} \label{eqn:lipschitz_9}.
\end{align}
The $p=\infty$ case follows similarly.

Recall the general formula $\min\{a,b\} = \frac{1}{2}(a + b - \lvert a - b \rvert)$ for real numbers $a$ and $b$. By the definitions of our functions, we therefore have $f = \frac{1}{2}(f_1 + f_2  - g)$ and $f' = \frac{1}{2}(f_1' + f_2' - g')$. We are now prepared to complete the proof:
\begin{align}
    &d_{\mathcal{N},p}(\mathsf{Q}_q(H),\mathsf{Q}_q(H')) \nonumber \\
    &\quad \leq \|\omega_{\mathsf{Q}_q} - \omega_{\mathsf{Q}_q}'\|_{L^p(\mu \otimes \mu)} \label{eqn:lipschitz_10} \\
    &\quad  \leq \|f - f'\|_{L^p(\mu \otimes \mu \otimes \nu)} \label{eqn:lipschitz_11} \\
    &\quad = \frac{1}{2}\|f_1 + f_2 - g - f_1' - f_2' + g'\|_{L^p(\mu \otimes \mu \otimes \nu)} \label{eqn:lipschitz_12} \\
    &\quad \leq \frac{1}{2}\left(\|f_1 - f_1'\|_{L^p(\mu \otimes \mu \otimes \nu)} + \|f_2 - f_2'\|_{L^p(\mu \otimes \mu \otimes \nu)} + \|g' - g\|_{L^p(\mu \otimes \mu \otimes \nu)}\right) \label{eqn:lipschitz_13} \\
    &\quad \leq 2 \| \omega - \omega'\|_{L^p(\mu \otimes \nu)} \label{eqn:lipschitz_14} \\
    &\quad = 2 d_{\mathcal{H},p}(H,H'), \nonumber
\end{align}
where the steps are justified as follows. Inequality \eqref{eqn:lipschitz_10} is Lemma \ref{lem:norm_lower_bound}, \eqref{eqn:lipschitz_11} comes from our estimate \eqref{eqn:lipschitz_1}, \eqref{eqn:lipschitz_12} follows from the formula for $\min\{a,b\}$ discussed above, \eqref{eqn:lipschitz_13} is the triangle inequality for the $L^p$ norm, and \eqref{eqn:lipschitz_14} follows from our estimates \eqref{lipschitz_4} and \eqref{lipschitz_6}. This shows that $\mathsf{Q}_q$ is 2-Lipschitz.

\para{Functorial.} We now demonstrate functoriality, beginning with expansive maps in the $q=\infty$ case. Let $(\phi,\psi)$ be an expansive map from $H = (X,\mu,Y,\nu,\omega)$ to $H' = (X',\mu',Y',\nu',\omega')$. We will show that $\phi:X \to X'$ defines an expansive map from $\mathsf{Q}_\infty(N)$ to $\mathsf{Q}_\infty(N')$. Indeed,
\begin{align*}
\omega_{\mathsf{Q}_\infty}(x_1,x_2) &= \| \min_j \omega(x_j,\cdot) \|_{L^\infty(\nu)} \leq \| \min_j \omega(\phi(x_j),\psi(\cdot)) \|_{L^\infty(\nu)} \\
&\leq \| \min_j \omega(\phi(x_j),\cdot) \|_{L^\infty(\nu')} = \omega_{\mathsf{Q}_\infty}'(\phi(x_1),\phi(x_2)),
\end{align*}
where the second inequality follows because the essential supremum over $\nu'$ is greater than or equal to the essential supremum over the pushforward of $\nu$ by $\psi$, which may not have full support. 

Finally, we show functoriality for $q \in [1,\infty]$ with respect to expansive couplings. Assume that $H,H'$ are compact measure hypernetworks and let $\pi \in \mathcal{C}(\mu,\mu')$ and $\xi \in \mathcal{C}(\nu,\nu')$ be expansive couplings of $H$ and $H'$. We will show that $\pi$ is an expansive coupling for $\mathsf{Q}_q(H)$ and $\mathsf{Q}_q(H')$; that is, $\mathsf{Q}_q(\pi,\xi) = \pi$. Before doing so, we easily observe that $\mathsf{Q}_q$ respects compositions, due to its forgetful structure: for $\pi' \in \mathcal{C}(\mu',\mu'')$ and $\xi' \in \mathcal{C}(\nu',\nu'')$, 
\[
\mathsf{Q}_q(\pi' \bullet \pi, \xi' \bullet \xi)  = \pi' \bullet \pi = \mathsf{Q}_q(\pi',\xi') \bullet \mathsf{Q}_q(\pi,\xi).
\]

For $\pi \otimes \pi \otimes \xi$-almost every $(x_1,x_1',x_2,x_2',y,y')$, it holds that $\omega(x_j,y) \leq \omega'(x_j',y')$ for both $j = 1,2$, and it follows that $\min_j \omega(x_j,y) \leq \min_j \omega'(x_j',y')$ holds $\pi \otimes \pi \otimes \xi$-almost everywhere. Therefore, for $\pi \otimes \pi$-almost every pair $(x_1,x_1')$ and $(x_2,x_2')$, 
\begin{align}
\omega_{\mathsf{Q}_q}(x_1,x_2) &= \| \min_j \omega(x_j,\cdot) \|_{L^q(\nu)} = \| \min_j \omega(x_j,\cdot) \|_{L^q(\xi)} \label{eqn:functorial_0} \\
    &\leq \| \min_j \omega'(x_j',\cdot) \|_{L^q(\xi)} = \| \min_j \omega'(x_j',\cdot) \|_{L^q(\nu)} = \omega_{\mathsf{Q}_q}'(x_1',x_2'). \nonumber 
\end{align}
In \eqref{eqn:functorial_0}, we consider the argument of the norm on the right as a function $Y \times Y' \to \Rspace$ given by $(y,y') \mapsto \min_j \omega(x_j,y)$. In particular, this function does not depend on $y'$, so the second equality follows by marginalizing. The remaining equalities follow by similar arguments, and this completes the proof.
\end{proof}

\begin{example}(Matrix Formulation of Graphifications)\label{ex:matrix_formulation}
Let $(X,Y)$ be a (combinatorial) hypergraph. In the finite setting, we can represent the binary incidence function as a matrix $\omega$ of size $\lvert X \rvert \times \lvert Y \rvert$, valued in $\{0,1\}$. One may observe that the $1$-clique expansion (respectively, $1$-line graph) considered above is (up to a constant factor, when all node and hyperedge probability distributions are uniform) equal to the matrix product $\omega \omega^T$ (respectively, $\omega^T \omega$), with the result being interpreted as the weighted adjacency matrix of a graph. Here, superscript $T$ denotes matrix transpose. For these maps to be well-defined in general on weak isomorphism classes (as in Lemma \ref{lem:well_defined}), one must involve probability measures. For example, we define the \emph{matrix product line graph map} to be $H = (X,\mu,Y,\nu,\omega) \mapsto \mathsf{L}_{\mathrm{mp}}(H) := (Y,\nu,\omega^T \mathrm{diag}(\mu) \omega)$, where $\mathrm{diag}(\mu)$ is the diagonal matrix with entries coming from $\mu$. 

The matrix formulations described above are computationally convenient, and have some nice properties. For example, one can check that the matrix product line graph map is indeed well-defined on weak isomorphism classes of finite measure hypernetworks. Moreover, it is functorial with respect to expansive couplings. However, the matrix formulations are no longer Lipschitz when extended to general hypernetwork functions. This is illustrated by the following example.

For $\alpha \geq 1$, set $H_\alpha = (X,\mu,Y,\nu,\omega_\alpha)$, where $\lvert X \rvert = \lvert Y \rvert = 2$, $\mu,\nu$ are uniform and $\omega_\alpha$ is described by the matrix $\omega_\alpha = \mathrm{diag}(\alpha,\alpha)$. The matrix product line graph map takes $H_\alpha$ to $\mathsf{L}_{\mathrm{mp}}(H_\alpha) = (Y,\nu,\mathrm{diag}(\alpha^2/2,\alpha^2/2))$. Now consider $d_{\mathcal{H},p}(H_\alpha,H_1)$ and $d_{\mathcal{N},p}(\mathsf{L}_{\mathrm{mp}}(H_\alpha),\mathsf{L}_{\mathrm{mp}}(H_1))$ for $\alpha \geq 1$ and $p \in [1,\infty)$.  One can show that $d_{\mathcal{H},p}(H_\alpha,H_1)$ is realized by optimal couplings $\pi = \xi = \mathrm{diag}(1/2,1/2)$ and that this coupling is also optimal for $d_{\mathcal{N},p}(\mathsf{L}_{\mathrm{mp}}(H_\alpha),\mathsf{L}_{\mathrm{mp}}(H_1))$. It follows by direct computation that 
\[
d_{\mathcal{H},p}(H_\alpha,H_1) = \frac{1}{2^{1/p}}\left(\alpha - 1\right)  \mbox{ and }  d_{\mathcal{N},p}(\mathsf{L}_{\mathrm{mp}}(H_\alpha),\mathsf{L}_{\mathrm{mp}}(H_1)) = \frac{1}{2^{1/p}}\left(\frac{\alpha^2}{2} - \frac{1}{2}\right).
\]
Taking $\alpha$ arbitrarily large shows that the matrix product line graph map $\mathsf{L}_{\mathrm{mp}}:\mathcal{H} \to \mathcal{N}$ is not Lipschitz. 

This example shows that one must take care when designing graphifications in order to ensure the desired properties.
\end{example}

We end this section with a structural result on graphifications showing (roughly) that $\mathsf{Q}_\infty$ is minimal among graphifications satisfying mild assumptions. A graphification $\mathsf{F}:\mathcal{H} \to \mathcal{N}$ is called:
\begin{itemize}
    \item a \emph{node-type functor} if it is of the form $\mathsf{F}(X,\mu,Y,\nu,\omega) = (X,\mu,\omega_\mathsf{F})$ and is functorial with respect to expansive maps, i.e., $\mathsf{F}:\mathcal{H}_\mathrm{m} \to \mathcal{N}_\mathrm{m}$;
    \item \emph{non-trivial} if there exists $H \in \mathcal{H}$ such that the network function of $\mathsf{F}(H)$ is not identically zero on a full measure set;
    \item \emph{well-defined on weak isomorphism classes} if whenever $H$ and $H'$ are weakly isomorphic as measure hypernetworks, $\mathsf{F}(H)$ and $\mathsf{F}(H')$ are weakly isomorphic as measure networks; and 
    \item \emph{scale-equivariant} if it satisfies the following condition for all $H = (X,\mu,Y,\nu,\omega)$. Let $\beta \geq 0$ and define $\beta \cdot H = (X,\mu,Y,\nu,\beta \cdot \omega)$; i.e., a rescaling of the hypernetwork function of $H$ by $\beta$. Similarly, for $N \in \mathcal{N}$, let $\beta \cdot N$ denote a rescaling of the network function of $N$ by $\beta$. Scale-equivariance then requires $\mathsf{F}(\beta \cdot H) = \beta \cdot \mathsf{F}(H)$.
\end{itemize}

\begin{proposition}\label{prop:structural_functor}
If $\mathsf{F}:\mathcal{H}_\mathrm{m} \to \mathcal{N}_\mathrm{m}$ is a node-type functor which is non-trivial, well-defined on weak isomorphism classes and scale-equivariant, then there exists a constant $c > 0$ such that for all $H = (X,\mu,Y,\nu,\omega) \in \mathcal{H}$ and for every $x_1,x_2 \in X$,
\begin{equation}\label{eqn:structural_bound}
\omega_\mathsf{F}(x_1,x_2) \geq c \cdot \omega_{\mathsf{Q}_\infty}(x_1,x_2).
\end{equation}
\end{proposition}

\begin{remark}
The Proposition can be stated in a more categorical language. First, various conditions on $\mathsf{F}$ can be stated diagramatically. Let $\mathcal{M}$ denote the category whose objects are pairs $(X,\mu)$ consisting of a Polish space $X$ and a Borel probability measure $\mu$, whose morphisms are functions between underlying sets. Let $\mathsf{M}_\mathcal{N}:\mathcal{N}_\mathrm{m} \to \mathcal{M}$ be the forgetful functor $(X,\mu,\omega) \mapsto (X,\mu)$; likewise, let $\mathsf{M}_\mathcal{H}:\mathcal{H}_\mathrm{m} \to \mathcal{M}$ be the forgetful functor $(X,\mu,Y,\nu,\omega) \mapsto (X,\mu)$ (taking an expansive map $(\phi,\psi)$ to the function $\phi$). The property of being a node-type functor is that the following diagram commutes.
\[\begin{CD} 
\mathcal{H}_\mathrm{m} @>\mathsf{F}>> \mathcal{N}_\mathrm{m}
\\ @V \mathsf{M}_\mathcal{H} VV @VV \mathsf{M}_\mathcal{N} V 
\\ \mathcal{M} @>=>> \mathcal{M}
\end{CD}\]
The property of being well-defined on weak isomorphism classes says that $\mathsf{F}$ is a functor between categories with weak equivalences; see \autoref{sec:distortion_zero_couplings_as_weak_equivalences} below. To restate scale-equivariance more categorically, we introduce the rescaling functor $\mathsf{T}_\mathcal{N}^\beta:\mathcal{N}_\mathrm{m} \to \mathcal{N}_\mathrm{m}$, $\beta \geq 0$, defined on objects by $\mathsf{T}_\mathcal{N}^\beta(N) = \beta \cdot N$ (as above) and defined to be the identity on morphisms. We similarly define a rescaling functor for hypernetworks $\mathsf{T}_\mathcal{H}^\beta: \mathcal{H}_\mathrm{m} \to \mathcal{H}_\mathrm{m}$. Then the scale equivariance property of $\mathsf{F}$ says that the following diagram commutes for all $\beta \geq 0$.
\[\begin{CD} 
\mathcal{H}_\mathrm{m} @>\mathsf{F}>> \mathcal{N}_\mathrm{m}
\\ @V \mathsf{T}_\mathcal{H}^\beta VV @VV \mathsf{T}_\mathcal{N}^\beta V 
\\ \mathcal{H}_\mathrm{m} @>\mathsf{F}>> \mathcal{N}_\mathrm{m}
\end{CD}\]
The nontriviality of $\mathsf{F}$ says that $\mathsf{F}$ does not factor as $\mathsf{T}_\mathcal{N}^0 \circ \mathsf{G}$ for some functor $\mathsf{G}$. Finally, condition \eqref{eqn:structural_bound} says that there is a natural transformation $\mathsf{T}^c_\mathcal{N} \circ \mathsf{Q}_\infty \Rightarrow \mathsf{F}$ whose components are identity morphisms.
\end{remark}

\begin{proof}[Proof of Proposition \ref{prop:structural_functor}]
For $\alpha \geq 0$, let $H_\alpha$ be the measure hypernetwork with one node $x^\alpha$, one hyperedge $y^\alpha$ and with hypernetwork function $\omega_\alpha(x^\alpha,y^\alpha) = \alpha$. Let $(\omega_1)_\mathsf{F}(x^1,x^1) = c$ for some $c \geq 0$. Our first goal is to show that $c > 0$.

By scale equivariance, we have $(\omega_\alpha)_\mathsf{F}(x^\alpha,x^\alpha) = c \cdot \alpha$. Now let $H = (X,\mu,Y,\nu,\omega) \in \mathcal{H}$ be arbitrary. Since $\omega$ is assumed to be bounded, there exists $\alpha \geq 0$ such that $\omega(x,y) < \alpha$ for every $x,y$. The unique functions $X \to \{x^\alpha\}$ and $Y \to \{y^\alpha\}$ define an expansive map from $H$ to $H_\alpha$, so functoriality implies that there is an expansive map from $\mathsf{F}(H)$ to $\mathsf{F}(H_\alpha)$. In particular, $\omega_\mathsf{F}(x_1,x_2) \leq (\omega_\alpha)_\mathsf{F}(x^\alpha,x^\alpha) = c \cdot \alpha$ for all $x_1,x_2 \in X$. It follows that $c > 0$, by non-triviality of $\mathsf{F}$.

Now let $H$ be an arbitrary measure hypernetwork, let $x_1,x_2 \in X$, $y \in Y$, and set $\beta = \min_j \omega(x_j,y)$. Let $\overline{H}_\beta$ denote the measure hypernetwork with nodes $\{x_1^\beta,x_2^\beta\}$, hyperedge set $\{y^\beta\}$, measures uniform and $\overline{\omega}_\beta(x^\beta_j,y^\beta) = \beta$ for $j = 1,2$. Observe that $\overline{H}_\beta$ is weakly isomorphic to $H_\beta$. There is an expansive map from $\overline{H}_\beta$ to $H$ given by $x_j^\beta \mapsto x_j$ and $y^\beta \mapsto y$. This induces an expansive map from $\mathsf{F}(\overline{H}_\beta)$ to $\mathsf{F}(H)$ so that for every choice of $x_1,x_2$ we have
\[
\omega_\mathsf{F}(x_1,x_2) \geq (\overline{\omega}_\beta)_\mathsf{F}(x_1^\beta,x_2^\beta) = (\omega_\beta)_\mathsf{F}(x^\beta,x^\beta) = c \cdot \beta = c \cdot \min_j \omega(x_j,y).
\]
Since this holds for any $y \in Y$, we have established \eqref{eqn:structural_bound}.
\end{proof}

Since $q$-clique expansions are non-trivial, well-defined on weak isomorphism classes and scale-equivariant, and since it is possible to construct $H$ containing nodes $x_1,x_2$ such that $\omega_{\mathsf{Q}_q}(x_1,x_2) < \omega_{\mathsf{Q}_\infty}(x_1,x_2)$, with an arbitrarily large gap, we have the following immediate corollary.

\begin{corollary}\label{cor:clique_not_functorial}
If $q < \infty$, then $\mathsf{Q}_q$ is not functorial with respect to expansive maps.
\end{corollary}

\section{Further Category-Theoretic Results}

This section collects some additional results on the categories of measure networks and hypernetworks, some of which involve more technical category-theoretic terminology and fall outside the discussion of graphification above. Our main reference for this terminology is \cite{borceux1994handbook}.

\subsection{Categorical and Metric Isomorphisms}\label{sec:isomorphisms}

Our categorical constructions for (hyper)networks raise a natural question: do isomorphisms in our categories correspond to metric notions of isomorphism introduced in Section \ref{sec:theory}? We address this question in this subsection.

Borrowing terminology from metric geometry, we say that two measure measure networks $N,N'$ are \emph{isometric} if there exists a (not necessarily measure-preserving) bijection $\phi:X \to X'$ such that $\omega'(\phi(x),\phi(y)) = \omega(x,y)$ for every $x,y \in X$. We call such a map an \emph{isometry}. Similarly, two measure hypernetworks $H,H'$ are \emph{isometric} if there exist $\phi:X \to X'$ and $\psi:Y \to Y'$ such that $\omega'(\phi(x),\psi(y)) = \omega(x,y)$ for all $x \in X$ and $y \in Y$.

In a category $\mathcal{N}_\mathrm{m}$ or $\mathcal{N}_\mathrm{c}$, we say that two objects $N,N'$ are \emph{categorically isomorphic} if there exist morphisms $N \to N'$ and $N' \to N$ such that each composition is equal to the identity. We define hypernetworks $H,H'$ in $\mathcal{H}_\mathrm{m}$ or $\mathcal{H}_\mathrm{c}$ to be \emph{categorically isomorphic} in a similar fashion. The following proposition relates categorical isomorphisms with our notions of metric isomorphisms.

\begin{proposition}\label{prop:isomorphism_equivalences}
We have the following equivalences.
    \begin{enumerate}
    \item Measure networks $N,N' \in \mathcal{N}_\mathrm{m}$ are categorically isomorphic if and only if they are isometric. Likewise, measure hypernetworks $H,H' \in \mathcal{H}_\mathrm{m}$ are categorically isomorphic if and only if they are isometric.
    \item Measure networks $N,N' \in \mathcal{N}_\mathrm{c}$ are categorically isomorphic if and only if they are strongly isomorphic. In particular, any categorical isomorphism $\pi$ from $N$ to $N'$ is of the form $\pi = (\mathrm{id}_X \times \phi)_\# \mu$ for some strong isomorphism $\phi$. Likewise, any categorical isomorphism $(\pi,\xi)$ from $H$ to $H'$ in $\mathcal{H}_\mathrm{c}$ is induced by a strong isomorphism $(\phi,\psi)$.
    \end{enumerate}
\end{proposition}

\begin{proof}
We prove the statements in the setting of measure networks, with the proofs for measure hypernetworks following by similar arguments.

For Part 1, suppose that $N \xrightarrow{\phi} N'$ and $N' \xrightarrow{\phi'} N$ are categorical isomorphisms with $\phi' \circ \phi = \mathrm{id}_X$ and $\phi \circ \phi' = \mathrm{id}_{X'}$. Then $\phi' = \phi^{-1}$ at a set level, and $\phi$ must be a bijection. Moreover, for every $x,y \in X$, 
\begin{equation}\label{eqn:isometry_equivalences}
\omega(x,y) \leq \omega'(\phi(x),\phi(y)) \leq \omega(\phi'(\phi(x)),\phi'(\phi (y)))=\omega(x,y),
\end{equation}
so that all inequalities are forced to be equalities. This shows that $\phi$ is an isometry, and the converse statement is trivial.

To prove Part 2, first suppose that $N \xrightarrow{\pi} N'$ and $N' \xrightarrow{\pi'} N$ are categorical isomorphisms with $\pi' \bullet \pi = \mathbbm{1}_N$ and $\pi \bullet \pi' = \mathbbm{1}_{N'}$. We first observe that
\begin{equation}\label{eqn:support_condition}
(x,x') \in \mathrm{supp}(\pi) \mbox{ and } (x',y) \in \mathrm{supp}(\pi') \quad \mbox{implies} \quad (x,x',y) \in \mathrm{supp}(\pi' \boxtimes \pi);
\end{equation}
this follows by the defining equation 
$
\pi' \boxtimes \pi (dx \times dx' \times dy) = \pi'_{x'}(dx) \pi_{x'}(dy) \mu'(dx').
$
From this, we deduce that if $(x,x') \in \mathrm{supp}(\pi)$, then $(x',x) \in \mathrm{supp}(\pi')$. Indeed, \eqref{eqn:support_condition} applied to $(x,x'),(x',x)$ tells us that $(x,x',x) \in \mathrm{supp}(\pi' \boxtimes \pi)$, so that $(x',x) \in \mathrm{supp}(\pi')$, by Part 1 of Lemma \ref{lem:supp1}.

Next, we show that $\pi$ is supported on the graph of a measure-preserving map. To this end, suppose that $(x,x'),(x,y') \in \mathrm{supp}(\pi)$. Then $(x',x) \in \mathrm{supp}(\pi')$ and an application of \eqref{eqn:support_condition} (with the roles of $\pi'$ and $\pi$ reversed) tells us that $(x',y,y') \in \mathrm{supp}(\pi \boxtimes \pi')$, hence that $(x',y') \in \mathrm{supp}(\pi \bullet \pi')$. Since $\pi \bullet \pi' = \mathbbm{1}_{N'}$, it follows that $x' = y'$, so that $\pi$ is supported on a graph of some measure preserving map $\phi:X \to X'$. Applying the above arguments to $\pi'$, we deduce that $\pi' = (\mathrm{id}_{X'} \times \phi')_\# \mu'$ for a measure-preserving expansive map $\phi':X' \to X$. It then follows from Proposition \ref{prop:embedding_graphs} that $\phi' = \phi^{-1}$; indeed, using the functor $\mathsf{P}$ from the proposition, we have 
\[
\mathsf{P}(\mathrm{id}_X) = \mathbbm{1}_N = \pi' \bullet \pi = \mathsf{P}(\phi' \circ \phi),
\]
so faithfulness implies $\phi' \circ \phi = \mathrm{id_X}$, and we similarly have $\phi \circ \phi' = \mathrm{id}_{X'}$. In particular, $\phi$ is a measure-preserving bijection with measure-preserving inverse. Finally, $\phi$ must preserve $\omega$ by an argument similar to the one used in \eqref{eqn:isometry_equivalences}. The converse statement for Part 2 is straightforward, so this completes the proof.
\end{proof}

\subsection{Distortion Zero Couplings as Weak Equivalences}\label{sec:distortion_zero_couplings_as_weak_equivalences}

The notion of weak isomorphism of measure (hyper)networks also has a natural categorical interpretation. Recall that a \emph{category with weak equivalences} is a pair $(\mathcal{C},\mathcal{W})$ consisting of a category $\mathcal{C}$ and a subcategory $\mathcal{W}$ such that $\mathcal{W}$ contains all (categorical) isomorphisms and satisfies the \emph{two-out-of-three property}, meaning that for any morphisms $f,g$ and $ g \circ f$ in $\mathcal{C}$, if two members of the set $\{f,g,g \circ f\}$ are morphisms of $\mathcal{W}$, then the third must be as well. Morphisms of $\mathcal{W}$ are called \emph{weak equivalences}.

Let $\mathcal{N}_{\mathrm{w}}$ denote the wide subcategory of $\mathcal{N}_{\mathrm{c}}$ whose morphisms are expansive couplings $\pi$ with $\mathrm{dis}_\infty^\mathcal{N}(\pi) = 0$. This does define a subcategory, since any two morphisms $\pi \in \mathcal{C}(\mu,\mu')$ and $\pi' \in \mathcal{C}(\mu',\mu'')$ with zero distortion must compose to a coupling with zero distortion---see Remark \ref{rmk:gluing_triangle_inequality} in the Appendix. Similarly, let $\mathcal{H}_\mathrm{w}$ be the wide subcategory of $\mathcal{H}_\mathrm{c}$ whose morphisms $(\pi,\xi)$ satisfy $\mathrm{dis}_\infty^\mathcal{H}(\pi,\xi) = 0$.

\begin{proposition}\label{prop:weak_equivalences}
The pairs $(\mathcal{N}_\mathrm{c},\mathcal{N}_\mathrm{w})$ and $(\mathcal{H}_\mathrm{c},\mathcal{H}_\mathrm{w})$ are categories with weak equivalences.
\end{proposition}

\begin{proof}
We only include details for $(\mathcal{N}_\mathrm{c},\mathcal{N}_\mathrm{w})$, as the measure hypernetworks case is conceptually the same, but requires more complicated notation. 

By Proposition \ref{prop:isomorphism_equivalences}, $\mathcal{N}_\mathrm{w}$ contains all (categorical) isomorphisms: a categorical isomorphism  is induced by a strong isomorphism, and therefore has zero distortion. It remains to show that the two-out-of-three property holds. Let $\pi \in \mathcal{C}(\mu,\mu')$ and $\pi' \in \mathcal{C}(\mu',\mu'')$. If $\pi$ and $\pi'$ are morphisms of $\mathcal{N}_\mathrm{w}$, then so is $\pi' \bullet \pi$, by virtue of the inequality
\[
\mathrm{dis}_\infty^\mathcal{N}(\pi' \bullet \pi) \leq \mathrm{dis}_\infty^{\mathcal{N}}(\pi') + \mathrm{dis}_\infty^\mathcal{N}(\pi)
\]
(see Remark \ref{rmk:gluing_triangle_inequality}). Suppose that $\pi$ and $\pi' \bullet \pi$ are morphisms of $\mathcal{N}_\mathrm{w}$. We wish to show that $\pi'$ has zero distortion, which is to say that
$
\omega(x',y') = \omega(x'',y'')
$
for all $(x',x''), (y',y'') \in \mathrm{supp}(\pi')$. Assuming $(x',x'') \in \mathrm{supp}(\pi')$, there exists $x \in X$ such that $(x,x',x'') \in \mathrm{supp}(\pi' \boxtimes \pi)$, by Lemma \ref{lem:supp1}, Part 2. It follows that $(x,x') \in \mathrm{supp}(\pi)$ and $(x,x'') \in \mathrm{supp}(\pi' \bullet \pi)$, by Part 1 of Lemma \ref{lem:supp1}. We can likewise choose $y \in Y$ such that $(y,y') \in \mathrm{supp}(\pi)$ and $(y,y'') \in \mathrm{supp}(\pi' \bullet \pi)$. By our assumption that $\pi$ and $\pi' \bullet \pi$ both have zero distortion, it follows that
\[
\omega(x',y') = \omega(x,y) = \omega(x'',y'').
\]
The remaining case ($\pi'$ and $\pi' \bullet \pi$ have zero distortion implies $\pi$ does as well) follows similarly.
\end{proof}

Let $(\mathcal{C},\mathcal{W})$ and $(\mathcal{C}',\mathcal{W}')$ be categories with weak equivalences. We say that a functor $\mathsf{F}:\mathcal{C} \to \mathcal{C}'$ \emph{respects weak equivalences} if for any morphism $\phi$ of $\mathcal{W}$, $\mathsf{F}(\phi)$ is a morphism of $\mathcal{W}'$.

\begin{corollary}
The functors $\mathsf{B},\mathsf{Q}_q,\mathsf{L}_q:\mathcal{H}_\mathrm{c} \to \mathcal{N}_\mathrm{c}$ respect weak equivalences $\mathcal{H}_\mathrm{w}$ and $\mathcal{N}_\mathrm{w}$.
\end{corollary}

\begin{proof}
    This follows since the functors were all shown to be Lipschitz. In particular, the zero distortion condition is preserved. 
\end{proof}

\subsection{Limit Constructions}

We end this section by studying the (non)existence of a few standard limit constructions in our categories of (hyper)networks. In particular, we study whether our categories have any of the following so-called \emph{basic limit constructions}: initial object, terminal object, products, coproducts, pullbacks and/or pushouts. 

First, we will show that limit constructions in the category $\mathcal{N}_\mathrm{m}$ of measure networks with expansive maps are closely related to limit constructions in the category of sets. Let $\mathcal{S}$ denote the category of sets and functions and let $\mathsf{F}:\mathcal{N}_\mathrm{m} \to \mathcal{S}$ denote the forgetful functor, taking $N = (X,\mu,\omega)$ to $X$ and an expansive map to its underlying function. For a fixed measure network $N_0 \in \mathcal{N}_\mathrm{m}$, let $\mathsf{H}_{N_0}:\mathcal{N}_\mathrm{m} \to \mathcal{S}$ denote the \emph{hom functor} taking $N$ to $\mathrm{Hom}_\mathrm{m}(N_0,N)$, the set of expansive morphisms from $N_0$ to $N$. For an expansive map $N \xrightarrow[]{\phi} N'$, $\mathsf{H}_{N_0}(\phi): \mathrm{Hom}_\mathrm{m}(N_0,N) \to \mathrm{Hom}_\mathrm{m}(N_0,N')$ is defined by $\mathsf{H}_{N_0}(\phi)(\psi) = \phi \circ \psi$. 

\begin{proposition}\label{prop:representable}
    The forgetful functor $\mathsf{F}:\mathcal{N}_\mathrm{m} \to \mathcal{S}$ is representable. It follows that $\mathsf{F}$ preserves limits.
\end{proposition}

\begin{proof}
We need to show that there is a natural isomorphism of $\mathsf{F}$ with $\mathsf{H}_{N_0}$ for some measure network $N_0$. Let $N_0 = (X_0,\mu_0, \omega_0)$, where $X_0 = \{x_0\}$, $\mu_0(x_0) = 1$ and $\omega_0(x_0,x_0) = 0$. Let $N \xrightarrow[]{\phi} N'$ be an expansive map. We define a natural transformation with component $\eta_N:\mathsf{F}(N) \to \mathsf{H}_{N_0}(N)$ given by $\eta_N(x) = \phi_{x}$, where $\phi_{x}:X_0 \to X$ is defined by $\phi_{x}(x_0) = x$; by our convention that network functions $\omega$ are nonnegative-valued, this is an expansive map. One can check that the requisite diagram
\[\begin{CD} 
X @>\eta_N>> \mathrm{Hom}_\mathrm{m}(N_0,N)
\\ @V \phi VV @VV \mathsf{H}_{N_0}(\phi) V 
\\ X' @>\eta_{N'}>> \mathrm{Hom}_\mathrm{m}(N_0,N') 
\end{CD}\]
commutes. Moreover, the component maps are isomorphisms, so this proves that $\mathsf{H}_{N_0}$ is representable. By \cite[Proposition 2.9.4]{borceux1994handbook}, $\mathsf{F}$ preserves limits.
\end{proof}

\begin{proposition}\label{prop:limits_maps}
    The categories $\mathcal{N}_\mathrm{m}$ and $\mathcal{H}_\mathrm{m}$ have countable nonempty products and coproducts, and pullbacks. The full subcategories $\mathcal{N}_\mathrm{m}^R$ and $\mathcal{H}_\mathrm{m}^R$ containing measure (hyper)networks whose (hyper)network functions are uniformly bounded above by $R \geq 0$ have all of the above, as well as a terminal object.
\end{proposition}

\begin{proof}
We will first focus on the categories $\mathcal{N}_\mathrm{m}$ and $\mathcal{N}_\mathrm{m}^R$, as the hypernetwork constructions are conceptually the same, but with more complicated notation. We summarize the constructions for hypernetworks at the end of the proof.

 Let $\{N_j = (X_j,\mu_j,\omega_j)\}_{j}$ be a countable collection of measure networks in $\mathcal{N}_\mathrm{m}$. A product $\Pi_{j} N_j$ must have underlying set $\Pi_j X_j$ (i.e., the product space), which is still a Polish space (with the product topology). Take the measure on the product to be any Borel probability measure (e.g., the product measure) and the network function $\omega_{\Pi_j N_j}$ to be  defined by
 \[
\omega_{\Pi_j N_j}((x_1,x_2,\ldots),(y_1,y_2,\ldots)) := \inf_j \omega_j(x_j,y_j).
 \]
 Then $\omega_{\Pi_j N_j}$ is measurable: the map $((x_i)_i,(y_i)_i) \mapsto \omega_j(x_j,y_j)$ is measurable for each $j$, so the infimum is measurable by \cite[Proposition 2.7]{folland1999real}. We define maps $\Pi_j N_j \xrightarrow[]{\phi_j} N_j$ to be coordinate projections; by construction of $\omega_{\Pi_j N_j}$, these are expansive maps. One can easily check that this construction has the universal property of a product. 
 
 To construct the coproduct $\bigsqcup_j N_j$, we take the underlying space to be the disjoint union $\bigsqcup_j X_j$, endowed with the disjoint union topology (which is Polish). The measure on the union is chosen to be an arbitrary Borel probability measure and the network function $\omega_{\sqcup_j N_j}$ is defined by
 \[
    \omega_{\sqcup_j N_j}(z,z') = \left\{\begin{array}{cl}
    \omega_\ell(z,z') & \mbox{if $z,z' \in X_\ell$} \\
    0 & \mbox{otherwise.}\end{array}\right.
 \]
It is straightforward to check that $\omega_{\sqcup_j N_j}$ is measurable and that the inclusion maps $i_\ell:X_\ell \to \bigsqcup_j X_j$ are expansive. Once again, it is easy to check that this construction has the required universal property.

A pullback of a diagram $N_1 \xrightarrow[]{\psi_1} N \xleftarrow[]{\psi_2} N_2$ in $\mathcal{N}_\mathrm{m}$ is given by the measure network $\overline{N}$ with underlying set $\overline{X}:=\{(x_1,x_2) \in X_1 \times X_2 \mid \psi_1(x_1) = \psi_2(x_2)\}$, endowed with the natural topology, an arbitrary choice of Borel probability measure $\overline{\mu}$, and network function $\overline{\omega}$ defined as in the product construction; that is,
\[
\overline{\omega}\left((x_1,x_2),(y_1,y_2)\right) = \min_i \omega_i(x_i,y_i).
\]
We define maps $\phi_i: \overline{X} \to X_i$ to be coordinate projections. These maps are expansive, by construction, and it is straightforward to show that the universal property for pullbacks holds. 

Observe that $\mathcal{N}_\mathrm{m}$ has no terminal object: by Proposition \ref{prop:representable}, a terminal object must have a singleton as its underlying set. Let $N$ be such a measure network and suppose that $\omega(x,x) = R \geq 0$ (with $x$ the unique element in the underlying set). Then there is no expansive map from any network whose network function is not bounded above by $R$ into $N$, so $N$ cannot be a terminal object. On the other hand, $\mathcal{N}_\mathrm{m}^R$ has terminal object given by $(\{x\},\delta_x,\omega_R)$, where $\delta_x$ is the Dirac measure and $\omega_R(x,x) = R$.

We now explain how to modify the constructions above to apply to measure hypernetworks. For a countable collection $\{H_j = (X_j,\mu_j,Y_j,\nu_j,\omega_j)\}_j$ of measure hypernetworks in $\mathcal{H}_\mathrm{m}$, a product is given by
\[
\left(\Pi_j X_j, \mu, \Pi_j Y_j, \nu, \omega_{\Pi_j H_j}\right),
\]
where $\Pi_j X_j$ and $\Pi_j Y_j$ are product spaces, $\mu$ and $\nu$ are arbitrary Borel probability measures and $\omega_{\Pi_j H_j}((x_1,x_2,\ldots),(y_1,y_2,\ldots)) = \inf_j \omega_j(x_j,y_j)$. A coproduct is given by 
\[
\left(\bigsqcup_j X_j, \mu, \bigsqcup_j Y_j, \nu, \omega_{\sqcup_j H_j}\right),
\]
where $\bigsqcup_j X_j$ and $\bigsqcup_j Y_j$ are disjoint unions endowed with compact Polish topologies, $\mu$ and $\nu$ are arbitrary Borel probability measures and 
\[
\omega_{\sqcup_j H_j}(w,z) = \left\{\begin{array}{cl}
\omega_\ell(w,z) & \mbox{if $w \in X_\ell$ and $z \in Y_\ell$} \\
0 & \mbox{otherwise.}
\end{array}\right.
\]
The remaining constructions extend similarly.
\end{proof}

\begin{remark}
Suppose that $N_1$ and $N_2$ are measure network representations of combinatorial graphs, where the network functions are adjacency functions. The product network function then satisfies 
\[
\omega_{\Pi_j N_j}((x_1,x_2),(y_1,y_2)) = 1 \Leftrightarrow \omega_1(x_1,y_1) = \omega_2(x_2,y_2) = 1;
\]
that is, the product has an edge if and only if there is an edge between each of the nodes in the corresponding graphs. This means that the categorical product in $\mathcal{N}_\mathrm{m}$ is a generalization of the usual tensor product (or Kronecker product) of combinatorial graphs. Similarly, the categorical coproduct in $\mathcal{N}_\mathrm{m}$ is a generalization of the disjoint union of two graphs.
\end{remark}

\begin{remark}
By Proposition \ref{prop:representable}, the underlying sets of limit constructions should agree with the usual constructions in $\mathcal{S}$. Since the categories considered in Proposition \ref{prop:limits_maps} do not contain objects whose underlying sets are empty, none of the categories has an initial object.
\end{remark}

\begin{remark}
The question of whether the categories in Proposition \ref{prop:limits_maps} admit pushouts is not resolved. By Proposition \ref{prop:representable}, if a pushout of a diagram $N_1 \xleftarrow[]{\psi_1} N \xrightarrow[]{\psi_2} N_2$ in $\mathcal{N}_\mathrm{m}$ was given by a measure network $\widetilde{N}$, then the underlying set would have to be $X_1 \sqcup X_2/\sim$, where $\sim$ is the equivalence relation generated by $\psi_1(z) \sim \psi_2(z)$. However, the choice of topology is not clear: one can construct a diagram $N_1 \xleftarrow[]{\psi_1} N \xrightarrow[]{\psi_2} N_2$ in $\mathcal{N}_\mathrm{m}$ such that the quotient topology on $X_1 \sqcup X_2/\sim$ is not Hausdorff, so cannot be Polish. Assuming for the moment that the quotient topology is Polish, we can proceed with the construction by choosing an arbitrary probability measure $\widetilde{\mu}$ and defining the network function $\widetilde{\omega}$ by 
\[
\widetilde{\omega}([x],[y]) = \max \{\omega_\ell(x_\ell, y_\ell) \mid \exists \, x_\ell \in X_\ell \cap [x] \mbox{ and } y_\ell \in Y_\ell \cap [y], \, \ell =1,2 \},
\]
where we take the convention that $\max \emptyset = 0$. The pushout maps $\phi_i:X_i \to \widetilde{X}$ are given by $\phi_i(x) = [x]$. 

Generally, it appears that pushouts exist in the case that $X_1 \sqcup X_2/\sim$ admits a Polish topology such that the function $\widetilde{\omega}$ is Borel measurable. This is the case, for example, when all networks in the diagram have finite underlying sets. On the other hand, we were not able to produce any explicit counterexample, so the existence of pushouts is an open question.
\end{remark}

On the other hand, the category $\mathcal{N}_\mathrm{c}$ of measure networks with expansive couplings has fewer limit constructions, due to measure-theoretic issues.

\begin{proposition}\label{prop:no_limits}
    The categories $\mathcal{N}_\mathrm{c}$ and $\mathcal{H}_\mathrm{c}$ have initial objects. The full subcategories $\mathcal{N}_\mathrm{c}^R$ and $\mathcal{H}_\mathrm{c}^R$ of compact measure (hyper)networks with (hyper)network functions uniformly bounded by $R \geq 0$ additionally have  terminal objects. However, none of these categories admit products, coproducts, pullbacks or pushouts.
\end{proposition}

We will use a technical lemma.

\begin{lemma}\label{lem:no_limits_lemma}
Let $\overline{N} \xrightarrow[]{\pi} N$ be an expansive coupling of measure networks $\overline{N},N \in \mathcal{N}_\mathrm{c}$ such that there exists an expansive coupling $N \xrightarrow[]{\xi} \overline{N}$ with $\pi \bullet \xi = \mathbbm{1}_N$. Then $\pi$ is induced by a map $\phi:\overline{X} \to X$ (i.e. $\pi = (\mathrm{id}_{\overline{X}} \times \phi)_\# \overline{\mu}$) and 
$
\mathrm{supp}(\xi) = \{(\phi(z),z) \mid z \in \overline{X}\}.
$
\end{lemma}

\begin{proof}
    For the first claim, it suffices to show that for all $z \in \overline{X}$ there is a unique $x \in X$ such that $(z,x) \in \mathrm{supp}(\pi)$, and this follows by a similar argument to the one used in Part 2 of Proposition \ref{prop:isomorphism_equivalences}. Let $\phi$ denote the map which induces $\pi$ and note that $\mathrm{supp}(\pi) = \{(z,\phi(z)) \mid z \in \overline{X}\}$.
    
    We now characterize the support of $\xi$. First, it must contain all points of the form $(\phi(z),z)$: if there exists $z \in \overline{X}$ such that $(\phi(z),z) \not \in \mathrm{supp}(\xi)$, then  $(\phi(z),\phi(z))$ does not lie in the support of $\pi \bullet \xi = \mathbbm{1}_N$, which is a contradiction. Conversely, suppose that $(x,z) \in \mathrm{supp}(\xi)$; then $(x,\phi(z)) \in \mathrm{supp}(\pi \bullet \xi) = \mathrm{supp}(\mathbbm{1}_N)$, and this implies $x = \phi(z)$.
\end{proof}

\begin{proof}[Proof of Proposition \ref{prop:no_limits}]
Similar to proof of Proposition \ref{prop:limits_maps}, we focus on the categories of measure networks. All of the arguments and constructions generalize straightforwardly to the categories of measure hypernetworks.

Let $N_R = (\{x\},\delta_x,\omega_R)$, for $R \geq 0$, be the measure network with $\delta_x$ the Dirac measure and $\omega_R(x,x) = R$. That $N_0$ is an initial object for $\mathcal{N}_\mathrm{c}$ follows because there is a unique coupling of $\delta_x$ and any other probability measure, and such a coupling will always be expansive, by construction of $\omega_0$. The lack of a terminal object for $\mathcal{N}_\mathrm{c}$ follows by an argument similar to the case of $\mathcal{N}_\mathrm{m}$ (Proposition \ref{prop:limits_maps}). The fact that $N_R$ is a terminal object for $\mathcal{N}_\mathrm{c}^R$ follows similarly.

Next, we consider the case of products and show, in particular, that there is no product of a measure network $N$ with itself, provided the cardinality of its underlying set is at least $2$. Suppose that there is such a product; we denote the product as $\overline{N}$ and the projection maps onto the left and right factors as $\overline{N} \xrightarrow[]{\pi_i} N$, $i=1,2$, respectively. Consider the diagram $N \xleftarrow[]{\mathbbm{1}_N} N \xrightarrow[]{\mathbbm{1}_N} N$. By the universal property of the product, there exists a unique $N \xrightarrow[]{\xi} \overline{N}$ such that $\pi_i \bullet \xi = \mathbbm{1}_N$ for $i=1,2$. From Lemma \ref{lem:no_limits_lemma}, we deduce that the projections $\pi_1$ and $\pi_2$ must be equal: they are induced by measure-preserving maps $\phi_1$ and $\phi_2$ and we have 
\[
\mathrm{supp}(\xi) = \{(\phi_1(z),z) \mid z \in \overline{X}\} = \{(\phi_2(z),z) \mid z \in \overline{X}\},
\]
so that the claim follows. To derive a contradiction, consider the diagram $N \xleftarrow[]{\mathbbm{1}_N} N \xrightarrow[]{\mu \otimes \mu} N$. The universal property implies that there exists and expansive coupling $\xi'$ of $\mu$ and $\overline{\mu}$ such that $\pi \bullet \xi' = \mathbbm{1}_N$ and $\pi \bullet \xi' = \mu \otimes \mu$. This is impossible, provided the underlying set of $N$ has at least two points.

The lack of a coproduct of a measure network $N$ with itself (provided $N$ has cardinality at least $2$) is proved similarly, so we only sketch the argument. Suppose that $\overline{N}$ is a coproduct of $N$ with itself, with associated maps $N \xrightarrow[]{\pi_i} \overline{N}$, $i=1,2$ for the left and right factors. Applying the universal property to the diagram $N \xrightarrow[]{\mathbbm{1}_N} N \xleftarrow[]{\mathbbm{1}_N} N$, together with Lemma \ref{lem:no_limits_lemma}, implies that $\pi_1 = \pi_2$. One then derives a contradiction by applying the universal property to the diagram $N \xrightarrow[]{\mathbbm{1}_N} N \xleftarrow[]{\mu \otimes \mu} N$.

To see that $\mathcal{N}_\mathrm{c}$ does not have pullbacks, let $N$ be a measure network whose underlying set has at least 2 points, and suppose that its network function $\omega$ is bounded above by $R \geq 0$. Let $N_R$ be the one-point measure network defined earlier in the proof. Then there is a unique expansive morphism $N \xrightarrow[]{\pi} N_R$. A pullback for the diagram $N \xrightarrow[]{\pi} N_R \xleftarrow[]{\pi} N$ would be a product of $N$ with itself, and we showed above that no such object exists. A similar argument illustrates the lack of pushouts.
\end{proof}

In summary, the similarity in structure between $\mathcal{N}_\mathrm{m}$ and the category of sets $\mathcal{S}$ means that it enjoys many basic limit constructions. However, the ambivalence of expansive maps to topological and measure structure may be seen as undesirable; for example, this causes the issue with functoriality for the $q$-clique expansion and $q$-line graph maps when $q < \infty$, and this category is unable to capture the notion of strong isomorphism. On the other hand, the additional measure-theoretic structure of $\mathcal{N}_\mathrm{c}$ leads to obstructions to many limit constructions. An interesting future direction of research will be to explore explore alternative category structures on measure networks which balance structure preservation and richness.

\section{Computational Examples}
\label{sec:examples}

In this section, we demonstrate our computational framework on various analysis tasks involving  hypergraph matching, comparison, and simplification.  

\subsection{Implementation Details}
\label{sec:implementation}
To compute the distance between two hypergraphs, we use the Python3 implementation of CO-Optimal Transport (COOT) due to Redko \etal~\cite{TitouanRedkoFlamary2020}~\footnote{https://github.com/PythonOT/COOT}. They give an efficient algorithm for efficiently approximating the solution to the optimization problem \eqref{eqn:distance} in the finite case, when $p = 2$. Computation of $d_{\mathcal{H},2}$ is a nonconvex bilinear programming problem \cite{gallo1977bilinear}, and is NP-Hard to solve exactly. According to~\cite{TitouanRedkoFlamary2020}, the overall computational complexity of their algorithm is $O(\min((n+n')dd' + n'^2n, (d+d')nn' + d'^2d))$, where $n = \lvert X\rvert$, $n'=\lvert X'\rvert$, $d=\lvert Y\rvert$ and $d' = \lvert Y'\rvert$. 
The dependencies of the COOT implementation include \emph{Numpy}, \emph{Matplotlib}, and \emph{Python Optimal Transport (POT)}~\cite{flamary2021pot}~\footnote{https://pythonot.github.io/}. Our main computational contributions are new processing algorithms for modeling hypergraphs as measure hypernetworks amenable to the COOT algorithm.

\subsection{Soft Matching Between Hypergraphs}
\label{sec:soft-matching-toy}
 
We first expand on Example~\ref{example:hypergraph-a} and give a refined method for generating a measure hypernetwork from a combinatorial hypergraph. 

Before we begin, we recall various choices available when modeling a hypergraph as a hypernetwork. 
Here we suppose that we are given a hypergraph $(X,Y)$ and are tasked with obtaining a hypernetwork $(X,\mu,Y,\nu,\omega)$.

First we consider $\mu$. Let $x\in X$. Two simple choices for $\mu(x)$ are as follows:

\begin{itemize}\denselist
\item Uniform: $\mu(x) := {1}/{\lvert X \rvert}.$ 
\item Normalized node degree: $\mu(x):= {\mathrm{deg}(x)}/{\sum_{x'\in X}\mathrm{deg}(x')}.$
\end{itemize}

Next we consider $\nu$. Let $y\in Y$. Two simple choices for $\nu(y)$ are as follows:

\begin{itemize}\denselist
\item Uniform: $\nu(y) := {1}/{\lvert Y \rvert}.$ 
\item Normalized sum of node degree: \newline $\nu(y):= {\widehat{\nu}(y)}/{\sum_{y'\in Y}\widehat{\nu}(y')}$, where $\widehat{\nu}(y) := \sum_{x\in y}\mathrm{deg}(x).$
\end{itemize}

The simplest choice for $\omega$ is to use the incidence function $\omega(x,y) = 1$ if $x \in y$ and $0$ otherwise. Describing more interesting choices for $\omega$ requires some setup. Recall the line graph construction in~\autoref{sec:transformations}: a line graph $\mathsf{L}(H)$ (in this section, $\mathsf{L}(H)$ will be used generically as a stand in for several specific line graph constructions) is constructed from a combinatorial hypergraph $H=(X,Y)$ by taking the node set of $\mathsf{L}(H)$ to be $Y$, and adding an edge $\{y,y'\}$ whenever $y\cap y' \neq \emptyset$---this coincides with the $\mathsf{L}_{\infty}$ line graph map defined in~\autoref{sec:clique_expansion}. Moreover, $\mathsf{L}(H)$ may be weighted---our line graph maps $\mathsf{L}_q$ for $q \in [1,\infty)$ define weighted variants of the line graph. 

In the hypergraph modeling task discussed here, the line graph map is only an intermediate step and we have found empirically that different weighting schemes are more useful. To assign a hypernetwork function to $H$, we will use a metric structure on the line graph, where it is more useful to intuitively think of edge weights as conductances rather than impedances---that is, hyperedges that share more nodes are considered as more similar and should thus receive a smaller edge weight. With this in mind, we use edge weights corresponding to \emph{intersection size} ${1}/{\lvert y \cap y' \rvert}$ (if $y\cap y' \neq \emptyset$) or \emph{reciprocal Jaccard index} ${\lvert y \cup y'\rvert}/{\lvert y \cap y'\rvert}$. Then, we may define the hypernetwork function as $\omega(x,y):= \min \{d_{\mathsf{L}(H)}(y',y) : x \in y'\}$, where $d_{\mathsf{L}(H)}$ denotes the shortest path metric with respect to the chosen edge weights. In summary, we consider three choices for $\omega$:
\begin{itemize}\denselist
\item Intersection size: $\omega$ is the shortest path distance based on the intersection size-weighted line graph. 
\item Jaccard index: $\omega$ is the shortest path distance based on the reciprocal Jaccard index-weighted line graph.
\item Incidence: $\omega(x,y) = 1$ if $x \in y$ and $0$ otherwise.  
\end{itemize}

\begin{example}[Hypergraph Matching Toy Examples]
\label{example:matching-toys}
\autoref{fig:matching-toys} shows two toy examples of matching between nodes and hyperedges of hypergraphs, computed simultaneously via the co-optimal transport distance $d_{\mathcal{H},2}$. Hypergraphs are visualized using the hybrid convex hull/incidence graph method of~\autoref{fig:hypergraph-a-b}. Both examples utilize node degree-based probability measures and Jaccard shortest path hypernetwork functions.

\begin{figure}[!ht]
    \centering
    \includegraphics[width=0.99\columnwidth]{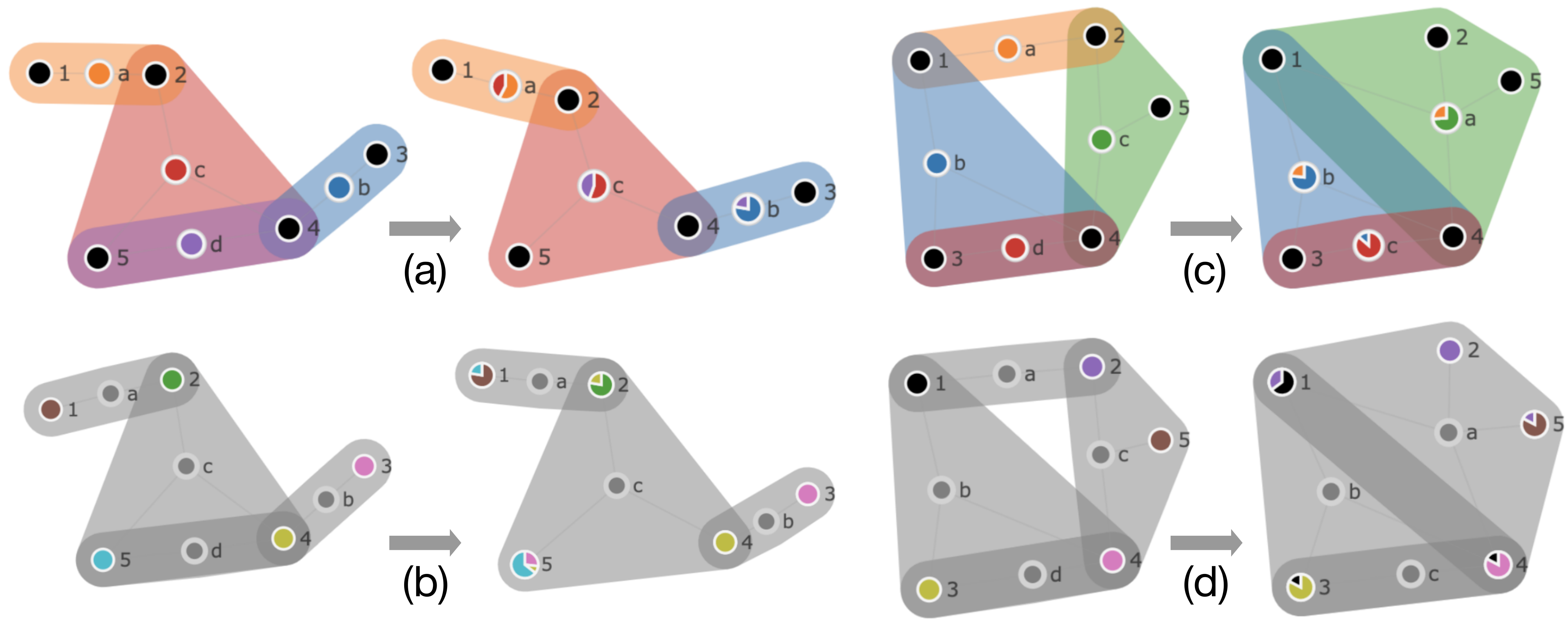}
    \caption{Two toy examples of hypergraph matching.}
    \label{fig:matching-toys}
\end{figure}

The coupling matrices after running the co-optimal transport optimization on the first pair of hypergraphs $H$ and $H'$ are shown below.   
The 1st example in (a)-(b) has optimal coupling matrices $\xi$ for hyperedge coupling (rows indexed by hyperedges of $H$ and columns indexed by hyperedges of $H'$) and $\pi$ for node coupling (rows indexed by nodes of $H$ and columns indexed by nodes of $H'$). 
\[
\xi = 
\begin{bmatrix}
0.158 & 0 & 0\\
0 & 0.211 & 0 \\
0.115 & 0 & 0.254 \\
0 & 0.062 & 0.201
\end{bmatrix} \;
\pi = 
\begin{bmatrix}
0.111 & 0 & 0 & 0 & 0\\
0 & 0.111 & 0 & 0 & 0 \\
0 & 0 & 0.143 & 0 & 0.079 \\
0 & 0.032 & 0 & 0.286 & 0.016 \\
0.032 & 0 & 0 & 0 & 0.190
\end{bmatrix}.
\]

In \autoref{fig:matching-toys}(a), we illustrate the hypergraph matching $H \to H'$, where the visual encodings capture information from the  hyperedge coupling matrix $\xi$.   Colored nodes (with white rim) and colored convex hulls in the hybrid visualization for $H$ represent the four hyperedges of $H$. 
The co-optimal transport results in a color transfer such that colored nodes in $H'$ capture information from the hyperedge coupling matrix $\xi$ via pie charts.  
For instance, the 3rd column of $\xi$ shows that hyperedges $c$ (red) and $d$ (purple) in $H$ are both matched to the hyperedge $c$ in $H'$ with non-zero probabilities.  
Therefore node $c$ in $H'$ is visualized by a pie chart in both red and purple. 
The convex hull associated with hyperedge $c \in H'$, on the other hand, carries the color of the hyperedge $c \in H$ with the highest coupling probability. 
Similarly in \autoref{fig:matching-toys}(b), node coupling between $H$ and $H'$ is visualized via a color transfer and pie charts. For example, both node 2 and node 4 in $H$ are matched with node 2 in $H'$ with non-zero probability, most likely due to the symmetry of these hypergraphs.  \autoref{fig:matching-toys}(c)-(d) show another visualization of a hypergraph matching.
\end{example}

\begin{figure}[!ht]
    \centering
    \includegraphics[width=\columnwidth,trim={15cm 15cm 15cm 15cm},clip]{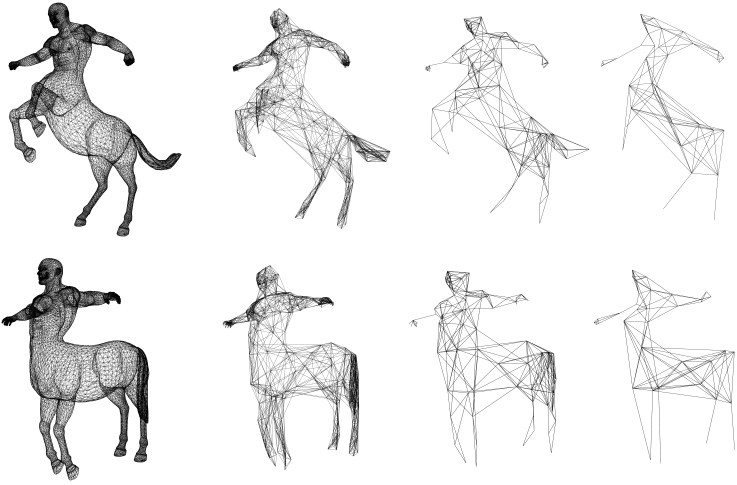}
    \caption{Multiscale reduction via iterated nerve graphs.}
    \label{fig:multiscale-nerve}
\end{figure}

\subsection{Multiscale Graph Matching}
\label{sec:multires-matching}

In this section, we give a use case of hypergraph matching for multiscale graph matching. Multiscale structure has implicitly been shown to be instrumental in providing efficient proxies for Gromov-Wasserstein matching through low-rank approximations \cite{scetbon2022linear}, quantization \cite{chowdhury2021quantized}, and iterated clustering approaches \cite{xu2019scalable, blumberg2020mrec}. While these aforementioned works have focused on demonstrating performance benchmarks against a well-defined problem of matching point clouds or graphs, we focus on demonstrating a new use case toward \emph{summarized} graph matching.  

Our method proceeds by first constructing multiple levels of simplification for two graphs, and then performing graph matching within each level of simplification \emph{while maintaining compatibility across levels of simplification}. Note that unless explicitly enforced, compatibility may be broken by the symmetries of the problem (e.g. a left limb being matched to a left limb at one scale and to a right limb at another scale). Hypergraph matching is specifically used to maintain the compatibility condition. Informally, we are summarizing our graph matching at multiple levels of granularity while maintaining consistency across summaries.

The first stage of our method is inspired by a construction known as the \emph{multiscale mapper} \cite{DeyMemoliWang2016}. We begin by recalling some concepts related to the \emph{mapper}  \cite{singh2007topological} construction. Given a topological space $X$ and a finite open cover $\mathcal{U}=\{U_i\}_{i\in I}$ of $X$, that is $X \subseteq \bigcup_{i \in I}U_i$,  the \emph{nerve} of $\mathcal{U}$ is the simplicial complex $\mathsf{N}(\mc{U})$ with vertices given by the index set $I$ and subsets $\{i_0,\ldots,i_k\} \subseteq I$ forming $k$-simplices whenever $U_{i_0} \cap \ldots \cap U_{i_k} \neq \emptyset$, for any $k\geq 1$. The \emph{nerve graph} is simply the 1-skeleton of this simplicial complex, and it is easier to compute in practice than the nerve complex because it involves checking only pairwise intersections. Next, given topological spaces $X$ and $Y$, a finite open cover $\mc{V}$ of $Y$, and a continuous map $f:X\to Y$, consider a pullback cover of $X$, defined as  $f^{-1}(\mc{V})$. The mapper construction is simply the nerve of such a pullback cover, denoted as $\mathsf{N}(f^{-1}(\mc{V}))$. In practice, one often considers only the nerve graph, and refers to the simplified construction as a \emph{mapper graph}.

The mapper graph provides a representation of data at a \emph{single} scale controlled by the scale of the cover $\mc{V}$. The multiscale mapper was devised as a method for relating mapper graphs constructed at different scales. In this setup, the key idea is that of a map of covers. Let $\mc{U}=\{U_i\}_{i\in I}, \mc{V}=\{V_j\}_{j\in J}$ be two covers of a topological space $Z$. Then a set map $\xi:I \to J$ is a \emph{map of covers} if $U_i \subseteq V_{\xi(i)}$ for all $i \in I$. The multiscale mapper construction utilizes the observation \cite{DeyMemoliWang2016} that a map of covers induces a map of pullback covers, and thus induces a map between mapper graphs at different scales.

\paragraph{Step 1: multiscale reduction}
With the mapper context in place, we now present the first stage of our multiscale graph matching method. Given a graph $G=(V,E)$, we first use a data-driven method (described below) to compute an overlapping cover $\mc{U}=\{U_i\}_{i\in I}$ of $V$ and its nerve graph $G_1$. We then repeat the process on the new nerve graph $G_1$, obtaining a reduced graph at each step of the iteration, $G_1, G_2, \dots$, etc. We terminate the reduction process after the number of nodes in the graph falls below a threshold $n_\alpha$. For a multiscale graph matching problem, we compute these reductions for each graph and pass the collections of nerve graphs into the second stage of the method.

To compute a data-driven cover, we use a procedure involving the graph heat kernel. The heat kernel and related spectral constructions are well-known for capturing shape signatures \cite{reuter2006laplace, memoli2011spectral}, and thus help produce semantically meaningful cover elements. Moreover, heat kernels have been shown to produce superior performance in GW graph matching applications \cite{chowdhury2021generalized}, and thus we are able to recycle computational elements across both stages of our multiscale graph matching method. Our exact procedure is described next.

Fix a graph $G$ with node set $V$. Initialize a set of visited nodes $U=\emptyset$ as well as a cover $Y=\emptyset$. Also fix a node $x\in V \setminus U$. Let $L$ denote the normalized graph Laplacian of $G$. We first compute its eigendecomposition $\Phi\Lambda \Phi^T$. Computing a full eigendecomposition is an $O(n^3)$ operation, so for large graphs we limit runtime by using Lanczos iterations to obtain extremal eigenpairs and using these to construct a low-rank factorization of $L$. In our experiments, we have found the top 300 largest eigenvalues and their eigenvectors to provide good results while keeping runtimes low (e.g. below 1 minute for a graph having over 15K nodes). Next, note that for a given $t>0$, we may compute the graph heat kernel $K^t := \operatorname{exp}(-tL)$ as $K^t = \Phi\operatorname{exp}(-t\Lambda)\Phi^T$. Given a Dirac delta vector $\delta_x$, the matrix-vector product $v:= K^t\delta_x$ can be interpreted as the diffusion of a unit mass of heat out from $x$ within time $t$, and essentially has the form of a Gaussian centered at $x$. Following the idea of full width at half maximum (FWHM), we mark the set $\{v \geq \max(v)/2\}$ as visited, and set $U \leftarrow U \cup \{v \geq \max(v)/2\}$. Additionally we set $Y \leftarrow Y \cup \{\{ v \geq \max(v)/4\}\}.$ We then select another $x \in V\setminus U$ and iterate the procedure until $U = V$. Note that the use of $U$ prevents us from sampling graph nodes too densely, and the extra $1/2$ multiplicative factor in $Y$ allows us to enforce overlaps between cover elements. Also note that the choice of $t$ is difficult to prescribe \emph{a priori}. In our experiments, we found $t:=\log_{10}(\vert V \vert)$ to produce good results. Note from this choice that we utilize different values of $t$ at each reduction step.

In summary, the multiscale reduction step iterates a particular technique of reducing a graph that factors through a process of computing covers and taking nerves, as illustrated in~\autoref{fig:multiscale-nerve}. The reduced graphs and the cover relations are then passed to the hypergraph matching step.

\begin{algorithm}
\caption{Cyclic Block Coordinate Descent for Multiscale Matching}\label{alg:multiscale-matching}
\begin{algorithmic}
\State Initialize $\pi_i,\xi_i$ for all $0\leq i \leq k$, $maxIter \in \N$, $n\gets 0$
\While{$n< maxIter$}
\State // Left-to-right sweep
\For{$i=0,1,\ldots, k$}
    \State // Matrix formulation of co-optimal transport cost
    \State // $\odot$ denotes elementwise square
    \State // $\mathbf{1}_*,\mathbf{1}_\star$ have dimensions needed to match third matrix

    \State $M \gets  (\omega_i^{\odot 2})^T \mu_i \mathbf{1}_*^T +  \mathbf{1}_\star(\mu'_i)^T(\omega'_i)^{\odot 2} - 2 \omega^T \pi_i \omega' $
    
    \State $\xi_i \gets OT(\nu_i,\nu_i',M)$

    \If{$i<k$}
        \State $\pi_{i+1} \gets \xi_i$
    \EndIf
\EndFor
\State // Right-to-left sweep
\For{$i=k,k-1,\ldots, 0$}
    \State // $\mathbf{1}_*,\mathbf{1}_\star$ have new dimensions as needed to match third matrix
    
    \State $M \gets  \omega_i^{\odot 2}\nu_i \mathbf{1}_*^T + \mathbf{1}_\star(\nu'_i)^T((\omega'_i)^{\odot 2})^T - 2 \omega \xi_i (\omega')^T$
    
    \State $\pi_i \gets OT(\mu_i,\mu_i',M)$

    \If{$i>0$}
        \State $\xi_{i-1} \gets \pi_i$
    \EndIf
\EndFor
\State $n \gets n+1$
\EndWhile
\State \Return $\pi_i$ (and optionally $\xi_i$) for all $0\leq i \leq k$
\end{algorithmic}
\end{algorithm}

\paragraph{Step 2: Hypergraph matching}

We now describe the graph matching stage of our method. 
A cover $\mathcal{U}=\{U_i\}_{i\in I}$ of a topological space $X$ defines a hypergraph with nodes given by $X$ and hyperedges given by $\mc{U}$. Thus each cover used in the construction above defines a hypergraph. Consequently, the iterated nerve construction described above provides a sequence of hypergraphs where the hyperedges of one hypergraph comprise the nodes of the next. We will first describe our method for graph matching that utilizes these hypergraphs, and then provide some comparison to alternative methods.

Fix a method for converting a hypergraph to a hypernetwork. Let $G,G'$ be two graphs. We first apply the iterated nerve construction described above for obtaining sequences of hypergraphs, and then the method for passing from hypergraphs to hypernetworks. Let $H_0, H_1, \ldots H_k$ denote the sequence of hypernetworks arising from $G$ (by letting $H_0:=G$), where each $H_i$ represents a tuple $(X_i,\mu_i,Y_i,\nu_i,\omega_i)$. The hyperedges of $H_i$ form the nodes of $H_{i+1}$ for each $0\leq i < k$. Next, let $H'_0,H'_1,\ldots, H'_k$ denote the sequence of hypernetworks arising from $G'$. We then solve a variant of a co-optimal transport problem that: (1) simultaneously returns couplings $\pi_i$ between $\mu_i, \mu'_i$ and $\xi_i$ between $\nu_i,\nu'_i$ for all $i$, and (2) enforces the constraint $\pi_{i+1}=\xi_i$ for all $0\leq i < k$. The $\pi_i$ yield correspondences between graph nodes at multiple scales, and the $\pi_{i+1}=\xi_i$ condition enforces consistency between correspondences at different scales. Whereas standard co-optimal transport \cite{TitouanRedkoFlamary2020} employs a block coordinate descent scheme to alternate between two variables, we employ cyclic block coordinate descent to sweep through $2k$ variables to achieve the described solution. We present our pseudocode in Algorithm \autoref{alg:multiscale-matching}.

We now describe some design considerations that went into our method. An alternative form of multiscale graph matching could be carried out by: (1) iteratively applying any black-box graph simplification method, and (2) applying any black-box graph matching method along each level of simplification. The second step of such a method benefits from being embarrassingly parallel, as the graph matchings across different levels of simplification are independent. As a consequence of this independence, however, such a method is unable to enforce consistency of correspondences across different levels of simplification. In contrast, our method utilizes a careful choice of a graph simplification method that immediately suggests an application of hypergraph matching. While hypergraph matching across different levels of simplification could be solved in a parallel, independent manner, we choose to create dependence by enforcing the $\pi_{i+1}=\xi_i$ condition. Thus we trade off performance (via parallelization) for consistency of matching across scales.

 \begin{figure}[!ht]
    \centering
    \includegraphics[width=\columnwidth,trim={15cm 15cm 15cm 15cm},clip]{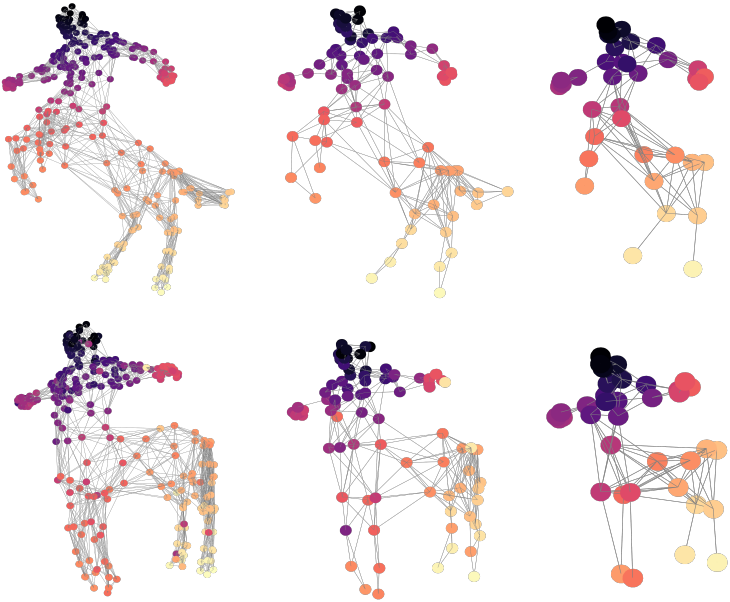}
    \caption{Color transfer via multiscale matching, as described by Algorithm \autoref{alg:multiscale-matching}. Each graph in the top row is displayed with user-defined node colors. For each of these graphs, the colors are transferred to the graph below using the transport matrices returned by Algorithm \autoref{alg:multiscale-matching}. Note that despite a few mistakes (appearing as discontinuities in the color gradient), the transport matrices perform reasonably well in matching semantically related parts.}
    \label{fig:multiscale-matching}
\end{figure}

\begin{example}[multiscale matching on TOSCA graphs]
\label{example:meshes}

We demonstrate the methods described in this section on a pair of Centaur meshes from the TOSCA dataset \cite{bronstein2008numerical}. We discard the mesh faces and retain only the underlying graph structures, which comprise over 15K nodes. In \autoref{fig:multiscale-nerve} we display these graphs as well as the output of the multiscale reduction via iterated nerve graphs. We then perform multiscale matching on these graphs using Algorithm \autoref{alg:multiscale-matching}. We demonstrate the quality of the matching in \autoref{fig:multiscale-matching}. Note that the smaller graphs provide summaries of the overall matching at multiple levels of granularity.

\end{example}

\para{Parameter analysis.}
We use our multiscale matching framework to explore how the choices of parameters ($\mu$, $\nu$, and $\omega$) affect the matching results. In order to have a ground truth to compare to, we permute the labels of nodes in a Centaur graph (from the TOSCA dataset), and perform the hypergraph matching between the original hypergraph and the permuted hypergraph, where the hypergraph structure is obtained via the method described in this subsection. If $(\pi,\xi)$ is a true pair of optimal couplings, then $\pi$ should be induced by the node label permutation; that is, the optimal coupling should match each node in the graph to itself. The node-level matching results with different parameter settings are shown in~\autoref{fig:centaur-self-matching}. For each parameter setting, we compute a putatively optimal node coupling $\pi$ and use it to infer a hard matching from the source to the target by declaring a node $x$ in the source to be matched to a node $y$ in the target if $y = \mathrm{argmax}_{y'} \pi(x,y')$. We then compute the average graph distance between the nodes in the original graph and their matched nodes in the permuted graph---theoretically, this should be zero, but this test measures the effectiveness of gradient descent at finding the true optimal matching under our various parameter regimes. 
We repeat the permutation 100 times.~\autoref{fig:centaur-graph-distances} shows the box plot of the average graph distances for each parameter setting. 
Our results show that defining $\omega$ via Jaccard index or intersection size (c, d, f, and g) performs better than 
 taking $\omega$ to be the incidence matrix (b and e).  
When $\omega$ is Jaccard index or intersection size, the results are fairly comparable, when $\mu$ and $\nu$ are either uniform or proportional to degrees. 
We repeated the above experiments across a few more meshes and obtained similar observations. 
The experimental results in~\autoref{sec:simplify} are generated for $\mu$ and $\nu$ being proportional to degrees, and $\omega$ being the Jaccard index.

 \begin{figure}[!ht]
    \centering
    \includegraphics[width=0.99\columnwidth]{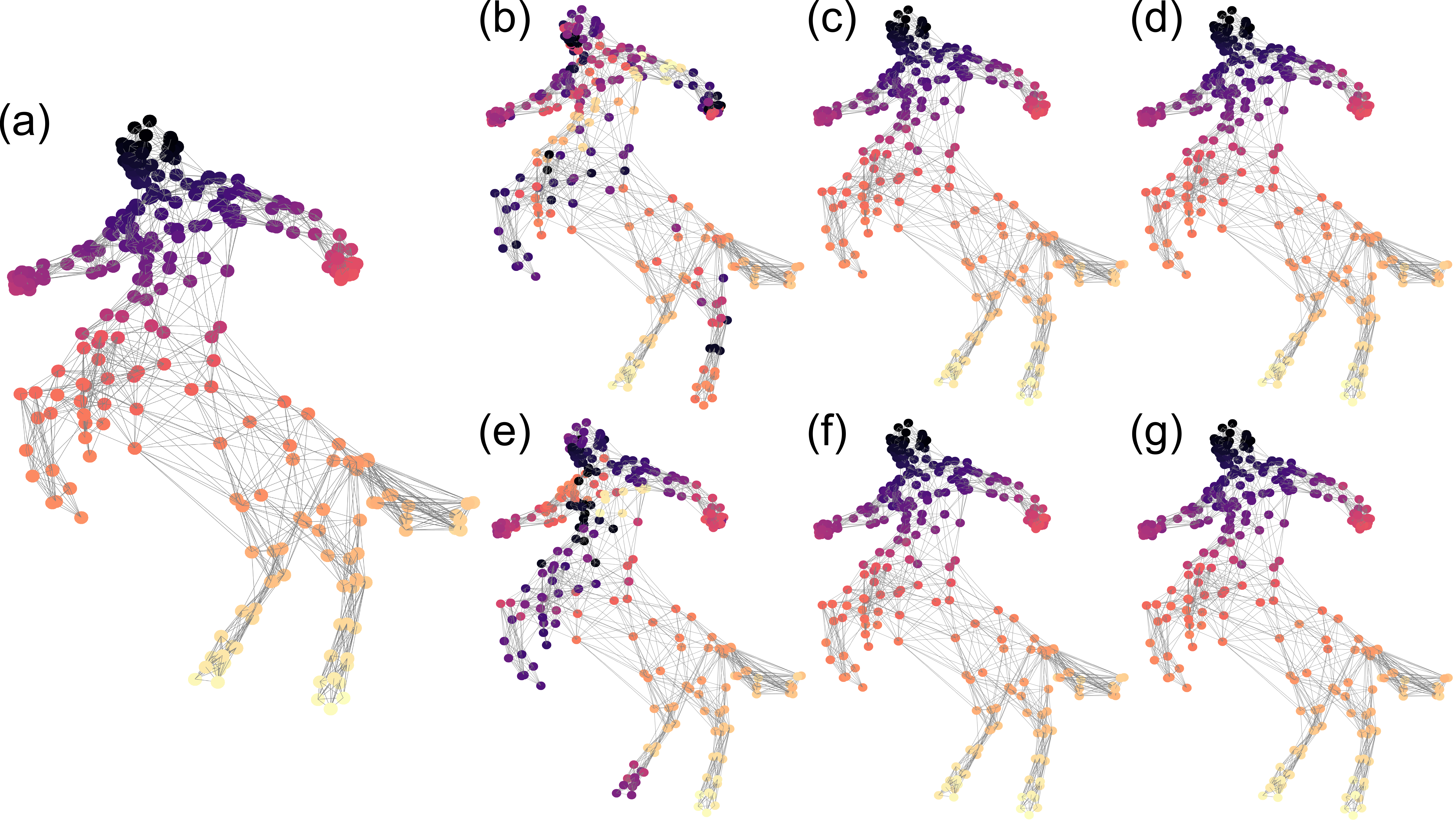}
    \caption{The hypergraph matching between the Centaur graph (a) and a permutation of the Centaur graph  with different parameter settings: (b-d) $\mu$ and $\nu$ being uniform, and $\omega$ being incidence, Jaccard index, and intersection size, respectively; (e-g) $\mu$ being normalized node degree, $\nu$ being normalized sum of node degrees, and $\omega$ being incidence, Jaccard index, and intersection size, respectively.}
    \label{fig:centaur-self-matching}
\end{figure}

\begin{figure}[!ht]
    \centering
    \includegraphics[width=0.99\columnwidth]{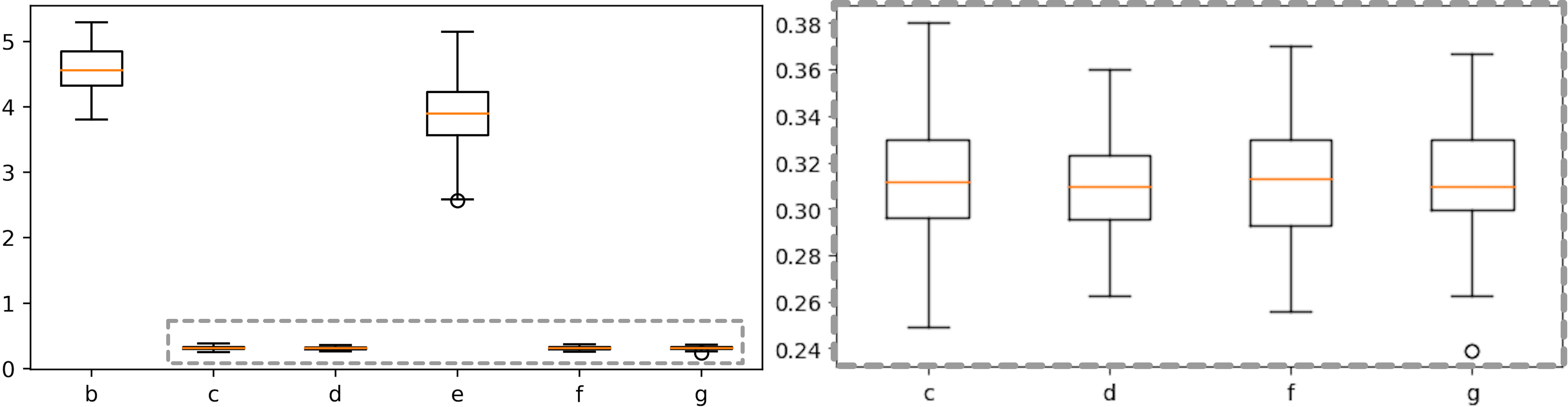}
    \caption{The distribution of average graph distances across 100 permutations for each parameter setting: (b-d) $\mu$ and $\nu$ being uniform, and $\omega$ being incidence, Jaccard index, and intersection size, respectively; (e-g) $\mu$ being normalized node degree, $\nu$ being normalized sum of node degrees, and $\omega$ being incidence, Jaccard index, and intersection size, respectively. }
    \label{fig:centaur-graph-distances}
\end{figure}

\subsection{Measure-Preserving Hypergraph Simplification}
\label{sec:simplify}

In this section, we give examples that capture  changes in hypernetwork distances as we apply multiscale hypergraph simplification based on   the framework of Zhou {\etal}~\cite{ZhouRathorePurvine2022}. 
Visualizing large hypergraphs is a challenging task. 
To reduce visual clutter, we may apply \emph{node collapse} and \emph{edge collapse} to reduce the size of the hypergraph, see Example~\autoref{example:collapses}. 
Zhou \etal~\cite{ZhouRathorePurvine2022} relaxed these notions by allowing nodes to be combined into a \emph{super-node} if they belong to \emph{almost the same} set of hyperedges, and hyperedges to be merged into a \emph{super-hyperedge} if they share \emph{almost the same} set of nodes; these operations are referred to as \emph{node simplification} and \emph{hyperedge simplification}, respectively.  
We follow similar setup from~\cite{ZhouRathorePurvine2022} and focus on  \emph{measure-preserving} hyperedge simplification for our experiments. 

To enable hyperedge simplification, we first convert a hypergraph $H$ into a weighted line graph $\mathsf{L}(H)$, as described above in~\autoref{sec:soft-matching-toy}. 
Each edge in $\mathsf{L}(H)$ is weighted by the multiplicative inverse of the Jaccard index of its end points. 
We compute a minimum spanning tree of $\mathsf{L}(H)$, denoted as $T_{\mathsf{L}(H)}$.  
We then perform a topological simplification of $H$ based on $T_{\mathsf{L}(H)}$. 
In particular, we sort edges in $T_{\mathsf{L}(H)}$ monotonically by increasing weights, $\{e_1, \cdots, e_k\}$.
We then simplify the hypergraph $H$ across multiple scales by merging hyperedges connected by edge $e_i \in T_{\mathsf{L}(H)}$ for each $i$, as $i$ increases from $1$ to $k$. 
To enable node simplification, we convert a hypergraph $H$ into a weighted line graph of its dual (i.e., a \emph{clique expansion}) and apply a similar procedure. 
We use the same visual encodings from Example~\autoref{example:collapses}, where super-nodes are visualized by concentric ring glyphs and super-hyperedge are visualized by pie-chart glyphs.

We utilize the co-optimal transport metric $d_{\mathcal{H},2}$ as a heuristic to determine simplification levels of interest---in prior work, simplification levels were tuned in a more ad hoc manner. The basic idea is to measure the distance at each simplification step to the original hypergraph, in order to find ``elbows" at simplification levels of interest. We now explain details of the co-optimal transport heuristic in the context of a real-world dataset.

\begin{example}[ActiveDNS Hyperedge  Simplification]
\label{example:active-DNS}

Our first example arises from Domain Name System (DNS), a naming database in which internet domain names are located and translated into IP addresses.
In domain aliasing, multiple domains may map to the same IP address. 
A single domain can be assigned to multiple servers for web hosting services, which correspond to multiple IP addresses.
We study a subset of DNS records, from April 26, 2018, as part of the ActiveDNS dataset~\cite{kountouras2016enabling}. 
We explore the relationship between domains and IP addresses via an ActiveDNS hypergraph $H_0$, where nodes are IP addresses and hyperedges are domains. 

\begin{figure}[!ht]
	\centering
	\vspace{-2mm}
	\includegraphics[width=0.8\columnwidth]{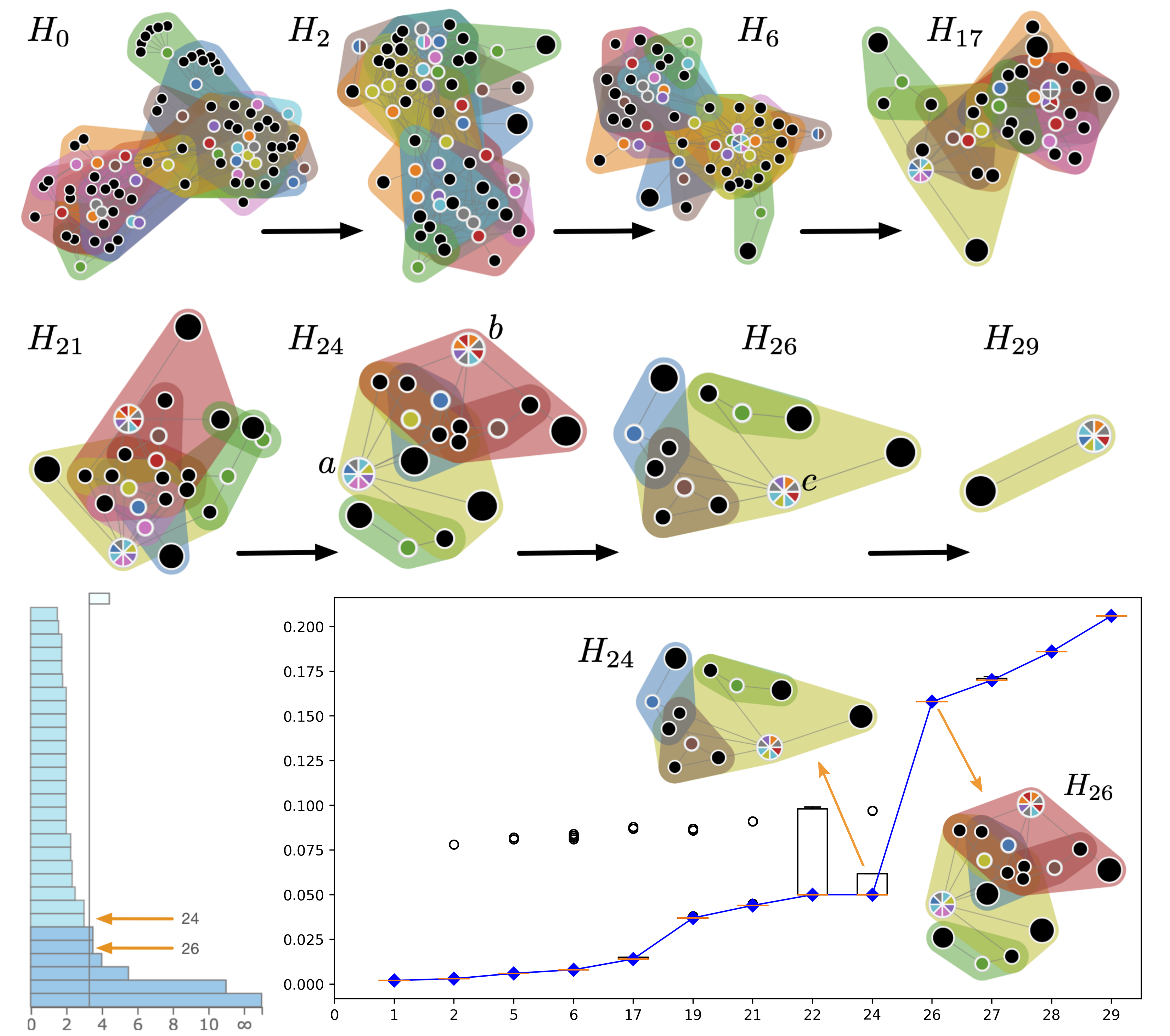}
	\caption{Multiscale simplification of an ActiveDNS hypergraph via hyperedge simplification, together with changes in hypernetwork distances.}
	\label{fig:DNS-plot}
\end{figure}

This dataset has been explored as a hypergraph in~\cite{ZhouRathorePurvine2022}, where we apply a multiscale hyperedge simplification  and obtain a sequence of simplified hypergraphs $H_1, \dots, H_{29}$ using the procedure described above, some of which are shown in \autoref{fig:DNS-plot}. 
We now model each $H_i$ ($0 \leq i \leq 29$) as a hypernetwork---specifically, we set $\mu$ to be normalized node degree, $\nu$ to be normalized sum of node degree and $\omega$ to be the incidence function (see \autoref{sec:soft-matching-toy}). 
We compute the hypernetwork distances $d_{\mathcal{H},2}(H_0, H_i)$ between $H_0$ and $H_i$, shown in~\autoref{fig:DNS-plot}. 
To escape from (undesirable) local minima while optimizing the distance, we use random initialization for the coupling matrices $\pi$ and $\xi$.
For each simplification step, we compute the distance multiple times and plot its distribution using a box plot. We connect the minimum distance values across simplification steps with a line plot.  
In the multiscale hypergraph simplification framework, edges in $T_{\mathsf{L}(H)}$ with the same weight value will be merged together simultaneously; thus we only plot the box plot at unique weight values.

As expected, $d_{\mathcal{H},2}(H_0, H_i)$ is shown in the line plot to increase as we simplify the hypergraph across multiple scales. 
To determine the appropriate simplification threshold, we use the ``elbow method'', meaning that we are interested in the simplification step where there is a sharp increase in distance.  
The elbow of the $d_{\mathcal{H},2}(H_0, H_i)$ curve is at $H_{24}$.  
A closer inspection reveals that $H_{24}$ is well-aligned with a simplified hypergraph manually tuned by an expert from~\cite{ZhouRathorePurvine2022}. 
In particular, $H_{24}$ shows a clear separation into two hyperedges $a$ and $b$ (shown with pie chart glyphs), each of which is the result of merging domains registered by the same organization. 
In particular, hyperedge $a$ arises from merging $14$ domains that are variants of 
“World’s Leading Cruise Lines” registered to Carnival Corporation. 
Hyperedge $b$ is the result of merging, for the most part, $7$ domains registered to Radio Free Europe.  
$a$ and $b$ then merge into hyperedge $c$ in $H_{26}$. 
This example provides initial evidence that the hypergraph distance $d_{\mathcal{H},2}$ may be used to quantify information loss during hypergraph simplification, thus supporting parameter tuning. 
\end{example}

\begin{example}[Coauthor Network Hyperedge Simplification]
\label{example:coauthor}

In the second example, we study a co-authorship  dataset from ArnetMiner \cite{TangZhangYao2008} from authors in the field of  information-retrieval.
We model the dataset as a hypergraph where each hyperedge represents an author, and nodes in a hyperedge are researchers who have coauthored at least one paper with the corresponding author. 
\autoref{fig:coauthor} shows the original hypergraph $H_0$ and the multiscale simplification results. 
We identify step 217 as the elbow point. 
The simplified hypergraph $H_{217}$ shows that David A. Grossman and Ophir Frieder are in two different super-hyperedges, who are the authors of the book “Information Retrieval: Algorithms and Heuristics”, an important introductory textbook in the field. In a subsequent simplified hypergraph ($H_{298}$), the two super-hyperedges merge into one.

\begin{figure}[!ht]
	\centering
	\vspace{-2mm}
	\includegraphics[width=0.8\columnwidth]{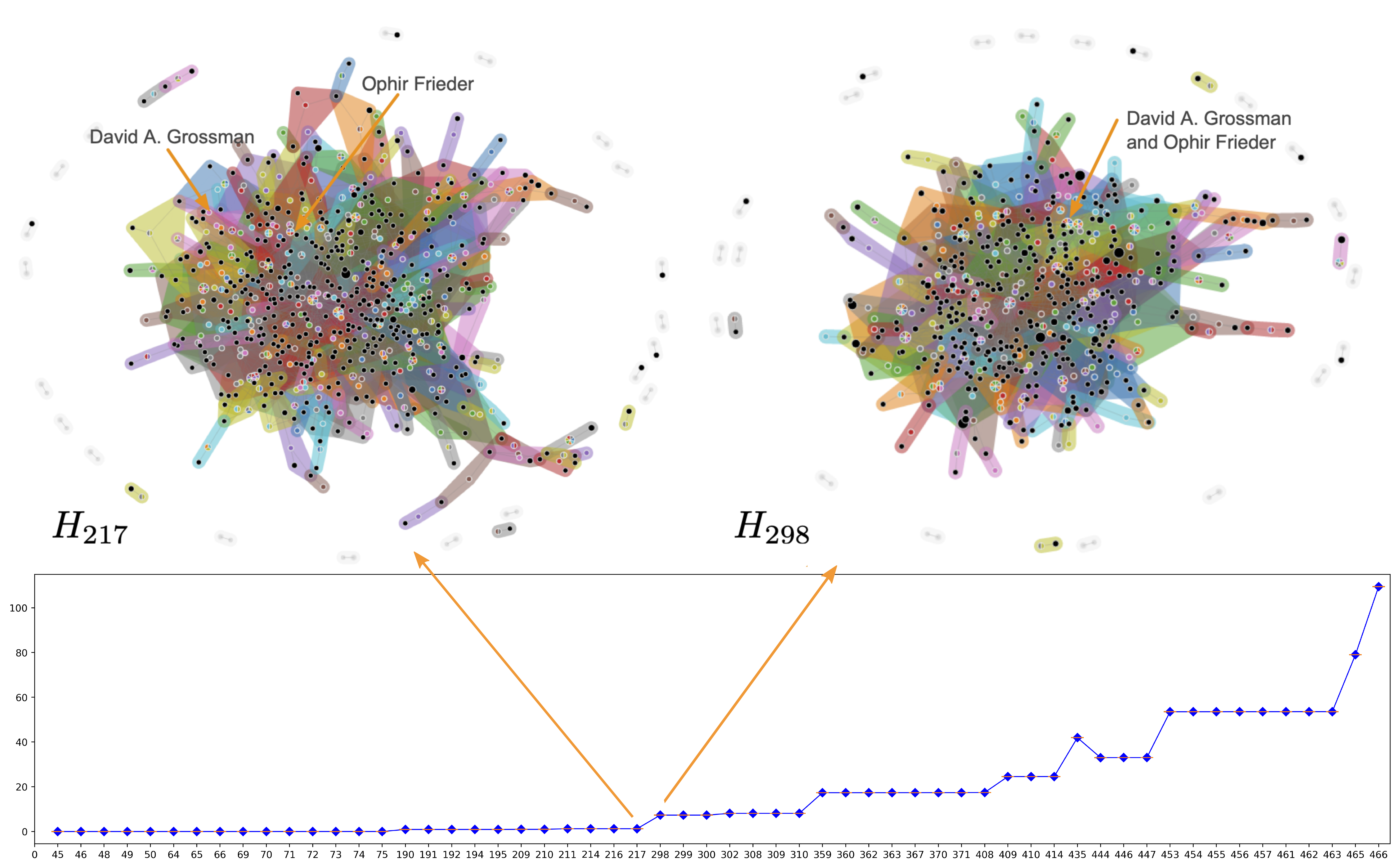}
	\caption{multiscale simplification of a coauthor network hypergraph via hyperedge   simplification, together with changes in hypernetwork distances.}
	\label{fig:coauthor}
\end{figure}
\end{example}

\begin{example}[Les Mis\'{e}rables Node  Simplification]
\label{example:LesMis}
In the final example, we consider Victor Hugo’s novel Les Mis\'{e}rables from the Stanford Graph Base~\cite{Knuth1994}. 
A set of characters are found within each volume, book, chapter, and scene of the story. 
Following the work in~\cite{ZhouRathorePurvine2022}, we form our hypergraph by considering each character to be a node and each (volume, book)-pair to be a hyperedge that contains  characters that appear within. 
We apply multiscale node simplification of this dataset and identify the hypergraph at step $62$ ($H_{62}$) as the elbow point in the $d_{\mathcal{H},2}$ curve, see~\autoref{fig:LesMis-plot}. 
$H_{62}$ is also the one identified in~\cite{ZhouRathorePurvine2022} as the manually tuned result, where three super-vertices are identified: $v_1$ contains many peripheral characters in the first volume who interact with Jean Valjean and Fantine; $v_2$ contains all of the “Friends of the ABC” revolutionary student group, and $v_3$ contains Jean Valjean, Cosette, Javert, Marius, and all of Marius’ family. 

\begin{figure}[!ht]
	\centering
	\vspace{-2mm}
	\includegraphics[width=0.8\columnwidth]{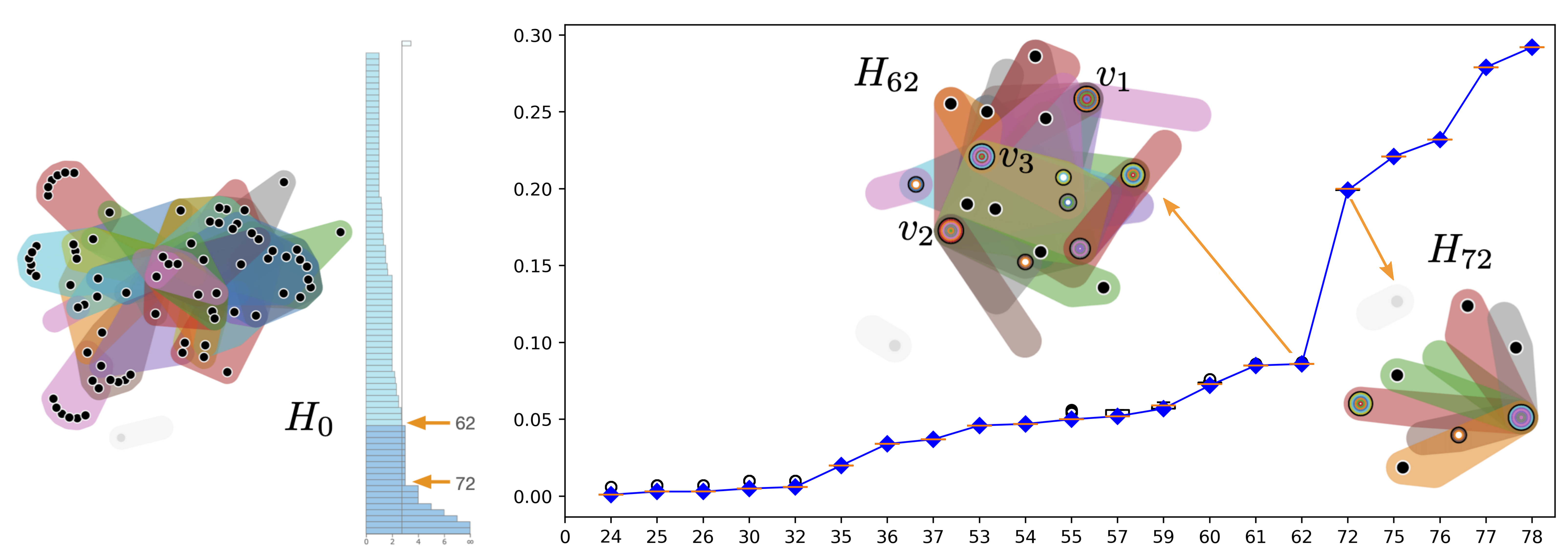}
	\caption{Multiscale simplification of a Les Mis\'{e}rables hypergraph via node  simplification, together with changes in hypernetwork distances.}
	\label{fig:LesMis-plot}
\end{figure}
\end{example}

An analysis of the runtime for these experiments is provided in the Appendix~\ref{sec:runtime_analysis}.

\section{Conclusion}
\label{sec:conclusion}

We view this work to be the first in a landscape of new research directions in hypergraph data analysis. In analogy with the measure network setting, we expect that future work could further develop the geometry of hypergraph space, and utilize this study to obtain methods for carrying out geometric statistics (e.g. Fr\'{e}chet means) on this space. 
From the applied perspective, we expect that future work could focus on applications to machine learning problems as well as scalability and deployment into deep learning pipelines. Integrating GW distance computation as part of the hypergraph visual analytic tool~\cite{ZhouRathorePurvine2022} would be of interest as well.
Finally, the study of morphisms between network and hypernetwork categories initiated in this work suggests new questions about obtaining families of Lipschitz maps and further understanding existing graphification methods from a categorical perspective. For example, a long term project would be to classify all Lipschitz, functorial graphifications satisfying reasonable axioms---Proposition \ref{prop:structural_functor} gives a first step in this direction. It will also be interesting to consider metric and categorical aspects of maps from networks to hypernetworks, or multiscale versions thereof, building on the already substantial literature on the category theory of data partitioning.

\backmatter

\bmhead{Acknowledgments}
This work was partially supported by NSF DMS 2107808, NSF IIS 1910733, NSF IIS 214549 and DOE DE-SC0021015.  




\bibliography{GW-refs.bib}
\nocite{label}

\appendix
\section{Appendix}
\label{sec:detailed-proofs}

\subsection{Proofs from  \autoref{sec:properties}}\label{sec:proofs_properties}

To prove~\autoref{thm:pseudometric}, we begin with a technical result.  

\begin{lemma}
\label{lemma:realized}
The infimum in the definition of $d_{\mathcal{H},p}$ is realized.
\end{lemma}
\begin{proof}
For a Polish space $X$, let $\mathrm{Prob}(X)$ denote the set of Borel probability measures over $X$, endowed with the weak topology.

Let $H = (X,\mu,Y,\nu,\omega), H' = (X',\mu',Y',\nu',\omega') \in \mathcal{H}$. As $X$ and $X'$ are Polish spaces and $\mu$ and $\mu'$ are Borel probability measures, we have that $\mathcal{C}(\mu,\mu')$ is compact in $\mathrm{Prob}(X\times X')$, by \cite[Lemma 10]{ChowdhuryMemoli2019}. It follows that $\mathcal{C}(\mu,\mu') \times \mathcal{C}(\nu,\nu')$ is compact in $\mathrm{Prob}(X\times X') \times \mathrm{Prob}(Y \times Y')$. 

Now we will show that the distortion functional $\mathrm{dis}_p^{\mathcal{H}} = \mathrm{dis}_{H,H',p}^\mathcal{H}$ is continuous on $\mathcal{C}(\mu,\mu') \times \mathcal{C}(\nu,\nu')$, beginning with the $1\leq p < \infty$ case. Since we are working in Polish spaces with finite measures, we have that bounded continuous functions are dense.  So for every $n\in\mathbb{N}$, we can choose continuous bounded functions $\omega_n\in L^p(\mu\otimes\nu)$ and $\omega'_n\in L^p(\mu'\otimes\nu')$ such that the following holds for all $x\in X, y\in Y, x'\in X',$ and $ y'\in Y'$:
\[
\|\omega(x,y)-\omega_n(x,y)\|_{L^p(\mu\otimes\nu)} \leq \frac{1}{n} \quad \mbox{and} \quad \|\omega'(x',y')-\omega'_n(x',y')\|_{L^p(\mu'\otimes\nu')} \leq \frac{1}{n}.
\]
Then for each $n\in \mathbb{N}$ we define a $n$-distortion functional $\mathrm{dis}_p^n: \mathcal{C}(\mu,\mu')\times \mathcal{C}(\nu,\nu')\longrightarrow \mathbb{R}_+$ as $\mathrm{dis}_p^n(\pi,\xi):=  \|\omega_n(x,y) - \omega'_n(x',y') \|_{L^p(\pi\otimes\xi)}$ 

Now let us show that the $n$-distortion functional is continuous.  Let $(\pi,\xi)\in\mathcal{C}(\mu,\mu')\times \mathcal{C}(\nu,\nu')$ and let $(\pi_m)_{m\in\mathbb{N}}$ and $(\xi_m)_{m\in\mathbb{N}}$ be sequences in $\mathcal{C}(\mu,\mu')$ and $\mathcal{C}(\nu,\nu')$ that converge to $\pi$ and $\xi$ with respect to the weak topology. Then
\begin{align}
    &\lim_{m\rightarrow\infty}\text{dis}_p^n(\pi_m,\xi_m) \\
    &= \lim_{m\rightarrow\infty}\left(\int_{Y\times Y'}\int_{X\times X'} \lvert \omega_n(x,y) - \omega'_n(x',y')  \rvert^p \pi_m (dx \times dx')\xi_m(dy \times dy')\right)^{\frac{1}{p}}  \\
    &= \left(\int_{Y\times Y'}\int_{X\times X'} \lvert \omega_n(x,y) - \omega'_n(x',y') \rvert^p\pi (dx \times dx')\xi(dy \times dy')\right)^{\frac{1}{p}} = \text{dis}_p^n(\pi,\xi).
\end{align}
So we have that $\text{dis}_p^n$ is sequentially continuous. As $\mathcal{C}(\mu,\mu')\times \mathcal{C}(\nu,\nu') \subseteq$ Prob$(X\times X'\times Y\times Y')$, it is a metrizable space \cite[Remark 5.1.1]{ambrosio2005gradient}. Therefore, as $\text{dis}_p^n$ is sequentially continuous in a metrizable space, it is also continuous. 

Now let us see that $\text{dis}_p^n$ converges to $\text{dis}^\mathcal{H}_p$ uniformly.  Let $(\pi,\xi)\in\mathcal{C}(\mu,\mu')\times \mathcal{C}(\nu,\nu')$ and observe
\begin{align*}
    & \lvert \text{dis}_p^{\mathcal{H}}(\pi,\xi) - \text{dis}_p^n(\pi,\xi) \rvert  \\
    &= \lvert \|\omega(x,y) - \omega'(x',y') \|_{L^p(\pi\otimes\xi)}-\|\omega_n(x,y) - \omega'_n(x',y') \|_{L^p(\pi\otimes\xi)} \rvert \\
    & \leq \|\omega(x,y) - \omega'(x',y') - \omega_n(x,y) + \omega'_n(x',y') \|_{L^p(\pi\otimes\xi)}  \\
    &\leq \|\omega(x,y)-\omega_n(x,y)\|_{L^p(\mu\otimes\nu)} +\|\omega'(x',y')-\omega'_n(x',y')\|_{L^p(\mu'\otimes\nu')} \leq \frac{2}{n}.
\end{align*}
Since $\mathrm{dis}_p^{\mathcal{H}}$ is the uniform limit of continuous functions, it is also continuous. It follows that the infimum in \eqref{eqn:distance} is realized with a minimal coupling.

For the case in which $p=\infty$, let $(\pi,\xi)$ be couplings in $\mathcal{C}(\mu,\mu')\times \mathcal{C}(\nu,\nu')$. As the domain of $\mathrm{dis}_p$, $\mathcal{C}(\mu,\mu')\times \mathcal{C}(\nu,\nu')$, is a subset of Prob$(X\times X'\times Y\times Y')$, Jensen's Inequality shows that if $1\leq p\leq q\leq \infty$, then we have that $\mathrm{dis}_p^{\mathcal{H}} \leq \mathrm{dis}_q^{\mathcal{H}}$. Also note that $\lim_{p\rightarrow\infty} \mathrm{dis}_p^{\mathcal{H}}(\pi,\xi) =\mathrm{dis}_\infty^{\mathcal{H}}(\pi,\xi)$. Then as for $p\in [1,\infty), \mathrm{dis}_p^{\mathcal{H}}$ is continuous, and $\mathrm{dis}_{\infty}^{\mathcal{H}} = \sup\{\mathrm{dis}_p^{\mathcal{H}} \: \vert \: p  \in [1,\infty) \}$ we have that $\mathrm{dis}_\infty^{\mathcal{H}}$ is lower semicontinuous. As $\mathrm{dis}_p^\mathcal{H}$ is at least lower semicontinuous over a compact domain, we have that the infimum in \eqref{eqn:distance} is realized with a minimal coupling for any $p\in[1,\infty]$.
\end{proof}

\begin{proof}[Proof of Theorem \ref{thm:pseudometric}.]

In the following, let $H=(X,\mu,Y,\nu,\omega)$,  $H'=(X',\mu',Y',\nu',\omega')$ and $H'' = (X'',\mu'',Y'',\nu'',\omega'')$ be measure hypernetworks.

\para{Pseudometric.} Symmetry is obvious. For the triangle inequality, let $(\pi_{12},\xi_{12})$ and $(\pi_{23},\xi_{23})$ be  couplings realizing $d_{\mathcal{H},p}(H,H')$ and $d_{\mathcal{H},p}(H',H'')$, respectively (see Lemma \ref{lemma:realized}). Then, using the gluing lemma (see, e.g., \cite[Lemma 1.4]{sturm2012space}), we construct  a probability measure $\pi$ on $X \times X' \times X''$ such that the $X \times X'$ marginal of $\pi$ is $\pi_{12}$, the $X' \times X''$ marginal is $\pi_{23}$ and the $X \times X''$ marginal is a coupling of $\mu$ and $\mu''$, which we denote $\pi_{13}$. Likewise, we construct a probability measure $\xi$ on $Y \times Y' \times Y''$ with analogous properties---in particular, the $Y \times Y''$ marginal is denoted $\xi_{13}$. Then, using the various marginal conditions, we have
\begin{align*}
    &d_{\mathcal{H},p}(H,H'') \\
    &\leq \mathrm{dis}_p^\mathcal{H}(\pi_{13},\xi_{13}) \\
    &=\|\omega -\omega''\|_{L^p(\pi\otimes\xi)}\\
    &\leq \|\omega -\omega'\|_{L^p(\pi\otimes\xi)}+\|\omega'-\omega''\|_{L^p(\pi\otimes\xi)}\\
    &=\|\omega -\omega'\|_{L^p(\pi_{12}\otimes\xi_{12})}+\|\omega'-\omega''\|_{L^p(\pi_{23}\otimes\xi_{23})} = d_{\mathcal{H},p}(H,H')+d_{\mathcal{H},p}(H',H'').
\end{align*}

\para{Complete.} We prove completeness by adapting the proof of \cite[Theorem 5.8]{sturm2012space}. Let $H_n = (X_n,\mu_n,Y_n,\nu_n,\omega_n)$ be a Cauchy sequence of measure hypernetworks. Then we can choose couplings $\pi_n \in \mathcal{C}(\mu_{n},\mu_{n+1})$ and $\xi_n \in \mathcal{C}(\nu_n,\nu_{n+1})$ such that $\mathrm{dis}_p^{\mathcal{H}}(\pi_n,\xi_n) \to 0$. We then apply the generalized Gluing Lemma \cite[Lemma 1.4]{sturm2012space} to $\pi_1,\ldots,\pi_{N-1}$ (respectively, to $\xi_1,\ldots,\xi_{N-1}$) to construct a measure $\Pi_N$ on $\prod_{n=1}^N X_n$ (respectively, $\Xi_N$ on $\prod_{n=1}^N Y_N$). We now define $\overline{\mu}$ to be the projective limit $\varprojlim \Pi_N$, a well-defined probability measure on $\prod_{n=1}^\infty X_n$ (e.g., \cite{bochner2020harmonic}). Set $\overline{X} := \mathrm{supp}(\overline{\mu})$; then $\overline{X}$ is a closed subspace of a countable product of Polish spaces and is therefore Polish. Similarly, define $\overline{\nu} := \varprojlim \Xi_N$ and $\overline{Y}:= \mathrm{supp}(\overline{\nu})$. Next we define $\Omega_N: \overline{X} \times \overline{Y} \to \Rspace$ by
\[
\Omega_N((x_i)_{i=1}^\infty,(y_i)_{i=1}^\infty) = \omega_N(x_N,y_N).
\]
The sequence $(\Omega_N)_{N=1}^\infty$ is a Cauchy sequence in the Banach space $L^p(\overline{X} \times \overline{Y}, \overline{\mu} \otimes \overline{\nu})$, since
\[
\|\Omega_N - \Omega_{N+1}\|_{L^p(\overline{\mu} \otimes \overline{\nu})} = \mathrm{dis}_p^{\mathcal{H}}(\pi_N,\xi_N) \to 0,
\]
and we define $\overline{\omega}$ to be its limit. Each $\omega_N$ is essentially bounded, so it follows that $\overline{\omega}$ is as well. Setting $\overline{H}:= (\overline{X},\overline{\mu},\overline{Y},\overline{\nu},\overline{\omega})$, we have $\overline{H} \in \mathcal{H}$ and
\[
d_{\mathcal{H},p}(H_n,\overline{H}) \leq \|\omega_n - \overline{\omega}\|_{L^p(\overline{\mu} \otimes \overline{\nu})} \to 0.
\]

\para{Geodesic.} Finally, we show that the metric on $\mathcal{H}$ modulo weak isomorphism is geodesic. For $H \in \mathcal{H}$, let $[H]$ denote its weak isomorphism class. We will construct an explicit geodesic between any two weak isomorphism classes of hypernetworks $[H]$ and $[H']$, following the main idea of the proof of \cite[Theorem 3.1]{sturm2012space}. Let $\pi \in \mathcal{C}(\mu,\mu')$ and $\xi \in \mathcal{C}(\nu,\nu')$ be optimal couplings (which exist, by Lemma \ref{lemma:realized}). Consider the path $\gamma:[0,1] \to \mathcal{H}$ defined by 
\begin{equation}\label{eqn:geodesic}
\gamma(t) = \left(\mathrm{supp}(\pi), \pi, \mathrm{supp}(\xi), \xi, \omega_t\right),
\end{equation}
where
\[
\omega_t((x,x'),(y,y')) = (1-t)\omega(x,y) + t \omega'(x',y').
\]
Observe that $\gamma(0) \in [H]$ and $\gamma(1) \in [H']$; indeed, coordinate projection maps $\phi:X \times X' \to X$ and $\psi:Y \times Y' \to Y$ define a basic weak isomorphism $\gamma(0) \to H$ and a similar construction works for $\gamma(1) \to H'$. To show that $\gamma$ defines a geodesic in $[\mathcal{H}]$, it suffices to show that for all $0 \leq s \leq t \leq 1$,
\begin{equation}\label{eqn:geodesic_ineq}
d_{\mathcal{H},p}(\gamma(s),\gamma(t)) \leq (t-s) d_{\mathcal{H},p}(H,H')
\end{equation}
(see, e.g., \cite[Lemma 1.3]{chowdhury2018explicit}). One can show by an elementary computation that 
\[
\mathrm{dis}_{\gamma(s),\gamma(t),p}^{\mathcal{H}}(\pi,\xi)^p \leq (t-s)^p \mathrm{dis}_{H,H',p}^{\mathcal{H}}(\pi,\xi)^p
\]
for any pair of couplings $(\pi,\xi)$---indeed, this holds because linear interpolations are geodesics in $L^p$ spaces---and \eqref{eqn:geodesic_ineq} follows.

\end{proof}

\begin{remark}\label{rmk:gluing_triangle_inequality}
The proof of the triangle inequality above uses the notion of gluing, which appears in our category theoretic framework in  \autoref{sec:graphification}. The coupling $\pi_{13}$ constructed above is written as $\pi_{13} = \pi_{23} \bullet \pi_{13}$ in the notation from \autoref{sec:graphification}. A byproduct of our proof is the following useful fact: for measure networks $N,N',N''$ and couplings $\pi \in \mathcal{C}(\mu,\mu')$ and $\pi' \in \mathcal{C}(\mu',\mu'')$, we have
\[
\mathrm{dis}_p^\mathcal{N}(\pi' \bullet \pi) \leq \mathrm{dis}_p^\mathcal{N}(\pi') + \mathrm{dis}_p^\mathcal{N}(\pi).
\]
A similar statement holds in the measure hypernetwork setting.
\end{remark}

\subsection{Proofs of Technical Lemmas from Section \ref{sec:clique_expansion}}\label{sec:proofs_clique}

\begin{proof}[Proof of Lemma \ref{lem:weak_iso_reps}]
We will prove the statement for hypernetworks, as the network case is similar. Let $\pi$ and $\xi$ be couplings realizing $d_{\mathcal{H},p}(H,H')$ (such couplings always exist, see Lemma \ref{lemma:realized} in the appendix). Define
\[
\overline{X} = \mathrm{supp}(\pi) \subset X \times X', \quad \overline{\mu} = \pi\mid_{\overline{X}}, \quad \overline{Y} = \mathrm{supp}(\xi) \subset Y \times Y', \quad \overline{\nu} = \xi\mid_{\overline{Y}}
\]
and define hypernetwork functions $\overline{\omega},\overline{\omega}':\overline{X} \times \overline{Y} \to \Rspace$ as follows. Let $p_Z:X \times X' \to Z$ be coordinate projection for $Z \in \{X,X'\}$ and let $p_Z:Y \times Y' \to Z$ be coordinate projection for $Z \in \{Y,Y'\}$. The hypernetwork functions are then given by
\[
\overline{\omega} = \omega \circ (p_X \times p_Y) \qquad \mbox{and} \qquad \overline{\omega}' = \omega' \circ (p_{X'} \times p_{Y'}).
\]
The coordinate projections are measure-preserving, due to the marginal constraints of $\pi$ and $\xi$, and the pair $(p_X,p_Y)$ (respectively, $(p_{X'},p_{Y'})$) therefore defines a basic weak isomorphism (Definition \ref{def:basic-weak-iso}) from $\overline{H}$ to $H$ (respectively, $\overline{H'}$ to $H'$). 

To see \eqref{eqn:weak_iso_reps}, observe first that $d_{\mathcal{H},p}(H,H') = d_{\mathcal{H},p}(\overline{H},\overline{H'})$ follows by weak isomorphism. To see the other equality for $p \in [1,\infty)$, we have
\begin{align*}
    d_{\mathcal{H},p}(H,H')^p &= \int_{Y \times Y'} \int_{X \times X'} \lvert \omega(x,y) - \omega'(x',y')\rvert^p \; \pi(dx \times dx') \xi(dy \times dy') \\
    &= \int_{\overline{Y}} \int_{\overline{X}} \lvert \omega(x,y) - \omega'(x',y')\rvert^p \; \overline{\mu}(dx \times dx') \overline{\nu}(dy \times dy'), \\
\end{align*}
since the value of the integral isn't changed by replacing the full spaces with supports of the measures. Using the change of variables $\overline{x} = (x,x') \in \overline{X}$ and $\overline{y} = (y,y') \in \overline{Y}$, the above is equal to 
\[
\int_{\overline{Y}} \int_{\overline{X}} \lvert \overline{\omega}(\overline{x},\overline{y}) - \overline{\omega}'(\overline{x},\overline{y})\rvert^p \; \overline{\mu}(d\overline{x}) \overline{\nu}(d\overline{y}) = \|\overline{\omega} - \overline{\omega}'\|_{L^p(\overline{\mu} \otimes \overline{\nu})}^p.
\]
For the $p = \infty$ case, a similar change of variables gives
\begin{align*}
    d_{\mathcal{H},\infty}(H,H') &= \inf \{c \mid \lvert\omega(x,y) - \omega'(x',y')\rvert \leq c \mbox{ for $\pi \otimes \xi$-almost every $(x,y,x',y')$}\} \\
    &= \inf \{c \mid \lvert\overline{\omega}(\overline{x},\overline{y}) - \overline{\omega}'(\overline{x},\overline{y})\rvert \leq c \mbox{ for $\overline{\mu} \otimes \overline{\xi}$-almost every $(\overline{x},\overline{y})$}\} \\
    &= \|\overline{\omega} - \overline{\omega}'\|_{L^\infty(\overline{\mu} \times \overline{\nu})}.
\end{align*}
\end{proof}

\begin{proof}[Proof of Lemma \ref{lem:well_defined}]
By Proposition \ref{prop:weak_isomorphism}, it suffices to show that if there is a basic weak isomorphism $(\phi,\psi)$ from $H$ to $H'$, then $\phi$ defines a basic weak isomorphism of $N := \mathsf{Q}_q(H) = (X,\mu,\omega_{\mathsf{Q}_q})$ and $N' := \mathsf{Q}_q(H') = (X',\mu',\omega_{\mathsf{Q}_q}')$. By an argument similar to the one used to prove functoriality in Theorem \ref{thm:p_clique}, the assumption that $(\phi,\psi)$ is a basic weak isomorphism implies that for $\mu \otimes \mu \otimes \nu$-almost every $(x_1,x_2,y) \in X \times X \times Y$, $\min_j \omega(x_j,y) = \min_j \omega'(\phi(x_j),\psi(y))$. In the $q < \infty$ case, we deduce that for $\mu \otimes \mu$ almost every $(x_1,x_2)$, 
\begin{align}
\omega_{\mathsf{Q}_q}'(\phi(x_1),\phi(x_2))^p  &= \int_{Y'} \min_j \omega'(\phi(x_j),y')^p \; \nu' (dy') \nonumber \\
&= \int_{Y} \min_j \omega'(\phi(x_j),\psi(y))^p \; \nu (dy) \label{eqn:change_of_variables}\\
&= \int_{Y} \min_j \omega(x_j,y)^p \; \nu (dy) \nonumber \\
&= \omega_{\mathsf{Q}_q}(x_1,x_2)^p, \nonumber
\end{align}
where \eqref{eqn:change_of_variables} follows by the change-of-variables formula (using that $\psi$ is measure-preserving). The $q=\infty$ case follows by a limiting argument.
\end{proof}

\begin{proof}[Proof of Lemma~\ref{lem:norm_lower_bound}]
This is \cite[Theorem 5.1 (d)]{dghlp-focm}, in the metric measure space setting. The proof goes through in the measure network setting.
\end{proof}

\subsection{Runtime Analysis for Hypergraph Simplification}\label{sec:runtime_analysis}

We ran the experiments using the following computing infrastructures:
\begin{itemize}
\item CPU models: 32 Intel Xeon Silver 4108 CPU \@ 1.80GHz cores (HT)
\item Amount of memory: 132GB of RAM
\item Operating system: OpenSUSE Leap 15.3 (x86\_64)
\end{itemize}
\autoref{table:hypergraph-size-runtime} summarizes the size of each dataset, including the number of hyperedges and the number of vertices contained in each hypergraph, the runtime of computing the barcode, the runtime of performing the simplification (simp.) at the first step, and the maximum runtime among all simplification steps for each dataset.~\autoref{fig:runtime-gw-distances} demonstrates the distribution and the minimum runtime of computing the hypernetwork distances $d_{\mathcal{H},2}(H_0, H_i)$ at each step $i$.
~\autoref{table:runtime-gw-distances} summarizes the minimum runtime at the first step as well as the minimum, maximum, mean, and median values of the minimum runtime of computing the distances across all simplification steps for each dataset.

The largest dataset is the \textit{Coauthor Network} dataset, which contains 506 nodes and 491 hyperedges. 
Therefore, this dataset gives rise to the largest runtime of computing the barcode, simplified hypergraphs, and hypernetwork distances. The maximum runtime among all steps is around 5 seconds for the \textit{Coauthor Network} dataset.

\begin{table}[ht]
\begin{center}
\begin{tabular}{||c||c|c||c||c|c||}
\hline
    Data & \#v & \#e & Barcode & Simp. (first step) & Simp. (maximum)  \\ \hline
    DNS &57 &30 & 0.0014 & 0.1130 & 0.1130 \\ \hline
    CN & 506 & 491 & 0.1317 & 4.0484 & 4.1506 \\ \hline    
    LM & 80 & 45 & 0.0041 & 0.2025 & 0.2438 \\ \hline
\end{tabular}
\end{center}
\vspace{2mm}
\caption{The number of vertices (\#v), the number of hyperedges (\#e), the runtime (in seconds) of computing the barcode, the runtime (in seconds) of performing the simplification at the first step, and the maximum runtime (in seconds) among all simplification steps in datasets \textit{ActiveDNS} (DNS), \textit{Coauthor Network} (CN), and \textit{Les Mis\'{e}rables} (LM).}
\label{table:hypergraph-size-runtime}
\end{table}

\begin{figure}[!ht]
	\centering
	\vspace{-2mm}
	\includegraphics[width=0.99\columnwidth]{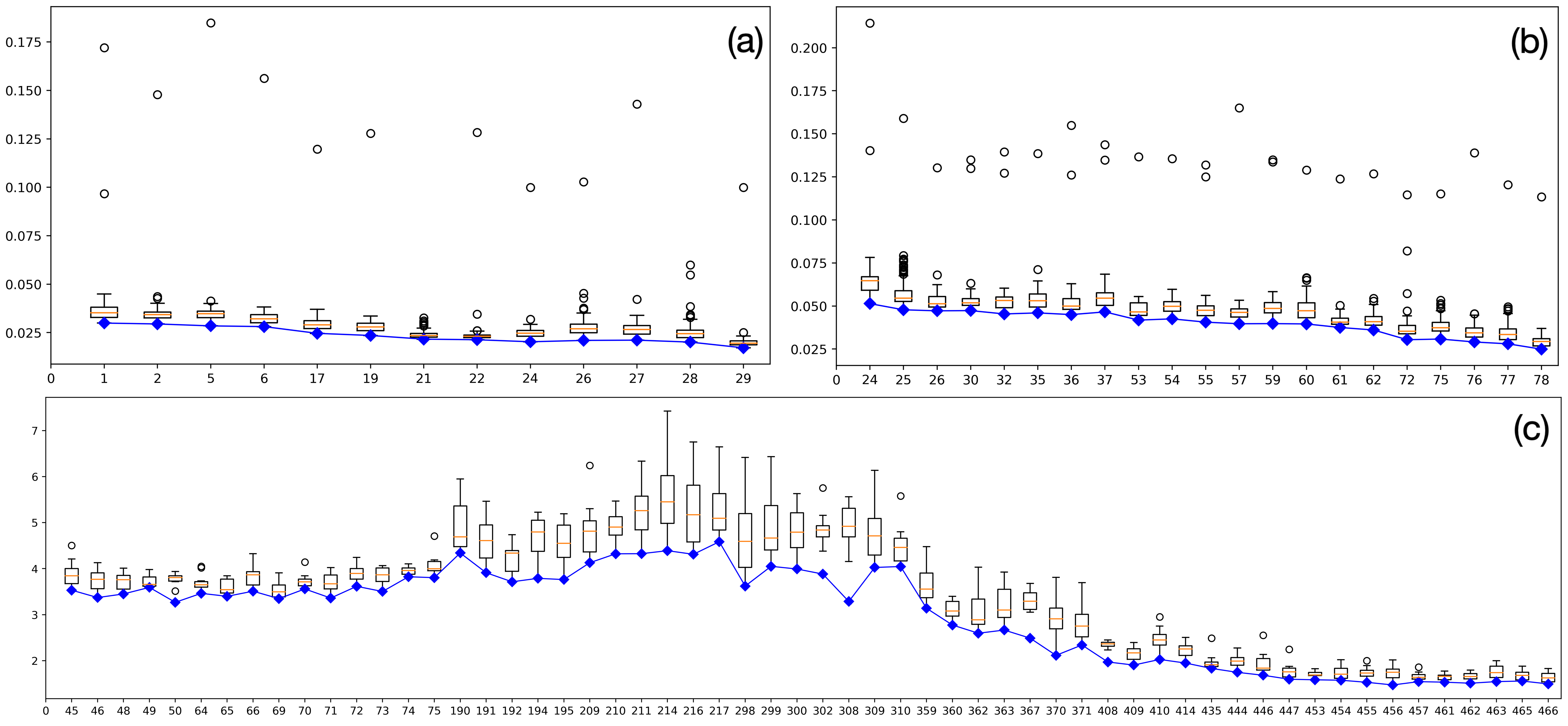}
	\caption{Runtime in computing the GW distances for \textit{activeDNS} (a), \textit{Les Mis\'{e}rables} (b) and \textit{Coauthor Network} (c) datasets. The $x$-axes display the simplification steps. The $y$-axes are time recorded in seconds.}
	\label{fig:runtime-gw-distances}
\end{figure}

\begin{table}[ht]
\begin{center}
\begin{tabular}{|c|c|c|c|c|c|}
\hline
    Data & First step & Minimum & Maximum & Mean & Median  \\ \hline
    DNS & 0.0299 & 0.0171 &  0.0299 & 0.0236 &  0.0216 \\ \hline
    CN & 3.5320 & 1.4722 & 4.5847 & 2.9872 & 3.3658  \\ \hline    
    LM & 0.0513 & 0.0249 & 0.0514 & 0.0399 & 0.0406  \\ \hline
\end{tabular}
\end{center}
\vspace{2mm}
\caption{The minimum runtime (in seconds) at the first step as well as the minimum, maximum, mean, and median values of the minimum runtime (in seconds) of computing the distances across all simplification steps for each dataset.}
\label{table:runtime-gw-distances}
\end{table}


\end{document}